\documentclass[oneside,english]{amsart}
\title{Higher spin representations of maximal compact subalgebras of simply-laced  Kac-Moody-algebras. }

\thanks{This work was supported by DFG grant KO 4323/13 and the Studienstiftung des Deutschen Volkes.}

\author{Robin Lautenbacher}
\address{Robin Lautenbacher\\
	Christian-Albrechts-Universit\"at zu Kiel,
	Mathematisches Seminar\\
	Heinrich-Hecht-Platz 6, 24118 Kiel, Germany}
\email{lautenbacher@math.uni-kiel.de}

\author{Ralf K\"ohl}
\address{Ralf K\"ohl\\
	Christian-Albrechts-Universit\"at zu Kiel,
	Mathematisches Seminar\\
	Heinrich-Hecht-Platz 6, 24118 Kiel, Germany}
\email{koehl@math.uni-kiel.de}

\usepackage{amssymb,url,amsmath,mathtools}
\usepackage{amsaddr}
\usepackage{hyperref}
\usepackage{xcolor}
\usepackage{geometry}
\usepackage{fancyhdr}
\newcommand{\mf}[1]{\mathfrak{#1}}

\theoremstyle{plain}
\newtheorem{theorem}[subsection]{Theorem}
\newtheorem{lemma}[subsection]{Lemma}
\newtheorem{corollary}[subsection]{Corollary}
\newtheorem{proposition}[subsection]{Proposition}
\theoremstyle{definition}

\newtheorem{defn}[subsection]{Definition}
\newtheorem{remark}[subsection]{Remark}

\begin{document}
	\pagestyle{plain}
	\begin{abstract}
		Given the maximal compact subalgebra $\mf{k}(A)$ of a split-real Kac-Moody algebra $\mf{g}(A)$ of type $A$, we study certain finite-dimensional representations of $\mf{k}(A)$, that do not lift to the maximal compact subgroup $K(A)$ of the minimal Kac-Moody group $G(A)$ associated to $\mf{g}(A)$ but only to its spin cover $Spin(A)$ described in \cite{Spin covers}.
		Currently, four \emph{elementary} of these so-called spin representations are known. 
		We study their (ir-)reducibility, semi-simplicity, and lift to the group level. 
		The interaction of these representations with the spin-extended Weyl-group is used to derive a partial parametrization result of the representation matrices by the real roots of $\mf{g}(A)$.
	\end{abstract}

\maketitle 

\section{Introduction}
The involutory subalgebra $\mf{k}(A)$ w.r.t. the Chevalley involution of a split-real Kac-Moody algebra $\mf{g}(A)$ (cp. \cite{Kac:}) is typically referred to as its \emph{maximal compact subalgebra}. 
If $A$ is a generalized Cartan matrix of finite type, $\mf{g}(A)$ is a semi-simple Lie algebra and  $\mf{k}(A)$ indeed is its maximal compact subalgebra.
If $A$ is not of finite type, then both $\mf{g}(A)$ and $\mf{k}(A)$ are infinite-dimensional and $\mf{k}(A)$ admits an invariant, negative definite bilinear form,  but it is not compact in a topological sense, i.e., it is not the Lie algebra of  a compact Lie group (cp. \cite{Gloeckner08}). 
There are at least three reasons to study the representations of $\mf{k}(A)$:
First, its representations and among these in particular the finite-dimensional ones reveal parts of the structure theory of $\mf{k}(A)$. 
Second, some indefinite $\mf{k}(A)$ are conjectured to arise as symmetries in theories of quantum gravity and therefore a well-developed representation theory of $\mf{k}(A)$ is required there.
Third, the representation theory of $\mf{k}(A)$ is expected to be important to the theory of Kac-Moody symmetric spaces, similar to the finite-dimensional case.

An early result concerning the structure theory of $\mf{k}(A)$ is a presentation by generators and relations given in \cite{Berman}. 
The major challenge in the study of $\mf{k}(A)$ is that it is in general not graded by a root system with finite-dimensional root spaces as $\mf{g}(A)$ is, but only admits a filtered structure w.r.t. the roots of $\mf{g}(A)$. 
In particular, $\mf{k}(A)$ is not a simple Kac-Moody algebra or sum thereof if $A$ is not of finite type (cp. \cite{Berman,Slodowy GIM}), and as a consequence the standard tools of representation theory such as highest weight representations and character formulas are not applicable.
It was observed in \cite{KleinschmidtLautenbacher2022} that the $\mf{k}(E_n)$-series can be characterized as the quotient of a generalized intersection matrix algebra (cp.\ \cite{Slodowy GIM}) but the representation theory of these is also rather poorly understood.
It is not obvious that $\mf{k}(A)$  even possesses finite-dimensional representations if $A$ is not of finite type, but of course, these provide interesting  ideals of $\mf{k}(A)$. 
At some point it may be possible to characterize $\mf{k}(A)$ as the co-limit of ideals of finite-dimensional representations. 
For the affine case, this has been shown in \cite{KleinschmidtLautenbacher2022}. 

Concerning the case that $A$ is an indefinite generalized Cartan matrix, there are currently four \emph{elementary} representations known.
The basic one has been first described in the physics literature (\cite{deBuyl05, Extended E8-invariance, Hidden symmetries}) under the name of the $K(E_{10})$-Dirac spinor. 
It has been studied in a mathematical setting and generalized to arbitrary symmetrizable types in \cite{Hainke}, where they were referred to as generalized spin representations.
Both names, Dirac spinor as well as generalized spin representation, stem from the fact that the first and most important example is the representation of $\mf{k}(E_{10})$ which extends the standard spinor representation of its naturally contained $\mf{so}(10)$-subalgebra.

\pagestyle{fancy}
\fancyhead[R]{\textbf{R. Lautenbacher \& R. K\"ohl}}
\fancyhead[L]{\textbf{Higher spin representations }}
\fancyfoot[C]{\thepage}

The so-called higher spin representations $\mathcal{S}_{\frac{3}{2}}$, $\mathcal{S}_{\frac{5}{2}}$, and $\mathcal{S}_{\frac{7}{2}}$ of $\mf{k}(A)$ with the exception of $\mathcal{S}_{\frac{3}{2}}$ were  introduced first in \cite{On higher spin}, again in a physics setting.
In \cite{Ext gen spin reps}, the authors of this paper derived a coordinate-free formulation of  $\mathcal{S}_{\frac{3}{2}}$ and $\mathcal{S}_{\frac{5}{2}}$  in the setting of simply-laced $A$ but a similar formulation for $\mathcal{S}_{\frac{7}{2}}$ remained elusive back then.
However, it became apparent that the Weyl group plays a central role in this construction, which builds these representations  on top of the generalized spin representation $\mathcal{S}_{\frac{1}{2}}$ from \cite{Hidden symmetries} and \cite{Extended E8-invariance}.
We provide a detailed construction and description of all these representations in section \ref{sec: higher spin representations} (cp. \ref{thm:Properties of spin rep's image}, \ref{thm:S32 and S52} and \ref{thm:7/2 spin rep}). 
The novelties in this section are that we spell out the importance of the Weyl group in this construction as clearly as possible, that we provide everything in a mathematical and coordinate-free formulation including $\mathcal{S}_{\frac{7}{2}}$ and that we provide an abstract foundation for the generalized $\Gamma$-matrices used in \cite{On higher spin} that does not rely on the use of Clifford algebras (subsection \ref{sec: Gamma amtrices}).
The subsection on generalized $\Gamma$-matrices won't be needed until the proof of prop. \ref{prop: no kernel inclusion} and section \ref{sec: Spin(A)}.

We analyze these representations in section \ref{sec:Irreducibility}, where we put the Weyl group-based formulation to work.
We show that $\mathcal{S}_{\frac{3}{2}}$ is irreducible if $A$ is indecomposable, regular and simply-laced and that the image of $\mf{k}(A)$ under this representation is a semi-simple Lie-algebra (prop. \ref{prop:S32 irrreducible} and cor. \ref{cor: image is semi-simple}).
Furthermore, we show that $\mathcal{S}_{\frac{5}{2}}$ is completely reducible, always contains an invariant submodule isomorphic to $\mathcal{S}_{\frac{1}{2}}$ and that its other invariant factors are controlled by the representation theory of $W(A)$, namely how the symmetric product $\text{Sym}^2(\mf{h}^\ast)$ of the dual Cartan subalgebra decomposes as a $W(A)$-module (prop. \ref{prop:trace-free S52} and lem. \ref{lem: Sym2V completely reducible}).
As for $\mathcal{S}_{\frac{3}{2}}$, we show that the image of $\mf{k}(A)$ under this representation is semi-simple (cor. \ref{cor: image of S52 is semi-simple}).
At the end of sec. \ref{sec:Irreducibility} we show that the kernels of some of these representations are not contained in each other (prop. \ref{prop: no kernel inclusion}) and that their tensor product has a smaller kernel than the individual representations (prop. \ref{prop:product kernels}).

In section \ref{sec: Spin(A)}, we study the spin representations' lift to the group level.
We confirm the common belief (cp. \cite{Hidden symmetries, Extended E8-invariance, Higher spin realizations, On higher spin, Ext gen spin reps}) that these representations are spinorial in the sense that they do not lift to the involutory subgroup $K(A)=G(A)^{\theta}$, where $G(A)$ is the minimal split-real Kac-Moody group of type $A$ and $\theta$ is its Chevalley involution, but instead lift only to its spin cover $Spin(A)$ introduced in \cite{Spin covers} . 
This belief is plausible if one compares the one-parameter subgroups induced by $\exp (\phi \sigma(X_i))$ and  $\exp (\phi \text{ad}(X_i))$, where $\sigma$ is a spin representation and $X_i$ is a so-called Berman generator of $\mf{k}(A)$. 
We show in prop. \ref{prop:Lift criterion} that it indeed suffices to look at these one-parameter subgroups.
Afterwards, we demonstrate that the spin representations' lift realizes an action of the spin-extended Weyl group from \cite{Spin covers} on the modules $\mathcal{S}_{\frac{n}{2}}$. 
We use that the action of $Spin(A)$ on $\mf{k}(A)$ factors through the adjoint action of $K(A)$ on $\mf{k}(A)$, to derive the representation matrices up to sign of all elements in the $W^{ext}(A)$-orbit of the Berman generators, where $W^{ext}(A)$ is the extended Weyl group. 
This amounts to providing the representation matrices up to sign of all $x\in\mf{k}_\alpha=\mf{k}(A)\cap \mf{g}_\alpha\oplus\mf{g}_{-\alpha}$ for $\alpha$ a positive real root (props. \ref{prop:Conjugation lemma for S12} and \ref{prop:Conjjugation lemma for Sn2}).

This article contains results that were obtained in the first author's PhD-thesis \cite{Lautenbacher22} but did not appear in any journal so far. Most of these results' proofs are as in \cite{Lautenbacher22} but occasionally we decided to provide a more streamlined version that skips lengthy but elementary computations. 
We also correct a sign error in \cite[lem. 5.6]{Lautenbacher22} and note that as a consequence \cite[lem. 6.5]{Lautenbacher22} and the proof of \cite[prop. 6.7]{Lautenbacher22} are incorrect and it is therefore unknown if \cite[prop. 6.7]{Lautenbacher22} is true. 
We provide the appropriate references to \cite{Lautenbacher22} for propositions and theorems but not for lemmas.

\section{Preliminaries}
\subsection{Kac-Moody-algebras}
We provide a constructive definition of symmetrizable Kac-Moody-algebras that is equivalent to the construction in \cite[chp. 1]{Kac:} due to the Gabber-Kac-theorem (cp. \cite[thm. 9.11]{Kac:}, originally \cite{Gabber-Kac-theorem}).

\begin{defn}
	\label{def:(Generalized-Cartan-matrix)}
	A matrix $A=\left(a_{ij}\right)_{i,j=1}^{n}\in\mathbb{Z}^{n\times n}$
	is called a \emph{generalized Cartan matrix} (GCM) if for all $i\neq j\in\left\{ 1,\dots,n\right\} $
	\begin{align*}
		a_{ii}  =  2\ ,\quad  a_{ij}  \leq  0\ ,\quad	a_{ij}  =  0\ \Leftrightarrow\ a_{ji}=0.
	\end{align*}
	$A$ is called \emph{symmetrizable} if there exists a
	regular, diagonal matrix $D$ and a symmetric matrix $B$ such that
	$A=DB$. The pair of matrices $D$ and $B$ is called a \emph{symmetrization}
	of $A$. The GCM $A$ is called \emph{simply-laced} if $a_{ij}\in\left\{ 0,-1\right\} $
	for all $i\neq j$. 
	We denote the set of unordered pairs $\{i,j\}$ s.t. $a_{ij}\neq 0$ by $\mathcal{E}(A)$, the edges of the generalized Dynkin diagram associated to $A$.
\end{defn}

\begin{defn}
	\label{def:Realization}
	Let $A\in\mathbb{Z}^{n\times n}$ be a GCM of rank $l\leq n$ and let $\mathfrak{h}$
	be a $\mathbb{K}$-vector space of dimension $2n-l$. 
	A $\mathbb{K}$-\emph{realization} $\left(\mathfrak{h},\Pi,\Pi^{\vee}\right)$ of $A$ consists of linearly independent subsets
	$\Pi=\left\{ \alpha_{1},\dots,\alpha_{n}\right\} \subset\mathfrak{h}^{\ast}$
	and $\Pi^{\vee}=\left\{ \alpha_{1}^{\vee},\dots,\alpha_{n}^{\vee}\right\}\subset\mathfrak{h} $ such that 
	\begin{align*}
	\alpha_{j}\left(\alpha_{i}^{\vee}\right)=a_{ij}\ \forall\,i,j\in\left\{ 1,\dots,n\right\} .
	\end{align*}
	One calls $\Pi$ the \emph{simple roots}, $\Pi^{\vee}$ the \emph{simple coroots}, $Q(A)\coloneqq\text{span}_{\mathbb{Z}}\Pi$	the \emph{root lattice} and $Q^{\vee}(A)\coloneqq\text{span}_{\mathbb{Z}}\Pi^{\vee}$ the \emph{coroot lattice}. 
\end{defn}

\begin{defn}
	Let $A\in\mathbb{Z}^{n\times n}$ be a symmetrizable GCM with $\mathbb{K}$-realization
	$\left(\mathfrak{h},\Pi,\Pi^{\vee}\right)$. 
	The	\emph{split Kac-Moody algebra over $\mathbb{K}$ of type $A$} is defined as the Lie algebra on generators $\mathfrak{h}\cup\left\{ e_{1},\dots,e_{n},f_{1},\dots,f_{n}\right\} $
	subject to the relations 
	\begin{align*}
		\left[h,h'\right]&=0,\ &\left[e_{i},f_{j}\right]&=\delta_{ij}\alpha_{i}^{\vee},\\
		\left[h,e_{i}\right]&=\alpha_{i}\left(h\right)e_{i},\ &\left[h,f_{i}\right]&=-\alpha_{i}\left(h\right)f_{i},\\
		0&=\text{ad}\left(e_{i}\right)^{1-a_{ij}}\left(e_{j}\right),\  &0&=\text{ad}\left(f_{i}\right)^{1-a_{ij}}\left(f_{j}\right)\:.		
	\end{align*}
	for all $h,h'\in\mathfrak{h}$ and for all $i,j\in\left\{ 1,\dots,n\right\} $.
\end{defn}

\begin{defn}
	Let $\mathfrak{g}=\mathfrak{g}\left(A\right)\left(\mathbb{K}\right)$ be a split Kac-Moody-algebra and set $\mathfrak{g}_{\alpha}\coloneqq\left\{x\in\mathfrak{g}\,\vert\,\left[h,x\right]=\alpha\left(h\right)x\ \forall\,h\in\mathfrak{h}\right\} $ for $\alpha\in\mathfrak{h}$.
	One calls $\alpha\neq0$ such that $\mathfrak{g}_{\alpha}\neq\left\{ 0\right\} $  a \emph{root} and $\mathfrak{g}_{\alpha}$ a \emph{root space}. 
	
	Denote the set of roots	and its decomposition into positive and negative roots by  $\Delta=\Delta_{-}\cup\Delta_{+}\subset Q$,
	where $\Delta_{+}\coloneqq\left\{ \alpha\in\Delta\,\vert\,\alpha>0\right\} $ and $\Delta_{-}=-\Delta_{+}$, then $\mf{g}$ admits the following \emph{root space decomposition} (as a vector space):
	\[
	\mathfrak{g}= \bigoplus_{\alpha\in\Delta_-}\mathfrak{g}_{\alpha} \oplus \mf{h}\oplus \bigoplus_{\alpha\in\Delta_+}\mathfrak{g}_{\alpha}\:.
	\]
\end{defn}

\begin{proposition}[This is thm. 1.2 and sec. 1.3 of \cite{Kac:} ] \label{prop:Chevalley-involution} 
	On $\mathfrak{g}\left(A\right)\left(\mathbb{K}\right)$ there exists an involutive automorphism $\omega$ called the \emph{Chevalley involution} that is determined by 
	\[
	\omega\left(e_{i}\right)=-f_{i},\ \omega\left(f_{i}\right)=-e_{i},\ \omega\left(h\right)=-h\ \forall\,h\in\mathfrak{h}.
	\]
	It  satisfies $\omega\left(\mathfrak{g}_{\alpha}\right)=\mf{g}_{-\alpha}$
	for all $\alpha\in\Delta$.
\end{proposition}

\begin{proposition}[Cp. \cite{Kac:}, thm. 2.2]
	\label{prop:Invariant-bilinear-form}
	Let $A\in\mathbb{Z}^{n\times n}$ be a symmetrizable GCM with symmetrization $A=DB$, where $D=\text{diag}\left(\varepsilon_{1},\dots,\varepsilon_{n}\right)$ is chosen s.t. $\varepsilon_{i}>0$ for all $i=1,\dots,n$ and let $\left(\mathfrak{h},\Pi,\Pi^{\vee}\right)$ be a $\mathbb{K}$-realization of $A$.	
	Set $\mathfrak{h}'\coloneqq\text{span}_{\mathbb{K}}\left\{ \alpha_{1}^{\vee},\dots,\alpha_{n}^{\vee}\right\} $ and denote by $\mathfrak{h}''$  a fixed complementary subspace of $\mathfrak{h}'\subset\mathfrak{h}$.	
	On $\mathfrak{g}=\mathfrak{g}\left(A\right)\left(\mathbb{K}\right)$, there exists a $\mathbb{K}$-bilinear form $\left(\cdot\vert\cdot\right)$  s.t.
	\begin{align*}
		\left(h\vert\alpha_{i}^{\vee}\right)&=\alpha_{i}\left(h\right)\varepsilon_{i}\ \forall\,h\in\mathfrak{h},\quad\left(h_{1}\vert h_{2}\right)=0\ \forall\,h_{1},h_{2}\in\mathfrak{h}'',\\
		\left(\left[x,y\right]\vert\,z\right) & =  \left(x\,\vert\left[y,z\right]\right)\ \forall\,x,y,z\in\mathfrak{g},\\
		\left(\mathfrak{g}_{\alpha}\vert\mathfrak{g}_{\beta}\right) & =  0\ \forall\,\alpha,\beta\in\Delta\ s.t.\ \alpha\neq-\beta,\\
		\left(\cdot\vert\cdot\right)_{\vert_{\mathfrak{g}_{\alpha}\oplus\mathfrak{g}_{-\alpha}}} &   \text{ is non-degenerate }\forall\,\alpha\in\Delta\\
		\left[x,y\right] & =  \left(x\vert y\right)\nu^{-1}\left(\alpha\right)\ \forall\,x\in\mathfrak{g}_{\alpha},y\in\mathfrak{g}_{-\alpha},\ \alpha\in\Delta,
	\end{align*}
	where $\nu:\mf{h}\rightarrow\mf{h}^\ast$ is the isomorphism induced by the bilinear form, i.e., one has  $\nu(h_1)(h_2)=(h_1\vert h_2)$ for all $h_1,h_2\in\mf{h}$.
	This form is referred to as the \emph{standard invariant bilinear form}.
	If $A$ is indecomposable, $(\cdot\vert\cdot)$ is unique up to scalar multiples.
\end{proposition}

The definition of the Weyl group $W(A)$ for a GCM $A$ in the Kac-Moody context is the straightforward generalization of the definition in the classical setting of crystallographic root systems:
\begin{defn}
\label{def: Weyl group}
Given a split Kac-Moody algebra $\mathfrak{g}\left(A\right)\left(\mathbb{K}\right)$ with GCM $A\in\mathbb{Z}^{n\times n}$, define the \emph{fundamental reflections} $s_{i}\in GL\left(\mathfrak{h}^{\ast}\right)$ for $i=1,\dots,n$ via 
\begin{align*}
s_{i}\left(\lambda\right)\coloneqq\lambda-\lambda\left(\alpha_{i}^{\vee}\right)\alpha_{i}\ \forall\,\lambda\in\mathfrak{h^{\ast}}
\end{align*}
and the \emph{Weyl group} $W(A)$ of $\mathfrak{g}\left(A\right)\left(\mathbb{K}\right)$
as $W\left(A\right)\coloneqq\left\langle s_{1},\dots,s_{n}\right\rangle \subset GL\left(\mathfrak{h}^{\ast}\right)$.
\end{defn}
It is a standard fact (cp. \cite[prop. 3.13]{Kac:}) that $W(A)$ admits a presentation as a Coxeter group. 
Furthermore, the action of $W(A)$ preserves the roots. One calls a root $\alpha\in\Delta$ \textit{real} if there exists $w\in W(A)$ s.t. $\alpha=w(\alpha_i)$ for some $\alpha_i\in\Pi$ and \textit{imaginary}, if this is not the case. 
The sets of real and imaginary roots are denoted by $\Delta^{re}$ and $\Delta^{im}=\Delta\setminus\Delta^{re}$.

\subsection{The maximal compact subalgebra $\mf{k}(A)$}

\begin{defn}
	\label{def:Berman subalgebra}
Given a split-real, symmetrizable Kac-Moody-algebra $\mf{g}\coloneqq\mathfrak{g}\left(A\right)\left(\mathbb{R}\right)$, its \emph{maximal compact subalgebra} $\mf{k}(A)$ is defined to be  the Chevalley-involution's fixed-point subalgebra, i.e., one has $\mf{k}(A)\coloneqq\left\{ x\in\mathfrak{g}\,\vert\,\omega(x)=x\right\} $. 
\end{defn}
If $\mf{g}(A)$ is a finite-dimensional split-real Lie algebra, then $\mf{k}(A)$ is maximally compact in the sense that the Killing form restricted to $\mf{k}(A)$ is negative definite and that $\mf{k}(A)$ is maximal with regard to this property. 
Correspondingly, if $G$ is a real Lie group with Lie algebra $\mf{g}(A)$, then $\mf{k}(A)$ is the Lie algebra of the maximal compact subgroup $K$ of $G$.
Note that the notion \emph{maximal compact subalgebra} must not be confused with the maximal compact form, which is defined for complex Lie algebras and involves an additional twist of the Chevalley involution by complex conjugation.
For $\mf{g}=\mf{sl}(n,\mathbb{R})$ one has $\mf{k}=\mf{so}(n,\mathbb{R})$ but the maximal compact form of $\mf{sl}(n,\mathbb{C})$ is $\mf{su}(n,\mathbb{R})$.
One also has to note that for GCMs of non-finite type, the term \emph{maximal compact subalgebra} is slightly misleading, as it is in fact not the Lie algebra of a compact Lie group; its name derives from the analogy to the finite-dimensional case and the fact that the invariant bilinear form is negative-definite on $\mf{k}(A)$ (cp. \cite[thm. 11.7]{Kac:}).
\begin{theorem}[Cp. thm. 1.8 of \cite{Hainke}, originally due to \cite{Berman}]
	\label{thm:presentation of k}
	
	Given the split-real symmetrizable Kac-Moody algebra $\mathfrak{g}\left(A\right)\left(\mathbb{R}\right)$, the maximal compact subalgebra $\mathfrak{k}\left(A\right)\left(\mathbb{R}\right)$ admits a presentation by generators $X_{1},\dots,X_{n}$ and relations
	\[
	P_{-a_{ij}}\left(\text{ad}\,X_{i}\right)\left(X_{j}\right)=0\ \forall\,i\neq j\in\left\{ 1,\dots,n\right\} ,
	\]
	where 
	\[
	P_{m}\left(t\right)\coloneqq\begin{cases}
		\prod_{k=0}^{\frac{m-1}{2}}\left(t^{2}+\left(m-2k\right)^{2}\right) & \text{if }m\text{ is odd,}\\
		t\cdot\prod_{k=0}^{\frac{m}{2}-1}\left(t^{2}+\left(m-2k\right)^{2}\right) & \text{if }m\text{ is even,}
	\end{cases}
	\]
	and $m=0$ spells out as $\left[X_i,X_j\right]=0$ whenever $a_{ij}=0$.
Explicitly, one has $X_{i}=e_{i}-f_{i}$ for $i=1,\dots,n$ in terms of the Chevalley generators $e_{1},\dots,e_{n},f_{1},\dots,f_{n}$ of $\mf{g}(A)(\mathbb{R})$.
The $X_i$ are called the \emph{Berman generators} of $\mathfrak{k}\left(A\right)\left(\mathbb{R}\right)$.
\end{theorem}

For $A$ simply-laced the above relations take the following simple form:
\begin{corollary}[Cp. \cite{Berman}, thm. 1.31 or \cite{Hainke}, thm. 1.8]
	\label{cor:Berman presentation for simply-laced A} 
	Let $A$ be simply-laced and denote by $\mathcal{E}\left(A\right)$
	the edges of the generalized Dynkin diagram, i.e., the unordered pairs $\{i,j\}$ s.t. $a_{ij}=-1$. 
	Then $\mathfrak{k}\left(A\right)\left(\mathbb{R}\right)$
	admits a presentation by generators $X_{1}=e_1-f_1,\dots,X_{n}=e_n-f_n$ and relations
	\[
	\left[X_{i},\left[X_{i},X_{j}\right]\right]=-X_{j}\ \forall\,\{i,j\}\in\mathcal{E}\left(A\right),\qquad \left[X_{i},X_{j}\right]=0\ \forall\,\{i,j\}\notin\mathcal{E}\left(A\right).
	\]
\end{corollary}
For later reference, we collect the following elementary result:
\begin{lemma}
	\label{lem:Filtered structure on k} The maximal compact subalgebra
	$\mathfrak{k}\left(A\right)\left(\mathbb{R}\right)$ as well as its
	complexification $\mathfrak{k}\left(A\right)\left(\mathbb{C}\right)\coloneqq\mathfrak{k}\left(A\right)\left(\mathbb{R}\right)\otimes_{\mathbb{R}}\mathbb{C}$
	is filtered by $\Delta_{+}$, i.e., one has
	\begin{equation}
		\mathfrak{k}\left(A\right)=\bigoplus_{\alpha\in\Delta_{+}}\mathfrak{k}_{\alpha}\label{eq:filtered decomposition}
	\end{equation}
	as vector spaces, where $\mathfrak{k}_{\alpha}\coloneqq\left(\mathfrak{g}_{\alpha}\oplus\mathfrak{g}_{-\alpha}\right)\cap\mathfrak{k}$.
	For $x_{\alpha}\in\mathfrak{k}_{\alpha}$, $x_{\beta}\in\mathfrak{k}_{\beta}$
	s.t. $\alpha-\beta\in Q_+$ one has 
	\begin{equation}
		\left[x_{\alpha},x_{\beta}\right]\in\mathfrak{k}_{\alpha+\beta}\oplus\mathfrak{k}_{\alpha-\beta}\:.\label{eq:filtered commutator relation}
	\end{equation}
\end{lemma}
\begin{proof}
	The filtered structure of $\mathfrak{k}(A)$ is a direct consequence of the graded structure of $\mathfrak{g}\left(A\right)\left(\mathbb{K}\right)$ (cp. \cite{Berman}  and \cite{Hainke} for a more detailed exposition).
\end{proof}

\section{Higher spin representations} \label{sec: higher spin representations}
Historically, the name \emph{spin representation} is reserved for a representation of $\mf{so}(p,q)$ that doesn't lift to the real Lie  group $SO(p,q)$ but only to its double cover, the so-called spin group $Spin(p,q)$.
The most elementary spin representations, corresponding to the missing fundamental weights not observed in representations that lift to the group level, are typically constructed by means of Clifford algebras.
In signature $(n,0)$, the representation matrices of suitably normalized generators such as the Berman generators possess only the eigenvalues $\pm \frac{i}{2}$ which is why they square to $-\frac{1}{4}Id$. 
As it turns out, one can generalize the Clifford algebra construction to any maximal compact subalgebra, not just $\mf{so}(n+1)$ which is $\mf{k}(A_n)$, but in the classical cases this provides nothing new. 
For GCMs of affine and indefinite type however, these representations are quite fascinating, because they show among other things that $\mf{k}(A)$ is not a Lie algebra of Kac-Moody type, since these do not admit finite dimensional representations if they are themselves infinite-dimensional (and simple).

\begin{defn}
\label{Def:gen spin rep for simply laced case}(Cp. \cite[def. 3.6]{Hainke})
Let $A$ be a simply-laced GCM and let $\rho\,:\ \mathfrak{k}(A)\rightarrow\text{End}\left(S\right)$ be a representation of $\mf{k}(A)$ for $S$ a finite-dimensional real vector space.
One calls $\rho$ a \emph{generalized spin representation} if 
\[
\rho\left(X_{i}\right)^{2}=-\frac{1}{4}Id\ \forall\,i=1,\dots,n\,,
\]
where $X_{1},\dots,X_{n}$ denote the Berman generators of $\mathfrak{k}(A)$. 
\end{defn}

\begin{proposition}[Cp. \cite{Hainke}, 3.7]
	\label{prop:properties of gen spin reps in simply laced case} 
	Let $A$ be a simply-laced GCM and let $\rho\,:\ \mathfrak{k}(A)\rightarrow\text{End }\left(S\right)$
	be a generalized spin representation and denote by $\left\{ A,B\right\} :=AB+BA$ the anti-commutator of matrices.
	Then one has for all $1\leq i\neq j\leq n$:
	\begin{eqnarray*}
		\left[\rho\left(X_{i}\right)\,,\,\rho\left(X_{j}\right)\right]=0 & \text{if }a_{ij}=0\:,\\
		\left\{ \rho\left(X_{i}\right)\,,\,\rho\left(X_{j}\right)\right\} =0 & \text{if }a_{ij}=-1 & .
	\end{eqnarray*}
	Conversely, the extension of the map $X_{i}\mapsto M_{i}$ defines a generalized
	spin representation, if the linear maps $M_{1},\dots,M_{n}\in\text{End}\,(S)$ satisfy
	\begin{eqnarray*}
		M_{i}^{2} & = & -\frac{1}{4}id_{s}\:,\\
		\left[M_{i},M_{j}\right] & = & 0\ \text{if }a_{ij}=0\:,\\
		\left\{ M_{i},M_{j}\right\}  & = & 0\ \text{if }a_{ij}=-1\ .
	\end{eqnarray*}
\end{proposition}

\begin{theorem}[This is the specialization of \cite{Hainke}, thms. 3.9 and 3.14 to $A$ simply-laced.]
	\label{thm:Properties of spin rep's image} 
	Let $A$ be a simply-laced GCM, then a generalized spin representation exists. 
	Its image is compact, hence reductive. 
	If the Dynkin diagram of $A$ has no isolated nodes, then the image is furthermore semi-simple.
\end{theorem}
Later on, we will refer to the representations described in the above theorem as $\frac{1}{2}$-spin representations, because the representation matrices of the Berman generators only possess the eigenvalues $\pm\frac{i}{2}$.
For a given $\mf{k}(A)$ we fix a representation denoted by $\rho$ and whose carrier space we will denote by $S$ in general and $\mathcal{S}_{\frac{1}{2}}$ if it is an irreducible $\mf{k}(A)$-module.

\subsection{Generalized spin representations and generalized $\Gamma$-matrices}\label{sec: Gamma amtrices}
Most explicit constructions of generalized spin representations use Clifford algebras. 
In \cite{On higher spin} for instance, the authors associate an element of a Clifford algebra to any root of the root system such that the simple roots map to (multiples of) the $\rho(X_i)$ and thus define a representation. 
They call these elements $\Gamma$-matrices in generalization of the Dirac-matrices which are typically denoted by $\gamma_\mu$.
We will construct and study these generalized $\Gamma$-matrices abstractly without reference to an underlying Clifford algebra but note that this realization appears to be achievable at least in the simply laced situation by using explicit basis expressions for the roots as in \cite{On higher spin}.
Later, this framework will be used to derive the representation matrices for any $x\in\mf{k}_{\alpha}$ with $\alpha\in\Delta^{re}$ up to sign. 

In order to properly introduce generalized $\Gamma$-matrices we need to associate normalized $2$-cocycles to a root lattice:
\begin{defn}
	\label{Def: associated two cocycle}
	Given the root lattice $Q$ of a Kac-Moody-algebra $\mf{g}(A)(\mathbb{K})$ with symmetrizable GCM $A$, we call $\varepsilon\,:\,Q\times Q\rightarrow C_{2}$
	 an \emph{associated}, \emph{normalized 2-cocycle} if
	 \begin{subequations} \label{eq:cocycle props}
	 	\begin{align}
	 		\varepsilon(\alpha,\beta)\varepsilon(\beta,\alpha) &=(-1)^{\left(\alpha\vert\beta\right)}\ ,\ \varepsilon(\alpha,0)=\varepsilon(0,\alpha)=1\label{eq:cocycle prop 1-1} \\
	 		\varepsilon(\alpha,\beta)\varepsilon(\alpha+\beta,\gamma) &=\varepsilon(\alpha,\beta+\gamma)\varepsilon(\beta,\gamma)\label{eq:cocycle associativity property-1}
	 	\end{align}
	 \end{subequations}
	for all $\alpha,\beta,\gamma\in Q$, where $(\cdot\vert\cdot)$ denotes the invariant bilinear form.
\end{defn}

The root lattice to any symmetrizable GCM possesses an associated, normalized $2$-cocycle (cp. for instance \cite[cor. 5.5]{Kac98}). 
As this can be checked in a few lines, we provide an explicit construction below.  
\begin{lemma}
	\label{lem:def and existence of standard 2-cocycle}
	Let $A$ be a symmetrizable GCM and fix its symmetrization $A=DB$ s.t. $\left(\alpha_{i}\vert\alpha_{i}\right)=b_{ii}\in2\mathbb{Z}$
	for all $i=1,\dots,n$. Define the bilinear form $\underline{\varepsilon}:Q\times Q\rightarrow\mathbb{Z}$ via bilinear extension of
	\[
	\underline{\varepsilon}\left(\alpha_{i},\alpha_{j}\right)\coloneqq\begin{cases}
		b_{ij} & \text{if }i<j\\
		\frac{1}{2}b_{ii} & \text{if }i=j\\
		0 & \text{if }i>j.
	\end{cases}
	\]
	Then
	\begin{align}
	\varepsilon:\,Q\times Q\rightarrow C_{2}\ ,\quad\varepsilon\left(\alpha,\beta\right)\coloneqq\left(-1\right)^{\underline{\varepsilon}\left(\alpha,\beta\right)}
	\end{align}
	is an associated, normalized $2$-cocycle which we refer to as  the \emph{standard	$2$-cocycle} to $Q$.
\end{lemma}

\begin{proof}
	One has for all $i,j=1,\dots,n$ that $\underline{\varepsilon}\left(\alpha_{i},\alpha_{j}\right)+\underline{\varepsilon}\left(\alpha_{j},\alpha_{i}\right)=b_{ij}$. 
	Spelling out $\beta=\sum_{i=1}^{n}b_{i}\alpha_{i}$	and $\gamma=\sum_{i=1}^{n}c_{i}\alpha_{i}$ this yields 
	\begin{align*}
		\underline{\varepsilon}\left(\beta,\gamma\right)+\underline{\varepsilon}\left(\gamma,\beta\right) &=  \sum_{i,j}b_{i}c_{j}\left(\underline{\varepsilon}\left(\alpha_{i},\alpha_{j}\right)+\underline{\varepsilon}\left(\alpha_{j},\alpha_{i}\right)\right)=\sum_{i,j}b_{i}c_{j}b_{ij} = \left(\beta\vert\gamma\right)\ .
	\end{align*}
	Thus, for all $\beta,\gamma\in Q(A)$ one has that 
	\[
	\varepsilon\left(\beta,\gamma\right)\varepsilon\left(\gamma,\beta\right)=\left(-1\right)^{\underline{\varepsilon}\left(\beta,\gamma\right)+\underline{\varepsilon}\left(\gamma,\beta\right)}=\left(-1\right)^{\left(\beta\vert\gamma\right)}
	\]
	and (\ref{eq:cocycle prop 1-1}) follows by bilinearity	of $\underline{\varepsilon}$. 
	Towards (\ref{eq:cocycle associativity property-1})	one uses 
	\begin{align*}
		\underline{\varepsilon}\left(\alpha,\beta\right)+\underline{\varepsilon}\left(\alpha+\beta,\gamma\right)  &=  \underline{\varepsilon}\left(\alpha,\beta\right)+\underline{\varepsilon}\left(\alpha,\gamma\right)+\underline{\varepsilon}\left(\beta,\gamma\right)
		=  \underline{\varepsilon}\left(\alpha,\beta+\gamma\right)+\underline{\varepsilon}\left(\beta,\gamma\right)
	\end{align*}
	to compute 
	\begin{eqnarray*}
		\varepsilon\left(\alpha,\beta\right)\varepsilon\left(\alpha+\beta,\gamma\right) & = & \left(-1\right)^{\underline{\varepsilon}\left(\alpha,\beta\right)+\underline{\varepsilon}\left(\alpha+\beta,\gamma\right)}=\left(-1\right)^{\underline{\varepsilon}\left(\alpha,\beta+\gamma\right)+\underline{\varepsilon}\left(\beta,\gamma\right)}\\
		& = & \varepsilon\left(\alpha,\beta+\gamma\right)\varepsilon\left(\beta,\gamma\right)\ \forall\,\alpha,\beta,\gamma\in Q(A).
	\end{eqnarray*}
\end{proof}

\begin{defn}
	\label{def:gen Gamma matrix}(Cp. \cite[eq. 4.6]{On higher spin})
	A map $\Gamma\,:\,Q\rightarrow\text{End}\,(S)$ for a finite-dimensional vector space $S$ is called
	a \emph{generalized $\Gamma$-matrix} if it satisfies
	\begin{subequations}\label{eq:Gamma matrix props}
		\begin{align}
			\Gamma(\alpha)\Gamma(\beta) &=(-1)^{\left(\alpha\vert\beta\right)}\Gamma(\beta)\Gamma(\alpha)\label{eq:Gamma matrices (anti-) commutator-1} \\
			\Gamma(0) &=Id\ ,\ \Gamma(\alpha)^{2} =\left(-1\right)^{\frac{1}{2}\left(\alpha\vert\alpha\right)}\ ,\ \Gamma(\alpha) =\Gamma(-\alpha)\label{eq:Gamma matrices normalization and squares-1} \\
			\Gamma(\alpha)\Gamma(\beta) &=\varepsilon(\alpha,\beta)\Gamma(\alpha+\beta)\ ,\label{eq:Gamma matrices homomorphism property-1}
		\end{align}
	\end{subequations}
	for all $\alpha,\beta\in Q$ and an associated, normalized	$2$-cocycle $\varepsilon$.
\end{defn}

The following properties of $2$-cocycles follow rather directly from (\ref{eq:cocycle props}) and will be used in the following computations:
\begin{subequations}
	\begin{align}
	\varepsilon(\alpha,\alpha) &=(-1)^{\frac{1}{2}\left(\alpha\vert\alpha\right)}\label{eq:cocycle of same root twice-1} \\
	\varepsilon(\alpha,\beta) &=\begin{cases}
		\ \ \varepsilon(\beta,\alpha) & \ if\ \left(\alpha\vert\beta\right)=0\,\mod\,2\\
		-\varepsilon(\beta,\alpha) & \ if\ \left(\alpha\vert\beta\right)=1\,\mod\:2.
	\end{cases}\label{eq:cocycle (anti-) symmetry-1}
	\end{align}
\end{subequations} 

\begin{proposition}[This is \cite{Lautenbacher22}, prop. 3.10]
	For a simply-laced GCM $A$, a generalized Gamma matrix $\Gamma\,:\,Q\rightarrow\text{End}\,(S)$
	defines a generalized spin representation $\rho\,:\ \mathfrak{k}\rightarrow\text{End}\,(S)$ via $\rho(X_i)\coloneqq\frac{1}{2}\Gamma(\alpha_{i})$.
\end{proposition}
\begin{proof}
	For simple roots corresponding to adjacent nodes one has 
	\[
	\Gamma(\alpha_{i})\Gamma(\alpha_{j})=-\Gamma(\alpha_{j})\Gamma(\alpha_{i})
	\]
	by eq. (\ref{eq:Gamma matrices (anti-) commutator-1}), whereas for  non-adjacent simple roots one has
	\[
	\Gamma(\alpha_{i})\Gamma(\alpha_{j})=\Gamma(\alpha_{j})\Gamma(\alpha_{i})\:.
	\]
	This shows that the $\rho(X_i)$ commute and anti-commute as required by prop.	\ref{prop:properties of gen spin reps in simply laced case}.	
	The correct normalization is checked via (\ref{eq:Gamma matrices normalization and squares-1}).
\end{proof}
The converse statement is also true.
\begin{proposition}[This is \cite{Lautenbacher22}, prop. 3.11]
	\label{prop:Existence of gen Gamma matrices}
	Let $A$ be a simply-laced GCM	and  $\rho\,:\ \mathfrak{k}\rightarrow\text{End}\,(S)$ a generalized spin representation.
	Then
	\begin{subequations} 
		\begin{align}
			\Gamma\left(\pm\alpha_{i}\right) &\coloneqq2\rho\left(X_{i}\right),\\ \Gamma\left(\alpha_{i_{1}}+\dots+\alpha_{i_{n}}\right) &\coloneqq\left(\prod_{k=1}^{n-1}\varepsilon\left(\alpha_{i_{k}},\alpha_{i_{k+1}}+\dots+\alpha_{i_{n}}\right)\right)\Gamma\left(\alpha_{i_{1}}\right)\cdots\Gamma\left(\alpha_{i_{n}}\right),\label{eq: decomposition of Gamma matrices}
		\end{align}
	\end{subequations}
	defines a generalized $\Gamma$-matrix $\Gamma:Q\rightarrow\text{End}\,(S)$.
\end{proposition}
\begin{proof}
	We begin by showing (\ref{eq:Gamma matrix props}) for $\alpha,\beta\in Q$ of height $1$ and $2$.
	We can always assume $\alpha=\sum_{i=1}^{n}k_{i}\alpha_{i}$ with $k_{i}\geq0$ since $\Gamma\left(-\alpha_{i}\right)=\Gamma\left(\alpha_{i}\right)$ and the standard $2$-cocycle only counts modulo $2$.
	For height $1$ there is nothing to show and for  height $2$ one first checks if the expression of $\Gamma\left(\alpha\right)$ for $\alpha=\alpha_{i}+\alpha_{j}$ is well-defined. 
	Towards this, compute
	\begin{align*}
		\Gamma\left(\alpha_{i}+\alpha_{j}\right) &= \varepsilon\left(\alpha_{i},\alpha_{j}\right) \Gamma\left(\alpha_{i}\right) \Gamma\left(\alpha_{j}\right)\:, \\
		\Gamma\left(\alpha_{j}+\alpha_{i}\right) &= \varepsilon\left(\alpha_{j},\alpha_{i}\right) \Gamma\left(\alpha_{j}\right) \Gamma\left(\alpha_{i}\right)\:.
	\end{align*}
	By the properties of generalized spin representations, the two matrices commute if   $\left(\alpha_{i}\vert\alpha_{j}\right)=0$ and anti-commute if $\left(\alpha_{i}\vert\alpha_{j}\right)=-1$.
	Also, in the first case one has $\varepsilon\left(\alpha_{i},\alpha_{j}\right)=1=\varepsilon\left(\alpha_{j},\alpha_{i}\right)$, whereas in the second case one has $\varepsilon\left(\alpha_{i},\alpha_{j}\right)\varepsilon\left(\alpha_{j},\alpha_{i}\right)=-1$.
	Thus, $\Gamma\left(\alpha_{i}+\alpha_{j}\right)=\Gamma\left(\alpha_{j}+\alpha_{i}\right)$. 
	The only part of (\ref{eq:Gamma matrix props}) left to show  is $\Gamma(\alpha)^{2} =\left(-1\right)^{\frac{1}{2}\left(\alpha\vert\alpha\right)}$ for $\alpha=\alpha_i+\alpha_j$. 
	One has 
	\begin{align*}
		\Gamma\left(\alpha_{i}+\alpha_{j}\right)^2 &= \varepsilon\left(\alpha_{i},\alpha_{j}\right)\varepsilon\left(\alpha_{j},\alpha_{i}\right) \Gamma\left(\alpha_{i}\right)^2 \Gamma\left(\alpha_{j}\right)^2
		 = (-1)^{\left(\alpha_i \vert \alpha_j\right)} (-1)^{\frac{1}{2}\left(\alpha_i \vert \alpha_i\right)} (-1)^{\frac{1}{2}\left(\alpha_j \vert \alpha_j\right)} \\
		 &= (-1)^{\frac{1}{2}\left(\alpha_i+\alpha_j \vert \alpha_i+\alpha_j\right)}\:.
	\end{align*}
	
	From here on we continue by induction on the height $n$, so assume that (\ref{eq:Gamma matrix props}) holds for all $\alpha,\beta\in Q_{+}$ with $\text{ht}\left(\alpha\right)+\text{ht}\left(\beta\right)\leq n-1$.
	In particular, $\Gamma\left(\gamma\right)$ is well-defined for $\text{ht}\left(\gamma\right)\leq n-1$ and one has for all $\alpha,\beta\in Q_{+}$ with $\text{ht}\left(\alpha\right)+\text{ht}\left(\beta\right)\leq n-1$ that 
	\[
	\varepsilon\left(\alpha,\beta\right)\Gamma\left(\alpha\right)\Gamma\left(\beta\right)=\Gamma\left(\alpha+\beta\right)\:.
	\]
	Given $\alpha=\alpha_{i_{1}}+\dots+\alpha_{i_{n}}$ we need to show 
	$\Gamma(\alpha)=\varepsilon\left(\beta,\gamma\right)\Gamma\left(\beta\right)\Gamma\left(\gamma\right)$ for any $\beta,\gamma\in Q(A)_{+}$ s.t. $\alpha=\beta+\gamma$ and $\text{ht}\left(\alpha\right)+\text{ht}\left(\beta\right)= n$.
	We assume w.l.o.g. that $\beta$ contains $\alpha_{i_{1}}$ and multiply  from the left with $\Gamma\left(\alpha_{i_{1}}\right)$ and exploit that $\Gamma\left(\alpha_{i_{1}}\right)=\Gamma\left(-\alpha_{i_{1}}\right)$:
	\begin{align*}
	\varepsilon\left(\beta,\gamma\right) \Gamma\left(\alpha_{i_{1}}\right) \Gamma\left(\beta\right) \Gamma\left(\gamma\right) &= \varepsilon\left(\beta,\gamma\right)\varepsilon\left(\alpha_{i_{1}},\beta\right)\Gamma\left(\beta-\alpha_{i_{1}}\right)\Gamma\left(\gamma\right),\\
	\Gamma\left(\beta-\alpha_{i_{1}}\right) &= \Gamma\left(-\alpha_{i_{1}}+\beta\right) = \varepsilon\left(-\alpha_{i_{1}},\beta\right) \Gamma\left(-\alpha_{i_{1}}\right) \Gamma\left(\beta\right) \\
	&= \varepsilon\left(\alpha_{i_{1}},\beta\right) \Gamma\left(\alpha_{i_{1}}\right)\Gamma\left(\beta\right).
	\end{align*}
	Since $\text{ht}\left(\beta-\alpha_{i_{1}}\right)+\text{ht}\left(\gamma\right)=n-1$ one has $\Gamma\left(\beta-\alpha_{i_{1}}\right)\Gamma\left(\gamma\right)=\varepsilon\left(\beta-\alpha_{i_{1}},\gamma\right)\Gamma\left(\beta-\alpha_{i_{1}}+\gamma\right)$ and in combination with $\varepsilon\left(\beta-\alpha_{i_{1}},\gamma\right)=\varepsilon\left(\beta+\alpha_{i_{1}},\gamma\right)$ and $\beta-\alpha_{i_{1}}+\gamma=\alpha_{i_{2}}+\dots+\alpha_{i_{n}}$ this yields
	\begin{align*}
		\varepsilon\left(\beta,\gamma\right) \Gamma\left(\alpha_{i_{1}}\right) \Gamma\left(\beta\right) \Gamma\left(\gamma\right) &= \varepsilon\left(\beta,\gamma\right) \varepsilon\left(\alpha_{i_{1}},\beta\right) \varepsilon\left(\beta+\alpha_{i_{1}},\gamma\right) \Gamma\left(\alpha_{i_{2}}+\dots+\alpha_{i_{n}}\right)\:.
	\end{align*}
	By (\ref{eq:cocycle associativity property-1}) one has $	\varepsilon\left(\alpha_{i_{1}}+\beta,\gamma\right) = \varepsilon\left(\alpha_{i_{1}},\beta\right) \varepsilon\left(\alpha_{i_{1}},\beta+\gamma\right) \varepsilon\left(\beta,\gamma\right)$ which implies $\varepsilon\left(\beta,\gamma\right)\varepsilon\left(\alpha_{i_{1}},\beta\right)\varepsilon\left(\beta+\alpha_{i_{1}},\gamma\right)=\varepsilon\left(\alpha_{i_{1}},\beta+\gamma\right)$ so that the above equation simplifies further to 
	\begin{align*}
		&\varepsilon\left(\beta,\gamma\right)\Gamma\left(\alpha_{i_{1}}\right)\Gamma\left(\beta\right)\Gamma\left(\gamma\right) \\ 
		&=  \varepsilon\left(\alpha_{i_{1}},\beta+\gamma\right)\Gamma\left(\alpha_{i_{2}}+\dots+\alpha_{i_{n}}\right)\\
		&=  \varepsilon\left(\alpha_{i_{1}},\alpha_{i_{1}}+\dots+\alpha_{i_{n}}\right)\left(\prod_{k=2}^{n-1}\varepsilon\left(\alpha_{i_{k}},\alpha_{i_{k+1}}+\dots+\alpha_{i_{n}}\right)\right)\Gamma\left(\alpha_{i_{2}}\right)\cdots\Gamma\left(\alpha_{i_{n}}\right)
	\end{align*}
	Together with 
	\begin{align*}
		\varepsilon\left(\alpha_{i_{1}},\alpha_{i_{1}}+\dots+\alpha_{i_{n}}\right) &=  \varepsilon\left(\alpha_{i_{1}},\alpha_{i_{1}}\right) \varepsilon\left(2\alpha_{i_{1}},\alpha_{i_{2}}+\dots+\alpha_{i_{n}}\right) \varepsilon\left(\alpha_{i_{1}},\alpha_{i_{2}}+\dots+\alpha_{i_{n}}\right)\\
		&=  -\varepsilon\left(\alpha_{i_{1}},\alpha_{i_{2}}+\dots+\alpha_{i_{n}}\right)
	\end{align*}
	one arrives at 
	\[
	\varepsilon\left(\beta,\gamma\right)\Gamma\left(\alpha_{i_{1}}\right)\Gamma\left(\beta\right)\Gamma\left(\gamma\right)=-\left(\prod_{k=1}^{n-1}\varepsilon\left(\alpha_{i_{k}},\alpha_{i_{k+1}}+\dots+\alpha_{i_{n}}\right)\right)\Gamma\left(\alpha_{i_{2}}\right)\cdots\Gamma\left(\alpha_{i_{n}}\right).
	\]
	One the other hand one has
	\begin{align*}
		\Gamma\left(\alpha_{i_{1}}\right)\Gamma\left(\alpha_{i_{1}}+\dots+\alpha_{i_{n}}\right) &=  \left(\prod_{k=1}^{n-1} \varepsilon\left(\alpha_{i_{k}},\alpha_{i_{k+1}}+\dots+\alpha_{i_{n}}\right)\right) \Gamma\left(\alpha_{i_{1}}\right)^{2} \Gamma\left(\alpha_{i_{2}}\right)\cdots \Gamma\left(\alpha_{i_{n}}\right)\\
		&=-\left(\prod_{k=1}^{n-1}\varepsilon\left(\alpha_{i_{k}},\alpha_{i_{k+1}}+\dots+\alpha_{i_{n}}\right)\right)\Gamma\left(\alpha_{i_{2}}\right)\cdots\Gamma\left(\alpha_{i_{n}}\right),
	\end{align*}
	and as multiplication with $\Gamma\left(\alpha_{i_{1}}\right)$ can be inverted this yields 
	\[
	\varepsilon\left(\beta,\gamma\right) \Gamma\left(\beta\right) \Gamma\left(\gamma\right) = \Gamma\left(\alpha_{i_{1}}+\dots+\alpha_{i_{n}}\right)
	\]
	for all $\beta,\gamma\in Q_{+}$ of height $<n$ s.t. $\beta+\gamma=\alpha_{i_{1}}+\dots+\alpha_{i_{n}}$ and thus shows well-definedness and (\ref{eq:Gamma matrices homomorphism property-1}) by induction.
	We can now use this to prove (\ref{eq:Gamma matrices (anti-) commutator-1}) by using $\Gamma\left(\alpha+\beta\right)=\Gamma\left(\beta+\alpha\right)$ and (\ref{eq:cocycle prop 1-1}):
	\[
	\varepsilon\left(\alpha,\beta\right)\Gamma\left(\alpha\right)\Gamma\left(\beta\right)=\varepsilon\left(\beta,\alpha\right)\Gamma\left(\beta\right)\Gamma\left(\alpha\right)\ \Leftrightarrow\ \Gamma\left(\alpha\right)\Gamma\left(\beta\right)=\left(-1\right)^{\left(\alpha\vert\beta\right)}\Gamma\left(\beta\right)\Gamma\left(\alpha\right)
	\]
	In order to show the missing parts of (\ref{eq:Gamma matrices normalization and squares-1}) one computes
	\begin{align*}
		\Gamma\left(\alpha+\beta\right)^{2} &= \varepsilon\left(\alpha,\beta\right) \Gamma\left(\alpha\right) \Gamma\left(\beta\right) \varepsilon\left(\beta,\alpha\right) \Gamma\left(\beta\right)\Gamma\left(\alpha\right)\\
		&=  \varepsilon\left(\alpha,\beta\right) \varepsilon\left(\beta,\alpha\right) \left(-1\right)^{\frac{1}{2}\left(\beta\vert\beta\right)} \left(-1\right)^{\frac{1}{2}\left(\alpha\vert\alpha\right)}\\
		&= \left(-1\right)^{\left(\alpha\vert\beta\right) +\frac{1}{2}\left(\beta\vert\beta\right) +\frac{1}{2}\left(\alpha\vert\alpha\right)} =\left(-1\right)^{\frac{1}{2}\left(\alpha+\beta\vert\alpha+\beta\right)}.
	\end{align*}
\end{proof}
There are some identities of $\Gamma$-matrices that we would like to state explicitly.
\begin{lemma}
	Let $A$ be a symmetrizable GCM with a symmetrization as in	lemma \ref{lem:def and existence of standard 2-cocycle} and denote the standard $2$-cocycle by $\varepsilon$.
 	Then any generalized $\Gamma$-matrix satisfies 
	\begin{align}
		\Gamma\left(\alpha+\beta\right) &=\Gamma\left(\alpha-\beta\right),\ \Gamma\left(\alpha+2\beta\right)=\Gamma\left(\alpha\right)\ \forall\,\alpha,\beta\in Q.\label{eq:Gamma matrices of certain sums of roots}
	\end{align}
\end{lemma}

\begin{proof}
	With the help of (\ref{eq:Gamma matrices normalization and squares-1}) one computes 
	\[
	\Gamma\left(\alpha+\beta\right)=\varepsilon\left(\alpha,\beta\right)\Gamma\left(\alpha\right)\Gamma\left(\beta\right),\ \Gamma\left(\alpha-\beta\right)=\varepsilon\left(\alpha,-\beta\right)\Gamma\left(\alpha\right)\Gamma\left(-\beta\right)
	\]
	and because	$\varepsilon$ is a standard $2$-cocycle one has further that  
	\[
	\varepsilon\left(\alpha,-\beta\right)=\left(-1\right)^{\underline{\varepsilon}\left(\alpha,-\beta\right)}=\left(-1\right)^{-\underline{\varepsilon}\left(\alpha,\beta\right)}=\left(-1\right)^{\underline{\varepsilon}\left(\alpha,\beta\right)}=\varepsilon\left(\alpha,\beta\right)\:,
	\]
	which shows the first part of the claim. 
	Regarding the second claim, one has that $\Gamma\left(\alpha+2\beta\right)=\Gamma\left(\alpha+\beta+\beta\right)=\Gamma\left(\alpha+\beta-\beta\right)=\Gamma(\alpha)$ by the first part.
\end{proof}
\begin{remark}
	The above lemma shows that $\Gamma$-matrices are unable to capture distinctions beyond modulo $2Q$. As a consequence, they can not distinguish different roots in symmetrizable root systems such as $\alpha_i+2\alpha_j$. 
	This should come as no surprise as $\Gamma$-matrices are tailored towards a more global description of spin representations for simply-laced $A$. 
	
	There exist generalized spin representations in the symmetrizable, non-simply-laced cases (cp. \cite{Kleinschmidt2020}) $\mf{k}(AE_3)$, $\mf{k}(G_2^{++})$ and $\mf{k}(BE_{10})$ which all arise from a generalized spin representation of simply-laced type of suitable quotients which are $\mf{k}(A_2)\oplus \mf{k}(A_1)$, $\mf{k}(A_4)$ and $\mf{k}(E_9)$ respectively.	
	
	In \cite{Hainke}, the generalized spin representations for the general symmetrizable case are constructed via embedding $\mf{k}(A)$ in a simply-laced cover $\mf{k}(\tilde{A})$ first.
	This cover would admit generalized $\Gamma$-matrices w.r.t. the root system of type $\tilde{A}$ but it is an open question how the ones needed for the representation of $\mf{k}(A)$ spell out in terms of the root system of type $A$.
	For this, it might be necessary to replace the $2$-cocycles by ones of higher order.

\end{remark}

\begin{proposition}[This is \cite{Lautenbacher22}, prop. 3.14] \label{prop:rep matrix is always prop to Gamma matrix}
	Let $A$ be	a simply-laced GCM as in lemma \ref{lem:def and existence of standard 2-cocycle}	and denote the standard $2$-cocycle by $\varepsilon$ and let $\left(\rho,S\right)$ be a generalized spin representation with associated $\Gamma$-matrix $\Gamma:Q\rightarrow End\left(S\right)$ according to prop. \ref{prop:Existence of gen Gamma matrices}.
	Then for all $x\in\mathfrak{k}_{\alpha}$ there exists $c(x)\in\mathbb{R}$	such that
	\begin{equation}
		\rho\left(x\right)=c(x)\cdot\Gamma\left(\alpha\right)\:.\label{eq:rep matrix is always prop to Gamma matrix}
	\end{equation}
\end{proposition}

\begin{proof}
	Given $\alpha\in\Delta_{+}(A)$ there exist an ordered series of simple roots $\beta_{1},\dots,\beta_{N}\in\Pi$ such that $\alpha=\beta_{1}+\dots+\beta_{N}$. 
	To such a fixed	decomposition of $\alpha$ we set $X_{\beta_{1}+\dots+\beta_{N}}\coloneqq\left[X_{\beta_{1}},\left[X_{\beta_{2}},\left[\dots,X_{\beta_{N}}\right]\dots\right]\right]$. 
	As $\mf{k}(A)$ is generated by the $X_i$ there exists at least one such ordered series such that $X_{\beta_{1}+\dots+\beta_{N}}$ is nonzero but for now we do not need to assume this.
	By prop. \ref{prop:Existence of gen Gamma matrices} one has
	\[
	\rho\left(X_{\beta_{1}+\dots+\beta_{N}}\right)=\frac{1}{2^{N}}\left[\Gamma\left(\beta_{1}\right),\left[\Gamma\left(\beta_{2}\right),\left[\dots,\Gamma\left(\beta_{N}\right)\right]\dots\right]\right]\:.
	\]
	From (\ref{eq:Gamma matrices (anti-) commutator-1}) one concludes that
	\[
	\left[\Gamma\left(\alpha\right),\Gamma\left(\beta\right)\right]=\Gamma\left(\alpha\right)\Gamma\left(\beta\right)-\Gamma\left(\beta\right)\Gamma\left(\alpha\right)=\begin{cases}
		0 & \text{if }\left(\alpha\vert\beta\right)\in2\mathbb{Z}\\
		2\Gamma\left(\alpha\right)\Gamma\left(\beta\right) & \text{if }\left(\alpha\vert\beta\right)\in\mathbb{Z}\setminus2\mathbb{Z},
	\end{cases}
	\]
	so that, exploiting $\Gamma(\alpha)\Gamma(\beta)=\varepsilon(\alpha,\beta)\Gamma(\alpha+\beta)$, there are two possibilities for $\rho\left(X_{\beta_{1}+\dots+\beta_{N}}\right)$:
	\begin{align*}
	\rho\left(X_{\beta_{1}+\dots+\beta_{N}}\right) &= \frac{1}{2^{N}} \left[\Gamma\left(\beta_{1}\right), \left[\Gamma\left(\beta_{2}\right), \left[\dots,\Gamma\left(\beta{}_{n}\right)\right] \dots \right] \right] \\
	&= \begin{cases}
		0\\
		\frac{1}{2}\varepsilon(\beta_{1},\beta_2+\dots+\beta_{N})\cdots\varepsilon(\beta_{N-1},\beta_N)\Gamma\left(\beta_{1}+\beta_{2}+\dots+\beta_{N}\right)\:.
	\end{cases}
	\end{align*}
	The first case occurs if there exists $i$ such that $\left(\beta_{i}\vert\beta_{i+1}+\dots+\beta_{N}\right)\in2\mathbb{Z}$ and the second case if no such $i$ exists.
	By (\ref{eq: decomposition of Gamma matrices}) one has that 
	\[
	\Gamma(\alpha)=\left(\prod_{k=1}^{N-1}\varepsilon\left(\beta_{k},\beta_{{k+1}}+\dots+\beta_{N}\right)\right)\Gamma\left(\beta_{{1}}\right)\cdots\Gamma\left(\beta_{N}\right)
	\] 
	so that one in fact obtains
	\begin{align}
		\rho\left(X_{\beta_{1}+\dots+\beta_{N}}\right) 		&= \begin{cases}
			0\\
			\frac{1}{2}\Gamma\left(\beta_{{1}}\right)\cdots\Gamma\left(\beta_{N}\right)\:,
		\end{cases}
	\end{align}
	 because any cocycle squared is $1$. The claim follows for all $X_{\beta_{1}+\dots+\beta_{N}}\in\mf{k}(A)$.
	 Note that the order of the nested commutator here is crucial and is reflected in the order of the $\Gamma(\beta_i)$ in the product.
	Recall that $\mf{k}$ is not graded by $\Delta$ but instead filtered by $\Delta_+$ and set  $\mathfrak{k}_{<\alpha}\coloneqq\bigoplus_{\beta<\alpha}\mathfrak{k}_{\beta}$.
	To $x\in\mathfrak{k}_{\alpha}$ there exist ordered decompositions	$\beta_{1}^{(j)}+\dots+\beta_{N}^{(j)}=\alpha$ for $j=1,\dots,k=\dim \mf{k}_{\alpha}$,
	$c_{j}\in\mathbb{K}$ and a remainder $r\in\mathfrak{k}_{<\alpha}$ such that
	\[
	\sum_{j=1}^{k}c_{j}X_{\beta_{1}^{(j)}+\dots+\beta_{N}^{(j)}}=x+r.
	\]
	For $y_{1}\in\mathfrak{k}_{\beta},y_{2}\in\mathfrak{k}_{\gamma}$ one has $\left[y_{1},y_{2}\right]\in\mathfrak{k}_{\beta+\gamma}\oplus\mathfrak{k}_{\pm(\beta-\gamma)}$ due to the filtered structure of $\mf{k}$.
	Now $r$ also has a decomposition $r=\bigoplus_{\gamma}y_{\gamma}$ which is furthermore	such that $y_{\gamma}\in\mathfrak{k}_{\gamma}$ is nonzero only if	$\gamma<\alpha$ and $\alpha-\gamma\in2Q(A)$.
	Thus, by induction on $\text{ht}\left(\alpha\right)$ one concludes
	\[
	\rho(x)=\sum_{j=1}^{k}c_{j}\rho\left(X_{\beta_{1}^{(j)}+\dots+\beta_{N}^{(j)}}\right)-\rho(r)=c(x)\cdot\Gamma(\alpha),
	\]
	 for all $x\in\mathfrak{k}_{\alpha}$ and $\alpha\in\Delta_{+}(A)$ because $\Gamma\left(\alpha\right)=\Gamma\left(\gamma\right)$
	if $\alpha-\gamma\in2Q(A)$ by (\ref{eq:Gamma matrices of certain sums of roots}).
\end{proof}
At the present point we are not able to say when $\rho(x)$ for $x\in\mf{k}_{\alpha}$ is nonzero or not or what the constant $c(x)$ is. 
We will pick up this thread in section \ref{sec: Spin(A)} once we introduced the spin cover of the maximal compact subgroup $K(A)$ and how it interacts with generalized spin representations $\rho$.
\subsection{The higher spin representations $\mathcal{S}_{\frac{3}{2}}$ and $\mathcal{S}_{\frac{5}{2}}$ }
In this subsection we provide a description of the higher spin representations from \cite{On higher spin} in the phrasing of \cite{Ext gen spin reps}, where we put more emphasis on how these representations feature representations of $W(A)$ as a crucial piece. 
The name stems from the fact that the weight system of these representations also includes larger eigenvalues than $\pm \frac{i}{2}$.
Throughout, we will assume that $A$ is simply-laced.

\begin{proposition}
	\label{prop:criterion for homomorphism} 
	Let $A$ be a simply-laced GCM, let $\rho$ be a generalized spin representation of $\mf{k}$ as in def. \ref{Def:gen spin rep for simply laced case}, where we denote the module by $S$, and let $V$ be a finite-dimensional vector space. 
	Let $X_{1},\dots,X_{n}$ denote the Berman generators of $\mathfrak{k}$ and set $\Delta^{re}\supset\widetilde{\Delta}=\left\{ \alpha_{1},\dots,\alpha_{n}\right\} \cup\left\{ \alpha_{i}+\alpha_{j}\ \vert\ (i,j)\in\mathcal{E}\left(A\right)\right\} $.
	Consider a map  $\tau\,:\widetilde{\Delta}\rightarrow\mathrm{End}\left(V\right)$ satisfying
	\begin{eqnarray}
		\left[\tau(\alpha),\tau(\beta)\right] & = & 0\quad\qquad\quad\quad\,\text{if }\left(\alpha\vert\beta\right)=0\label{eq:commutator for tau as endo}\\
		\left\{ \tau(\alpha),\tau(\beta)\right\}  & = & \tau(\alpha\pm\beta)\quad\text{if \ensuremath{\left(\alpha\vert\beta\right)=\mp1}and \ensuremath{\alpha\pm\beta\in\widetilde{\Delta}}}\label{eq:anticommutator for tau as endo}
	\end{eqnarray}
	for all $\alpha,\beta\in\widetilde{\Delta}$.
	Then the assignment $\sigma\left(X_{i}\right):=\tau\left(\alpha_{i}\right)\otimes2\rho\left(X_{i}\right)\in\mathrm{End}\left(V\otimes S\right)$ extends to a finite-dimensional	representation $\sigma$ of $\mathfrak{k}(A)$.
\end{proposition}

\begin{proof}
	This is originally \cite[eq. (5.1)]{On higher spin} but the phrasing is as
	in \cite{Ext gen spin reps}.
\end{proof}

If the map $\tau$ admits an extension to all real roots one has the following identity: 
\begin{lemma}
	\label{lem:How sigma relates to tau and rho on commutators}
	Let $\tau\,:\Delta^{re}\rightarrow\mathrm{End}\left(V\right)$ 
	satisfy equations (\ref{eq:commutator for tau as endo})  and (\ref{eq:anticommutator for tau as endo}) with $ \widetilde{\Delta}$ replaced by $\Delta^{re}$.
	Let $\{i,j\},\,\{ j,k\}\in\mathcal{E}(A)$ but $\{i,k\}\notin\mathcal{E}(A)$,
	then 
	\begin{eqnarray}
		\sigma\left(\left[X_{i},X_{j}\right]\right) & = & \tau\left(\alpha_{i}+\alpha_{j}\right)\otimes2\rho\left(\left[X_{i},X_{j}\right]\right)\label{eq:sigma on single commutator}\\
		\sigma\left(\left[X_{i},\left[X_{j},X_{k}\right]\right]\right) & = & \tau\left(\alpha_{i}+\alpha_{j}+\alpha_{k}\right)\otimes2\rho\left(\left[X_{i},\left[X_{j},X_{k}\right]\right]\right)\ .\label{eq:sigma on triple commutator}
	\end{eqnarray}
\end{lemma}
\begin{proof}
	The proof consists of elementary computations exploiting the commutation and anti-commutation relations among the $\rho(X_i)$ as well as (\ref{eq:commutator for tau as endo})  and (\ref{eq:anticommutator for tau as endo}).
\end{proof}

Given a simply-laced GCM $A$ and real roots $\alpha,\beta$ with $\left(\alpha,\beta\right)=-1$, the reflections  $s_{\alpha}$, $s_\beta$   w.r.t. $\alpha$ and $\beta$ generate a subgroup of $\mathfrak{S}_{\alpha,\beta}:=\left\langle s_{\alpha},s_{\beta}\right\rangle <W(A)$ which is isomorphic to $\mathfrak{S}_{3}$, the symmetric group on three letters.
The three distinct irreducible representations of $\mathfrak{S}_{3}$ are the trivial representation $U$ (dimension $1$), the sign representation $U'$ (dimension $1$) and the standard representation  $E$ of dimension $2$. 

\begin{proposition}[This is \cite{Ext gen spin reps}, rem. 4.2 (iv), based on an observation by Paul Levy]
	\label{prop: master eq. holds away from sign reps} 
	Let $A$ be simply-laced and $\eta:W(A)\rightarrow\text{End}(V)$ be a finite-dimensional representation	of  $W(A)$. 
	Then 
	\begin{equation}
		\tau:\Delta^{re}(A)\rightarrow\text{End}\left(V\right),\ \alpha\mapsto\eta\left(s_{\alpha}\right)-\frac{1}{2}Id\label{eq:weyl group ansatz}
	\end{equation}
	satisfies eqs. (\ref{eq:commutator for tau as endo}) and (\ref{eq:anticommutator for tau as endo})
	if the restriction of $\eta$ to any $\mathfrak{S}_{\alpha_{i},\alpha_{j}}$
	such that $\alpha_{i}$, $\alpha_{j}$ are adjacent simple roots does
	not contain the sign representation of $\mathfrak{S}_{3}$ as an irreducible
	factor.
\end{proposition}

\begin{proof}
	If $\left(\alpha\vert\beta\right)=0$ then $s_{\alpha},s_{\beta}$
	commute and so do $\tau(\alpha)$, $\tau(\beta)$ as required. 
	For $\left(\alpha\vert\beta\right)=-1$ one has  with $s_{\alpha+\beta}=s_{\beta}s_{\alpha}s_{\beta}$ that
	that $\tau\left(\alpha+\beta\right)=\left\{ \tau(\alpha),\tau(\beta)\right\} $
	is equivalent to 
	\begin{equation}\label{eq: W(A)-rep. criterion}
	0\overset{!}{=}-\eta\left(s_{\beta}s_{\alpha}s_{\beta}\right)+\eta\left(s_{\alpha}\right)\eta\left(s_{\beta}\right)+\eta\left(s_{\beta}\right)\eta\left(s_{\alpha}\right)-\eta\left(s_{\alpha}\right)-\eta\left(s_{\beta}\right)+Id.
	\end{equation}
	For the trivial representation this is easily seen to be true, whereas
	it is false for the sign representation as then 
	\[
	-Id=\eta\left(s_{\beta}s_{\alpha}s_{\beta}\right)=\eta\left(s_{\alpha}\right)=\eta\left(s_{\beta}\right).
	\]
	One can set up the standard representation as the subspace $\text{span}_{\mathbb{R}}\left\{ \alpha,\beta\right\} \subset\mathfrak{h}^{\ast}$.
	In this basis one checks explicitly that (\ref{eq: W(A)-rep. criterion}) holds. 
	Since any finite-dimensional representation of $\mathfrak{S}_{3}$
	is completely reducible one concludes that (\ref{eq:weyl group ansatz})
	provides a representation if $\eta$ restricted to $\mathfrak{S}_{\alpha_{i},\alpha_{j}}$ contains no copies of the sign representation.
	Note that it suffices to consider simple roots, because the relations (\ref{eq:anticommutator for tau as endo}) and (\ref{eq:commutator for tau as endo}) need only be satisfied for $\alpha,\beta\in\widetilde{\Delta}$, where $\widetilde{\Delta}$ consists of simple roots and all roots $\alpha_i+\alpha_j$ where $i$ and $j$ are adjacent. 
	Those are the only root spaces involved in the pairwise Berman relations $\left[X_i, \left[X_i,X_j\right]\right]=-X_j$.
\end{proof}
\begin{lemma} 
\label{prop:No sign reps in V and Sym^2V}
Let $\left(\mf{h},\Pi,\Pi^\vee\right)$ be a realization of $A$ and set $V_{1}:=\mathfrak{h}^{\ast}$, $V_{2}:=\text{Sym}^{2}\left(\mathfrak{h}^{\ast}\right)$, the symmetric product of $\mf{h}^\ast$ with itself.
Then the standard representation of $W\left(A\right)$ on $V_1$ and the induced representation on $V_{2}$ contain no copies of the sign representation upon restriction to any $\mathfrak{S}_{\alpha_{i},\alpha_{j}}$ for $\{i,j\}\in\mathcal{E}(A)$.
\end{lemma}

\begin{proof}
After restriction to $\mathfrak{S}_{\alpha_{i},\alpha_{j}}$ consider a basis $\left\{ \alpha_{i},\alpha_{j}\right\} \cup\left\{ b_{1},\dots,b_{m-2}\right\} $ of $\mf{h}^{\ast}$, where $m=2n-\text{rnk}(A)$, such that the $b_{i}$ are orthogonal to both $\alpha_{i}$ and $\alpha_{j}$. 
Then $\text{span}\left\{ \alpha_{i},\alpha_{j}\right\} $
forms a copy of the standard representation of $\mathfrak{S}_{\alpha_{i},\alpha_{j}}$
while $\text{span}\left\{ b_{1},\dots,b_{m-2}\right\} $ decomposes
into $m-2$ copies of the trivial representation. 
For $\text{Sym}^{2}\left(\mathfrak{h}^{\ast}\right)$
one can use the basis
\begin{align*}
\left\{ \alpha_{i}\alpha_{i},\,\alpha_{j}\alpha_{j},\,\left(\alpha_{i}+\alpha_{j}\right)\left(\alpha_{i}+\alpha_{j}\right)\right\} \cup\left\{ \alpha_{i}b_{k},\alpha_{j}b_{k}\ \vert\,k=1,\dots,m-2\right\} \\\cup\left\{ b_{k}b_{l}\,\vert\,1\leq k\leq l\leq m-2\right\} .
\end{align*}
to see that  only the trivial representation and the standard
representation occur in the decomposition of $V_2$ restricted to $\mf{S}_{\alpha_i,\alpha_j}$. 
\end{proof}

We obtain the following result as an immediate consequence of the previous lemma and prop. \ref{prop:criterion for homomorphism}. 
Note that prop. \ref{prop:criterion for homomorphism} can be applied even if $A$ is not regular.
\begin{theorem}(Cp. \cite{Higher spin realizations} and \cite{Ext gen spin reps})
	\label{thm:S32 and S52} Let $\left(\eta_{1},V_{1}\right)$ and $\left(\eta_{2},V_{2}\right)$ with $V_{1}:=\mathfrak{h}^{\ast}$, $V_{2}:=\text{Sym}^{2}\left(\mathfrak{h}^{\ast}\right)$  be the standard and induced representations of $W(A)$, respectively.
	Then  $\tau_{s}:\widetilde{\Delta}\rightarrow\text{End}\left(V_{s}\right)$,	$\tau_{s}(\alpha)=\eta_{s}\left(s_{\alpha}\right)-\frac{1}{2}Id$ for $s\in\left\{ 1,2\right\} $  satisfies eqs. (\ref{eq:commutator for tau as endo}) and (\ref{eq:anticommutator for tau as endo}).
	Let $X_{1},\dots,X_{n}$ denote the Berman generators of $\mathfrak{k}(A)$ and assign 
	\[
	\sigma_{\frac{2s+1}{2}}:\ X_{i}\mapsto\tau_{s}\left(\alpha_{i}\right)\otimes2\rho\left(X_{i}\right),
	\]
	where $\rho$ is a $\frac{1}{2}$-spin representation as in thm.	\ref{thm:Properties of spin rep's image}.
	Then $\sigma_{\frac{2s+1}{2}}$ extends to a representation of $\mathfrak{k}(A)$.
\end{theorem}
The above representation is called the \emph{$\frac{2s+1}{2}$-spin	 representation} in \cite{Higher spin realizations} and \cite{Ext gen spin reps}, motivated by the fact that the eigenvalues of $\sigma_{\frac{2s+1}{2}}( X_{j})$ are half-integral up to a factor of $i$.
However, it is somewhat misleading that the largest magnitude of eigenvalues occurring is $\frac{3}{2}$ in both cases and not $\frac{2s+1}{2}$.
The above ansatz fails for $V_{3}:=\text{Sym}^{3}\left(\mathfrak{h}^{\ast}\right)$, because  $\alpha\beta(\alpha+\beta)$ spans a sign representation of $\mathfrak{S}_{\alpha,\beta}$ for adjacent simple roots $\alpha,\beta$. 
However, this sign representation is the only one that appears in the decomposition of $V_{3}$, regardless of $A$.

\begin{lemma}
	\label{lem:One sign rep on Sym^3V} 
	Let $A$ be simply-laced and let  $\alpha,\beta\in\Pi$ be adjacent simple roots. 
	Then the induced representation of  $W(A)$ on $\text{Sym}^{3}\left(\mathfrak{h}^{\ast}\right)$  restricted	to $\mathfrak{S}_{\alpha,\beta}$	contains exactly one copy of the sign representation, spanned by $\alpha\beta\left(\alpha+\beta\right)\in\text{Sym}^{3}\left(\mathfrak{h}^{\ast}\right)$.
\end{lemma}
\begin{proof}
	Consider a basis $\mathcal{B}_1=\{ \alpha,\beta,b_{1},\dots,b_{m-2}\}$ for $\mathfrak{h}^{\ast}$ such that $b_{1},\dots,b_{m-2}$ are orthogonal  to $\alpha$ and $\beta$. 
	This provides a basis $\mathcal{B}_3=\{e_1e_2e_3\ \vert\ e_1,e_2,e_3\in\mathcal{B}_1,\ e_1\leq e_2 \leq e_3\}$ for $V_3$ w.r.t. some ordering of $\mathcal{B}_1$.
	Then elements $x\in\mathcal{B}_3$ which contain $b_{k}$ for some $k$ behave like copies of	$V$ or $\text{Sym}^{2}(V)$ and therefore span trivial or standard	representations according to prop. \ref{prop:No sign reps in V and Sym^2V}.
	The remaining subspace of products containing only $\alpha$ and/or $\beta$ is $4$-dimensional and computation of its character reveals that it contains one of each irreducible representations of $\mf{S}_{\alpha,\beta}$, i.e. the trivial, the sign and the standard representation.
	One also checks by direct computation that the sign	representation is spanned by $\alpha\beta\left(\alpha+\beta\right)$. 	
\end{proof}

\subsection{\label{subsec:7/2-spin rep}The higher spin representation $\mathcal{S}_{\frac{7}{2}}$}

In this section's thm. \ref{thm:7/2 spin rep} we provide a coordinate-free version of the so-called $\frac{7}{2}$-spin representation described in \cite{On higher spin}.
The main idea is a suitable extension of the previous Weyl group type ansatz on the module $V_3\otimes S$, where $V_3\coloneqq\text{Sym}^{3}V$ with $V=\mf{h}^\ast$ and $S$ is the carrier space of the $\frac{1}{2}$-spin representation from thm. \ref{thm:Properties of spin rep's image}.
For better readability, we will denote the standard invariant form on $\mathfrak{h}^{\ast}$
by $b\left(\cdot,\cdot\right)$ in this section. 
According to prop. \ref{prop:criterion for homomorphism}, we need to find maps $\tau: \Delta^{re}\rightarrow End(V_3)$
satisfying 
\begin{subequations}\label{eq: full master equation}
	\begin{eqnarray}
		\left[\tau(\alpha),\tau(\beta)\right] & = & 0\ \text{if }b\left(\alpha,\beta\right)=0\label{eq:trivial part of master equation}\\
		\left\{ \tau(\alpha),\tau(\beta)\right\}  & = & \tau\left(\alpha\pm\beta\right)\ \text{if }b\left(\alpha,\beta\right)=\mp1\label{eq:master equation}\:.
	\end{eqnarray}
\end{subequations}
We will need some standard results on the structure of symmetric products of vector spaces. 
Fix the normalization on $\text{Sym}^{3}V$ w.r.t. $V^{\otimes3}$ via 
\begin{equation}
	\text{Sym}^{3}V\ni v_{1}\cdot v_{2}\cdot v_{3}:=\frac{1}{3!}\sum_{\sigma\in\mathfrak{S}_{3}}v_{\sigma(1)}\otimes v_{\sigma(2)}\otimes v_{\sigma(3)}\ .\label{eq:multiplication in Sym(3)}
\end{equation}
The induced $W(A)$-invariant, non-degenerate bilinear form on $\text{Sym}^{3}V$ is given by
\begin{eqnarray*}
	b\left(v_{1}\cdot v_{2}\cdot v_{3}\,,\,u_{1}\cdot u_{2}\cdot u_{3}\right) & = &  \frac{1}{3!}\sum_{\sigma\in\mathfrak{S}_{3}}b\left(v_{\sigma(1)},u_{1}\right)\dots b\left(v_{\sigma(3)},u_{3}\right)
\end{eqnarray*}
Let $e_{1},\dots,e_{m}$ be a basis of $V$ and set 
\begin{equation}
	\omega_{ij}:=b\left(e_{i},e_{j}\right),\ \left(\omega^{ij}\right):=\left(\omega_{ij}\right)^{-1}\ \Leftrightarrow\ \sum_{l}\omega^{kl}\omega_{ln}=\delta_{\,n}^{k}.\label{eq:definition of omega and its inverse}
\end{equation}
Define the \emph{symmetric insertion} $\psi\,:\,V\rightarrow\text{Sym}^{3}V$ via 
\begin{eqnarray*}
	\psi\left(v\right) & = & \frac{1}{3!}\cdot\sum_{k,l=1}^{m}\omega^{kl}\left(v\otimes e_{k}\otimes e_{l}+e_{k}\otimes v\otimes e_{l}+e_{k}\otimes e_{l}\otimes v\right)\ .
\end{eqnarray*}
Symmetric insertions play an important role in invariant theory and
the above definition of $\psi$ does not depend on the chosen basis (cp.
\cite[secs. 17.3 \& 19.5]{Fulton Harris}).
In the language of theoretical physics used in \cite{On higher spin}, the coordinates of the element $\psi(\alpha)$ w.r.t. some basis would be given as $\alpha_{(a}G_{bc)}$, where $a,b,c=1,\dots, \dim(V)$, $G_{bc}$ denote the elements of the bilinear form in the chosen basis (corresponding directly to the $\omega_{ij}$ above) and the parentheses denote complete symmetrization.

We consider the ansatz 
\begin{subequations}\label{eq:ansatz 7/2}
	\begin{align}
		\tau(\alpha)&=s_{\alpha}-\frac{1}{2}Id+f(\alpha)\quad\forall\,\alpha\in\Delta^{re}\left(A\right),\\
		f\left(\alpha\right)&\coloneqq v\left(\alpha\right)\cdot b\left(v(\alpha)\,\vert\,\cdot\right) \in End(\text{Sym}^3V),\\
		v\left(\alpha\right)&=p\cdot\alpha\alpha\alpha+q\cdot\psi\left(\alpha\right)\ \in \text{Sym}^3V,\label{eq:ansatz for v(alpha) in 7/2 spin}
	\end{align}
\end{subequations}
where $s_{\alpha}$, by slight abuse of notation, denotes the action of the Weyl reflection $s_\alpha$ on $\text{Sym}^{3}V$ that is induced by its action on $V=\mf{h}^\ast$,  and $p,q\in \mathbb{R}$ are unknowns to be determined. 
It may seem natural to assume $f(\alpha)$ to be of rank $1$, given that there is only a single copy of the sign representation.
The above assumption on the shape of $f(\alpha)$ is even more restrictive and can only be justified afterwards: the above ansatz will suffice to solve (\ref{eq: full master equation}) and replacing one of the $v(\alpha)$ above by an independent vector $w(\alpha)$ would lead to the requirement $v(\alpha)=w(\alpha)$ after a tedious and lengthy computation.
There are only two vectors in $\text{Sym}^3(V)$ which are related to $\alpha$ in a natural way: $\alpha\alpha\alpha$
and $\psi\left(\alpha\right)$, which motivates (\ref{eq:ansatz for v(alpha) in 7/2 spin}).

\begin{lemma}
	\label{lem:scalar products of natural elements in Sym^3V}
	Let $A$ be simply-laced, $m\coloneqq\dim V$ and  $\alpha,\beta\in\Delta^{re}(A)$. Then
	\[
	s_{\alpha}\left(\psi(\beta)\right)=\psi\left(s_{\alpha}\beta\right),
	\]
	\[
	b\left(\psi\left(\alpha\right),\psi\left(\beta\right)\right)=\frac{m+2}{12}b\left(\alpha,\beta\right)\ ,\ b\left(\alpha\alpha\alpha,\psi\left(\beta\right)\right)=b\left(\alpha,\beta\right)\ .
	\]
\end{lemma}

\begin{proof}
	The first statement reduces to checking that $\left(\sum_{k,l=1}^{m}\omega^{kl}e_{k}\otimes e_{l}\right)\in Sym^2(V)$ is invariant under action of $s_{\alpha}$, which is a result of $b$ being $W(A)$-invariant.
	Therefore $s_{\alpha}$	intertwines with $\psi$, i.e. $s_{\alpha}\circ\psi=\psi\circ s_{\alpha}$.
	One computes directly: 
	\begin{align*}
		b\left(\psi(\alpha),v_{1}v_{2}v_{3}\right) &= \frac{1}{6}\left[b\left(\alpha,v_{1}\right)b\left(v_{2},v_{3}\right)+b\left(\alpha,v_{2}\right)b\left(v_{1},v_{3}\right)+b\left(\alpha,v_{3}\right)b\left(v_{1},v_{2}\right)\right]\,.
	\end{align*}
	For $v_{1}v_{2}v_{3}=\beta\beta\beta$ this specializes to 
	\[
	b\left(\psi(\alpha),\beta\beta\beta\right)=\frac{1}{2}b\left(\alpha,\beta\right)b\left(\beta,\beta\right)=b\left(\alpha,\beta\right),
	\]
	since $b\left(\beta,\beta\right)=2$ as $\beta\in\Delta^{re}(A)$ and $A$ is simply-laced. 
	The computation of 
	\begin{align*}
		b\left(\psi\left(\alpha\right),\psi\left(\beta\right)\right) &=  \frac{m+2}{12}b\left(\alpha,\beta\right),
	\end{align*}
	is straightforward but a bit lengthy and one uses the identity 
	\[
	\sum_{k,l,i,j}\omega^{kl}\omega^{ij}\omega_{ki}\omega_{lj}=\sum_{k,i,j}\delta_{\,j}^{k}\omega^{ij}\omega_{ki}=\sum_{k,i}\omega^{ik}\omega_{ki}=\sum_{i}\delta_{\,i}^{i}=m.
	\]
\end{proof}

\begin{theorem}[This is \cite{Lautenbacher22}, thm. 3.23 and lem. 3.24]
	\label{thm:7/2 spin rep}
	Let $A$ be a simply-laced GCM,  let $\rho$	be a $\frac{1}{2}$-spin representation of $\mathfrak{k}(A)$
	as in theorem \ref{thm:Properties of spin rep's image} and denote by $X_1,\dots,X_n$ the Berman generators of $\mathfrak{k}(A)$. 
	The ansatz (\ref{eq:ansatz 7/2}) satisfies (\ref{eq:trivial part of master equation})
	and (\ref{eq:master equation}) if one fixes $p$ and $q$ to be ($\varepsilon=\pm1$,
	$m\coloneqq\dim\mathfrak{h}$ and the signs can be chosen independently)
	\begin{equation}
		q_{\pm,\varepsilon}=-\varepsilon\frac{12\mp2\sqrt{6(m+8)}}{(m+2)\sqrt{3}},\ p_{\varepsilon}=\varepsilon\frac{1}{\sqrt{3}}\label{eq:q in 7/2}\:.
	\end{equation}
	One has further that $f(\alpha)$ in (\ref{eq:ansatz 7/2}) satisfies
	\begin{equation}
		f\left(\alpha\right)^{2}=4f\left(\alpha\right)\ \forall\,\alpha\in\Delta^{re}.\label{eq:eigenvalue of f(alpha)}
	\end{equation}
	With $\tau$ as in (\ref{eq:ansatz 7/2}),
	\begin{align*}
	\sigma\left(X_{i}\right)&=\tau\left(\alpha_{i}\right)\otimes2\rho\left(X_{i}\right)\ \forall\,i=1,\dots,n
	\end{align*}
	extends to a representation of $\mathfrak{k}(A)$.	
\end{theorem}
\begin{remark}
	The above theorem in fact describes two representations as there are two signs to chose but only the relative sign matters, because $v(\alpha)$ appears twice in $f(\alpha)$. In the remainder we pick both signs as $+$ when referring to \emph{the} higher spin representation $\mathcal{S}_{\frac{7}{2}}$. 
	The preferred $\mf{k}(A)$-module is denoted by $\mathcal{S}_{\frac{7}{2}}$ in the remainder of the text.
	\end{remark}
\begin{proof}
Plugging the ansatz (\ref{eq:ansatz 7/2}) into the l.h.s. of (\ref{eq:trivial part of master equation})
yields
\begin{align*}
	\left[\tau(\alpha),\tau(\beta)\right] & =   \left[s_{\alpha},s_{\beta}\right]+\left[s_{\alpha},f(\beta)\right]+\left[f(\alpha),s_{\beta}\right]
	   +\left[f(\alpha),f(\beta)\right]
\end{align*}
and is required to vanish for $b\left(\alpha,\beta\right)=0$. 
For $\alpha,\beta\in\Delta^{re}$ s.t. $b\left(\alpha,\beta\right)=0$ one has $s_{\beta}\left(v(\alpha)\right)=v(\alpha)$ and $b\left(v\left(\alpha\right),v\left(\beta\right)\right)=0$ by lemma \ref{lem:scalar products of natural elements in Sym^3V}.
Thus, $\left[s_{\alpha},f(\beta)\right]=0$ by invariance of $b$ under $W(A)$ and $\left[f(\alpha),f(\beta)\right]$ because of $b\left(v\left(\alpha\right),v\left(\beta\right)\right)=0$.
Since $s_{\alpha}$ and $s_{\beta}$ commute, (\ref{eq:trivial part of master equation}) is satisfied.

For $b\left(\alpha,\beta\right)=\mp1$, plugging the ansatz (\ref{eq:ansatz 7/2}) into  (\ref{eq:master equation})
yields 
\begin{align*}	
	 \left\{ s_{\alpha}-\frac{1}{2}Id,s_{\beta}-\frac{1}{2}Id\right\} +\left\{ s_{\alpha},f(\beta)\right\} & +\left\{ f(\alpha),s_{\beta}\right\} +\left\{ f(\alpha),f(\beta)\right\} \\
	    -f(\alpha)-f(\beta)
	&\overset{!}{=} s_{\alpha\pm\beta}-\frac{1}{2}Id+f\left(\alpha\pm\beta\right)\ .
\end{align*}
This equation is satisfied if and only if
\begin{subequations}
	\begin{align}
		\left\{ s_{\alpha}-\frac{1}{2}Id,s_{\beta}-\frac{1}{2}Id\right\} +T\left(\alpha,\beta\right)=s_{\alpha\pm\beta}-\frac{1}{2}Id,\ \text{with}\label{eq:AC in general case}
	\end{align}
	\begin{align}
		T\left(\alpha,\beta\right)&=\left\{ s_{\alpha},f(\beta)\right\} +\left\{ f(\alpha),s_{\beta}\right\} +\left\{ f(\alpha),f(\beta)\right\} \label{eq:the relevant part of AC} \\
		&\ -f(\alpha)-f(\beta)-f\left(\alpha\pm\beta\right)\ . \nonumber
	\end{align}
\end{subequations}
One has by prop. \ref{prop: master eq. holds away from sign reps} that
$\left\{ s_{\alpha}-\frac{1}{2}Id,s_{\beta}-\frac{1}{2}Id\right\} =s_{\alpha\pm\beta}-\frac{1}{2}Id$
holds on all $\mathfrak{S}_{3}=\langle s_{\alpha},s_{\beta}\rangle$-representations whose decomposition does not  contain the sign representation.
Thus, the support of $T\left(\alpha,\beta\right)$ needs to be the span of $\alpha\beta(\alpha+\beta)$, as this is the only occurring sign representation of $\mathfrak{S}_{3}=\langle s_{\alpha},s_{\beta}\rangle$ by lemma \ref{lem:One sign rep on Sym^3V}, because (\ref{eq:AC in general case}) yields $T(\alpha,\beta)u=0$ for all $u$ in a $\mathfrak{S}_{3}$-representation that is not the sign-representation.
Furthermore, the image of $T(\alpha,\beta)$ must also be the span of $\alpha\beta(\alpha+\beta)$, because the rest of (\ref{eq:AC in general case}) acts diagonally on $\mathbb{R}\alpha\beta(\alpha+\beta)$.
In conclusion, support and image of $T\left(\alpha,\beta\right)$ need to be the span of $\alpha\beta(\alpha+\beta)$.

One computes for $b\left(\alpha,\beta\right)=\mp 1$:
\begin{align*}
s_{\alpha}v(\beta)&=-b\left(\alpha,\beta\right)\cdot v\left(\alpha\pm\beta\right),
\end{align*}
\begin{align*}
	v(\beta)\cdot b\left(v(\beta),s_{\alpha}\left(u_{1}u_{2}u_{3}\right)\right) &= -b\left(\alpha,\beta\right)\cdot v(\beta)\cdot b\left(v\left(\alpha\pm\beta\right),u_{1}u_{2}u_{3}\right)
\end{align*}
for all $u_{1}u_{2}u_{3}\in\text{Sym}^{3}V$.
Thus,
\begin{align*}
	\left\{ s_{\alpha},f(\beta)\right\} +\left\{ f(\alpha),s_{\beta}\right\}  & =  v\left(\alpha\right) b\left(v\left(\alpha\pm\beta\right),\cdot\right)+v\left(\alpha\pm\beta\right) b\left(v\left(\alpha\right),\cdot\right)\\
	   &-b\left(\alpha,\beta\right) v(\beta) b\left(v\left(\alpha\pm\beta\right),\cdot\right) -b\left(\alpha,\beta\right) v\left(\alpha\pm\beta\right) b\left(v\left(\beta\right),\cdot\right)\ .
\end{align*}
Setting $X(\alpha,\beta):=b\left(v\left(\alpha\right),v\left(\beta\right)\right)$ one has furthermore
\begin{align*}
\left\{ f(\alpha),f(\beta)\right\} =X\left(\alpha,\beta\right)\left[v\left(\alpha\right) b\left(v\left(\beta\right),\cdot\right)+v\left(\beta\right) b\left(v\left(\alpha\right),\cdot\right)\right]\ .
\end{align*}
With this, one computes
\begin{eqnarray*}
	T\left(\alpha,\beta\right) 
	& = & v\left(\alpha\right) b\left(v\left(\alpha\pm\beta\right)+X\left(\alpha,\beta\right)v\left(\beta\right)-v\left(\alpha\right),\,\cdot\right)\\
	&  & +v(\beta) b\left(X\left(\alpha,\beta\right)v\left(\alpha\right)-b\left(\alpha,\beta\right)v\left(\alpha\pm\beta\right)-v\left(\beta\right),\,\cdot\right)\\
	&  & +v\left(\alpha\pm\beta\right) b\left(v\left(\alpha\right)-b\left(\alpha,\beta\right)\cdot v\left(\beta\right)-v\left(\alpha\pm\beta\right),\,\cdot\right).
\end{eqnarray*}
Now the demand that $T\left(\alpha,\beta\right)$ may only be supported
on the span of $V_{\alpha,\beta}:=\alpha\cdot\beta\cdot\left(\alpha\pm\beta\right)\in\text{Sym}^{3}V$ leads, after rearrangements of order and signs,  to three equations ($k_1,k_2,k_3\in\mathbb{R}$):
\begin{eqnarray*}
	v\left(\alpha\pm\beta\right)-v\left(\alpha\right)+X\left(\alpha,\beta\right)v\left(\beta\right) & = & k_{1}\cdot V_{\alpha,\beta}\:,\\
	-b\left(\alpha,\beta\right)\left[v\left(\alpha\pm\beta\right)-b\left(\alpha,\beta\right)^{-1}X\left(\alpha,\beta\right)v\left(\alpha\right)+b\left(\alpha,\beta\right)^{-1}v\left(\beta\right)\right] & = & k_{2}\cdot V_{\alpha,\beta}\:,\\
	-\left[v\left(\alpha\pm\beta\right)-v\left(\alpha\right)+b\left(\alpha,\beta\right)\cdot v\left(\beta\right)\right] & = & k_{3}\cdot V_{\alpha,\beta}\:.
\end{eqnarray*}
By the definition of $v(\alpha)$ and linearity of $\psi$ one has
\begin{eqnarray*}
	v\left(\alpha\pm\beta\right)-v\left(\alpha\right)+\underset{=\mp1}{\underbrace{b\left(\alpha,\beta\right)}}\cdot v\left(\beta\right) & = -3p\,b\left(\alpha,\ensuremath{\beta}\right) V_{\alpha,\beta}.
\end{eqnarray*}
as $v(\beta)$ is linearly independent from $V_{\alpha,\beta}$, all three equations concerning the support are satisfied if $X\left(\alpha,\beta\right)=b\left(\alpha,\beta\right)$.
Miraculously this satisfies the constraint concerning the image as well, since $X\left(\alpha,\beta\right)=b\left(\alpha,\beta\right)$ implies
\begin{align*}
	T\left(\alpha,\beta\right) & =  -9p^{2}\, V_{\alpha,\beta} b\left(V_{\alpha,\beta},\,\cdot\right)\:.
\end{align*}
By use of lemma \ref{lem:scalar products of natural elements in Sym^3V} one computes
\begin{align}
	b\left(v\left(\alpha\right),v\left(\beta\right)\right) &= p^{2}\,b\left(\alpha,\beta\right)^{3}+2pq\,b\left(\alpha,\beta\right)+q^{2}\frac{m+2}{12}\,b\left(\alpha,\beta\right), \label{eq:product of v(a) with v(b)}
\end{align}
so that $X\left(\alpha,\beta\right)=b\left(v\left(\alpha\right),v\left(\beta\right)\right)=b\left(\alpha,\beta\right)$ is equivalent to 
\begin{equation}
	p^{2}+2pq+\frac{m+2}{12}q^{2}=1\ .\label{eq:q and p in 7/2 spin part 1}
\end{equation}
This determines $p$ and $q$ in the ansatz for $v\left(\alpha\right)$ such that $T\left(\alpha,\beta\right)$ has correct support and image. 
Evaluation of the l.h.s. of (\ref{eq:AC in general case}) on $V_{\alpha,\beta}$ fixes the scale of $v(\alpha)$ (note that $s_\alpha,s_\beta,s_{\alpha\pm \beta}$ act on $V_{\alpha,\beta}$ as $-Id$):
\begin{align}
	\left[\left\{ s_{\alpha}-\frac{1}{2}Id,s_{\beta}-\frac{1}{2}Id\right\} +T\left(\alpha,\beta\right)\right]V_{\alpha,\beta} & =  \left[s_{\alpha\pm\beta}-\frac{1}{2}Id\right]V_{\alpha,\beta}\nonumber \\
	\Leftrightarrow\ p^{2} & =  \frac{1}{3}\nonumber 
\end{align}
Together with (\ref{eq:q and p in 7/2 spin part 1}) one has with $\varepsilon=\pm 1$:
\begin{subequations} \label{eq: p and q in 7/2}
	\begin{align}
		p_{\varepsilon} & =  \varepsilon\frac{1}{\sqrt{3}}\:,\label{eq:p in 7/2}\\
		q_{\pm,\varepsilon} &=-\varepsilon\frac{12\mp2\sqrt{6(m+8)}}{(m+2)\sqrt{3}}.\label{eq:q in 7/2 in proof}
	\end{align}
\end{subequations}
Thus, by props. \ref{prop:criterion for homomorphism} and \ref{prop: master eq. holds away from sign reps}, one concludes that 
\begin{align*}
	\sigma\left(X_{i}\right)&=\tau\left(\alpha_{i}\right)\otimes2\rho\left(X_{i}\right)\ \forall\,i=1,\dots,n
\end{align*}
with $\tau$ as in (\ref{eq:ansatz 7/2}) and $p,q$ as in (\ref{eq: p and q in 7/2}) indeed extends to a representation of $\mf{k}(A)$.

Concerning $f(\alpha)^2$ one has  with (\ref{eq:product of v(a) with v(b)}) that 
\begin{align*}
	b\left(v\left(\alpha\right),v\left(\alpha\right)\right) &= 8p^{2}+4pq+q^{2}\frac{m+2}{6} 
	= \frac{8}{3}+\frac{24+4m+32-48}{3(m+2)}=4\:
\end{align*}
and therefore 
\begin{align*}
	f\left(\alpha\right)^{2}&=v\left(\alpha\right)\cdot b\left(v\left(\alpha\right)\vert v\left(\alpha\right)\right)\cdot b\left(v\left(\alpha\right)\vert\cdot\right)=4\cdot f\left(\alpha\right)
\end{align*}
\end{proof}

\begin{remark}
	As of now, all attempts at extending the ansatz (\ref{eq:ansatz 7/2})
	to higher powers $\text{Sym}^{n}\left(\mathfrak{h}^{\ast}\right)$
	or other Schur modules $\mathcal{S}_{\lambda}\left(\mathfrak{h}^{\ast}\right)$ were unsuccessful. 
	Typically, the number of sign representations exceeds the number of free parameters that is determined by the elements naturally associated to a root $\alpha$, which we tested for all Schur modules associated to partitions of $n$ up to $n=5$. 
	For almost all Schur modules over $\mf{h}^\ast$ the number of sign representations depends on the dimension of $\mathfrak{h}^{\ast}$.
	In that regard, $\mathcal{S}_{\frac{7}{2}}$ is somewhat special because the occurrence of exactly one sign representation is universal in that it does not depend on $\mathfrak{h}^{\ast}$. 
	
\end{remark}

\section{Reducibility and Irreducibility}\label{sec:Irreducibility}
In this section we derive some results concerning the irreducibility of  higher spin representations and properties of their tensor products.
These results have in full generality only appeared in \cite{Lautenbacher22} and are genuinely new  except for choice examples such as $\mathcal{S}_{\frac{3}{2}}$ of $\mf{k}(E_{10})$ where the image is known to be isomorphic to $\mf{so}(32,288)$ by \cite{Higher spin realizations}.

On the technical level the most important tool will be that the action factorizes according to the tensor product structure:
\begin{lemma} \textnormal{(This is the corrected\footnote{In \cite[5.6]{Lautenbacher22}, the r.h.s. of (\ref{eq:quartic identity relation}) misses an overall minus sign. } version of \cite[5.6]{Lautenbacher22})}\\
	\label{lem:Split of the k(A)-action}
	Let $\left(\sigma,V\otimes S\right)$ denote the representation $\mathcal{S}_{\frac{3}{2}}$ or $\mathcal{S}_{\frac{5}{2}}$ of 	$\mf{k}\coloneqq\mathfrak{k}\left(A\right)\left(\mathbb{K}\right)$ for $A$ simply-laced and let $X_1,\dots,X_n$ denote the Berman generators of $\mf{k}$.
	Then one has for all $i=1,\dots,n$ that
	\begin{subequations}\label{eq:split action}
		\begin{align}
			Id\otimes\rho\left(X_{i}\right) &=\frac{2}{3}\sigma\left(X_{i}\right)^{3}+\frac{7}{6}\sigma\left(X_{i}\right)\:, \label{eq:split action a}\\ \eta\left(s_{i}\right)\otimes Id&=-\frac{20}{9}\sigma\left(X_{i}\right)^{4}-\frac{41}{9}\sigma\left(X_{i}\right)^{2}\:,\label{eq:split action b} \\
			Id\otimes Id &= -\frac{16}{9}\sigma \left(X_{i}\right)^{4} - \frac{40}{9} \sigma\left(X_{i}\right)^{2}\:, \label{eq:quartic identity relation}
		\end{align}
	\end{subequations} 
	where $\eta$ denotes the representation of the Weyl group $W\left(A\right)$	on $V$. 
	Hence, for all $w\in W(A)$ there exists $y_{1}\in\mathcal{U}\left(\mathfrak{k}\right)$	s.t. $\sigma\left(y_{1}\right)=\eta(w)\otimes Id$ and for all $x\in\mathfrak{k}$	there exists $y_{2}\in\mathcal{U}\left(\mathfrak{k}\right)$ s.t.	$\sigma\left(y_{2}\right)=Id\otimes\rho\left(x\right)$.
\end{lemma}
\begin{remark}
	The above formulas also hold for $\alpha\in\Delta_{+}^{re}$ and $X_{\alpha}\in\mathfrak{k}_{\alpha}$ s.t. $\left(X_{\alpha}\vert X_{\alpha}\right)=\left(X_{i}\vert X_{i}\right)$ for some $i=1,\dots,n$, due to the conjugation result \ref{prop:Conjjugation lemma for Sn2} that we are going to prove later.
\end{remark}
\begin{proof}
	One has that 
	\begin{align}
	\sigma\left(X_{i}\right) &=\left(\eta\left(s_{i}\right) -\frac{1}{2}\right)\otimes\left(2\rho\left(X_{i}\right)\right),
	\end{align}
	and from this one computes
	\begin{align*}
		\sigma\left(X_{i}\right)^{2} &= \left(\eta\left(s_{i}\right)-\frac{5}{4}\right)\otimes Id\:, \\
		\sigma\left(X_{i}\right)^{3} &= \left( -\frac{7}{4}\eta\left(s_{i}\right)+\frac{13}{8} \right)\otimes 2\rho(X_i) = -\frac{7}{4}\sigma\left(X_{i}\right) +\frac{3}{4}Id\otimes2\rho\left(X_{i}\right)\:, \\
		\sigma\left(X_{i}\right)^{4} &= -\frac{5}{2}\left(\eta\left(s_{i}\right) -\frac{41}{40}\right)\otimes Id\:,
	\end{align*}
	
	from which (\ref{eq:split action}) follows. 
	As $W(A)$ is generated by the $s_{i}$ and $\mathfrak{k}$ is generated by the $X_{i}$
	it is always possible to find $y_{1},y_{2}\in\mathcal{U}\left(\mathfrak{k}\right)$
	s.t. $\sigma\left(y_{1}\right)=\eta(w)\otimes Id$ and $\sigma\left(y_{2}\right)=Id\otimes\rho\left(x\right)$. \newline
	Concerning the remark: By prop. \ref{prop:Conjjugation lemma for Sn2} one has that 
	\begin{align*}
		\sigma\left(X_{\alpha}\right) &= \pm\left(\eta\left(s_{\alpha}\right) -\frac{1}{2}\right)\otimes\left(2\rho\left(X_{\alpha}\right)\right)
	\end{align*}
	for $\alpha\in\Delta^{re}$ and the sign $\pm$ does not change the computation.
\end{proof}
We need to collect a few basic results concerning the action of $W(A)$ on $\mf{h}^\ast$:
\begin{lemma}
	\label{lem:h* is irreducible W(A)-module for regular A}
	Let $A\in\mathbb{Z}^{n\times n}$ be an indecomposable, regular, not necessarily simply-laced GCM and let $(\mathfrak{h}, \Pi, \Pi^\vee)$ a realization of $A$. Then $\mathfrak{h}^{\ast}$
	is an irreducible $W(A)$-module.
\end{lemma}
\begin{proof}
	Since $A$ is regular, $\mathfrak{h}^{\ast}$ is spanned by $\alpha_1,\dots,\alpha_n$ which are by definition linearly independent and therefore to any $0\neq \lambda\in U$ in an invariant submodule $U\subset\mf{h}^\ast$ there exists $i\in\left\{ 1,\dots,n\right\} $ such that $\lambda\left(\alpha_{i}^\vee\right)\neq0$.	
	But then $s_{i}\lambda=\lambda-\lambda\left(\alpha_{i}^\vee\right)\alpha_{i}\in U$ implies $\alpha_{i}\in U$.
	Now $s_{j}(\alpha_i)-\alpha_i = a_{ji}\alpha_j \in U$ shows by successive application that all simple roots are contained in $U$ because $A$ is indecomposable.
	Thus, $\mathfrak{h}^{\ast}$	is an irreducible $W(A)$-module.
\end{proof}	
We remark that the assumption of regularity is necessary. 
If $A$ is an affine GCM, then $\mathbb{K}\cdot\delta$, where $\delta$ denotes the null root, is an invariant subspace of $\mathfrak{h}^{\ast}$ w.r.t. the action of $W(A)$.
Also, indecomposability is necessary because otherwise the span of simple roots corresponding to each sub-block of $A$ are an invariant subspace of $\mf{h}^\ast$.

\begin{proposition}[This is \cite{Lautenbacher22}, prop. 5.8]	\label{prop:S32 irrreducible} 
	Let $A\in\mathbb{Z}^{n\times n}$ be	an indecomposable, simply-laced, regular  GCM  and $n\geq2$. 
	Then there exists an irreducible generalized spin representation $\mathcal{S}_{\frac{1}{2}}$ of $\mathfrak{k}\left(A\right)$.
	Furthermore, the  representation $\mathcal{S}_{\frac{3}{2}}$ built on this  $\mathcal{S}_{\frac{1}{2}}$ is irreducible as well.
\end{proposition}
\begin{proof}
	Since the generalized spin representations from thm. \ref{thm:Properties of spin rep's image} are finite-dimensional, restriction to a smallest invariant sub-module is always possible and yields an irreducible generalized spin representation $\mathcal{S}_{\frac{1}{2}}$.
	
	By lemma \ref{lem:h* is irreducible W(A)-module for regular A} $\mathfrak{h}^{\ast}$ is an irreducible $W\left(A\right)$-module because $A$ is regular and by lemma \ref{lem:Split of the k(A)-action} one can act independently on the factors of $\mathcal{S}_{\frac{3}{2} }= \mf{h}^\ast \otimes \mathcal{S}_{\frac{1}{2} }$.
	Irreducibility of $\mathcal{S}_{\frac{3}{2} }$ follows, if one can show that any invariant submodule contains an elementary tensor $\lambda\otimes u \in \mathcal{S}_{\frac{3}{2} }$.
	In order to do so, introduce the projections $p_i\coloneqq \frac{1}{2}(Id-s_i)\in End(\mf{h}^\ast)$, where $i=1,\dots,n$, to the simple roots.
	Since $(\cdot,\cdot)$ is non-degenerate $(\alpha_i,\lambda)=0$ for all $i=1,\dots,n$ and $\lambda\in\mf{h}^\ast$ is equivalent to $\lambda=0$ and in consequence to any $\lambda\neq0$ there exists $i$ such that $p_i(\lambda)\neq0$.
	Thus, to $0\neq v\in \mf{h}^\ast \otimes \mathcal{S}_{\frac{1}{2}}$ there exists $i$ such that $(p_i\otimes Id)(v)\neq 0$.
	But $(p_i\otimes Id)(v)= \alpha_i \otimes u$ for some $0\neq u\in \mathcal{S}_{\frac{1}{2}}$.
	This shows that any invariant submodule of $\mathcal{S}_{\frac{3}{2}}$ contains an elementary tensor and  therefore  that $\mathcal{S}_{\frac{3}{2}}$ is irreducible.
\end{proof}

\begin{corollary}\label{cor: image is semi-simple}
		Let $A\in\mathbb{Z}^{n\times n}$ and $\mathcal{S}_{\frac{3}{2}}$ be as in prop. \ref{prop:S32 irrreducible} and denote the representation by $\sigma:\mf{k}(A)\rightarrow \text{End}(\mathcal{S}_{\frac{3}{2}})$.
		Then $\mathrm{im}\,(\sigma)$ is semi-simple.
\end{corollary}
\begin{proof}
	According to \cite[I.6.4 prop. 5]{Bourbaki75}, one of 7 equivalent characterizations for a Lie algebra $\mf{g}$ to be reductive, is to possess a faithful, finite-dimensional, completely reducible\footnote{The exact wording is \emph{semi-simple}, which is used synonymously for \emph{completely reducible} in \cite{Bourbaki75} when applied to $\mf{g}$-modules (see \cite[I.3.1, def. 2]{Bourbaki75}). Note also, that  the reference to I.6.4 prop. 5 is sensitive to publisher and edition.} representation.
	By definition, $\text{im}\,(\sigma)\subset \text{End}(\mathcal{S}_{\frac{3}{2}})$ possesses a faithful finite-dimensional representation, because $\mathcal{S}_{\frac{3}{2}}$ is a finite-dimensional vector space.
	By prop. \ref{prop:S32 irrreducible}, $\mathcal{S}_{\frac{3}{2}}$ is an irreducible $\mf{k}(A)$-module and therefore an irreducible, hence completely reducible, $\text{im}\,(\sigma)$-module.
	Thus, $\text{im}\,(\sigma)$ is reductive.
	Since $A$ is indecomposable and simply-laced, $\mf{k}(A)$ is perfect, i.e., $[\mf{k}(A), \mf{k}(A)]=\mf{k}(A)$ (This is essentially a consequence of the fact that each Berman generator $X_j$ can be written as the commutator $-[X_i, [X_i,X_j]]$ for $i$ adjacent to $j$).
	One concludes that $\text{im}\,(\sigma)$ is semi-simple.
\end{proof}

\begin{lemma} 
	\label{lem:skew-adjointness of rep. matrices}
	Let $A$ be a simply-laced GCM. Then the higher spin representations $\left(\sigma,\mathcal{S}_{\frac{2s+1}{2}}\right)$	of $\mathfrak{k}\left(A\right)\left(\mathbb{R}\right)$ for $s\in\left\{1,2,3\right\} $ admit a non-degenerate bilinear	form $\left\langle \cdot,\cdot\right\rangle $ with respect to which the representation matrices are skew-adjoint.
	This so-called contravariant bilinear form is given explicitly by 
	\begin{align}
		 \left\langle a\otimes s,b\otimes t\right\rangle =\left(a\vert b\right)_{V}\cdot\left(s\vert t\right)_{S}\ \forall\,a,b\in V,\ s,t\in S\:, \label{eq: contravariant form on Sn2} 
	\end{align}
	with $V\in \{\mf{h}^\ast , \text{Sym}^2(\mf{h}^\ast), \text{Sym}^3(\mf{h}^\ast)\}$ and $S$ a generalized spin representation.
	Here, $(\cdot\vert\cdot)_V$ denotes the symmetric bilinear form invariant under the action of $W(A)$ induced by the standard invariant form on $\mf{h}^\ast$ and $(\cdot\vert\cdot)_S$ denotes an inner product on $S$.
\end{lemma}
\begin{proof}
	The generalized spin representation $S$	is compact according to prop. \ref{prop:properties of gen spin reps in simply laced case}.
	Hence, $S$ admits an inner product $\left(\cdot\vert\cdot\right)_{S}$ w.r.t. which the $\rho\left(X_{i}\right)$ are skew-adjoint. 
	Furthermore $V\in \{\mf{h}^\ast , \text{Sym}^2(\mf{h}^\ast), \text{Sym}^3(\mf{h}^\ast)\}$ carries the invariant bilinear form induced by $\left(\cdot\vert\cdot\right)$ on $\mf{h}^\ast$.
	The (induced) Weyl reflection $s_\alpha$ as well as the projection $f(\alpha)$ from \ref{eq:ansatz 7/2} for $\alpha\in\Pi$ are symmetric w.r.t. $\left(\cdot\vert\cdot\right)_{V}$.
	The bilinear form  (\ref{eq: contravariant form on Sn2}) is non-degenerate because $\left(\cdot\vert\cdot\right)_{V}$ and the inner product on $S$ are.
	Skew-adjointness of $\sigma(X_i)$ extends to the image of $\sigma$ as it is generated as a Lie algebra by the $\sigma(X_i)$.
\end{proof}

We would like to remark that for the existence of a contravariant bilinear form, compactness of the underlying generalized spin representation appears to be crucial. 
One can also construct spin representations for other involutory subalgebras where the involution has a sign twist for some of the Berman generators which typically results in a non-compact representation (cp. \cite{West03, Keurentjes04}). 
In this case, a contravariant bilinear form needs not exist (cp. \cite[sec. 2.3]{Damour06b}).

\begin{proposition}[This is  \cite{Lautenbacher22}, prop. 5.10]
	\label{prop:trace-free S52}
	Let $A$ be an indecomposable, simply-laced, regular	GCM and let $\mathcal{S}_{\frac{5}{2}}$ be the representation from	thm. \ref{thm:S32 and S52} built on an irreducible  generalized spin representation.  $\mathcal{S}_{\frac{5}{2}}$ decomposes into an orthogonal sum of invariant submodules  
	\begin{align}
	\mathcal{S}_{\frac{5}{2}}\cong\widetilde{\mathcal{S}}_{\frac{5}{2}}\oplus\mathcal{S}_{\frac{1}{2}} \label{eq: universal S52-decomp}
	\end{align}
	w.r.t. the contravariant form (\ref{eq: contravariant form on Sn2}).
	The module $\widetilde{\mathcal{S}}_{\frac{5}{2}}$, called the	\emph{trace-free} part of $\mathcal{S}_{\frac{5}{2}}$, is irreducible	if the $W(A)$-module $\mathrm{Sym}^{2}\left(\mathfrak{h}^{\ast}\right)$ decomposes into exactly two irreducible factors, where one of them is always the	trivial representation.
\end{proposition}

\begin{proof}
	One again uses that the action	of $\mathfrak{k}\left(A\right)$ on $V\otimes S$ can be split into
	the action on $V=\text{Sym}^2(\mf{h}^\ast)$ and $S=\mathcal{S}_{\frac{1}{2}}$, where the action on $V$ is essentially that of $W(A)$ on $V$.
	The symmetric element 
	\begin{equation}
		\Psi:=\sum_{k,l}\omega^{kl}e_{k}\otimes e_{l}\label{eq:symmetric element}
	\end{equation}
	with $\omega$ from (\ref{eq:definition of omega and its inverse})	is $W(A)$-invariant as it is invariant under any $A\in End\left(V\right)$ that is induced by	a $g\in O\left(\mathfrak{h}^{\ast}\right)$	(cp. \cite[secs. 17.3 \& 19.5]{Fulton Harris}).
	Since $S$ admits an inner product and $\left(\Psi\vert\Psi\right)=\dim\mathfrak{h}^{\ast}$, one has that $\mathbb{K}\cdot\Psi\otimes S$ is anisotropic w.r.t. the contravariant form (\ref{eq: contravariant form on Sn2}).
	Thus, its  orthogonal complement is an invariant submodule which we call the trace-free part.
	
	In order to describe irreducibility of its complement in terms of the action of $W(A)$ we proceed as before and try to show that there exists an elementary tensor in $\widetilde{\mathcal{S}}_{\frac{5}{2}}$.
	It may be possible to derive this by the induced action of $W(A)$ on $\Psi^\perp$ similar to the proof of prop. \ref{prop:S32 irrreducible}, but this time it is in fact more economic to work with $S$ instead.
	One has from thm. \ref{thm:Properties of spin rep's image}	that $\text{im}\,(\rho)\subset End(S)$ is semi-simple. 
	The irreducible $\text{im}\,(\rho)$-module $S$ becomes a finite-dimensional highest weight module of some semi-simple complex Lie algebra $\mathring{\mf{g}}=\text{im}\,(\rho)_{\mathbb{C}}$ after complexification.
	This implies that $S$ decomposes into weight spaces w.r.t. to a Cartan subalgebra $\mathring{\mathfrak{h}}$ of $\mathring{\mathfrak{g}}$ whose triangular   decomposition we denote by $\mathring{\mathfrak{n}}_{-}\oplus\mathring{\mathfrak{h}}\oplus\mathring{\mathfrak{n}}_{+}$; we also denote the Chevalley generators of $\mathring{\mathfrak{n}}_{+}$ by $\mathring{e}_{1},\dots\mathring{e}_{m}$ and the weights of $S_\mathbb{C}$ by $P(S)$.
	Any invariant sub-module $U_\mathbb{C}$ of $\widetilde{\mathcal{S}}_{\frac{5}{2},\mathbb{C}}$ therefore admits a basis whose elements are linear combinations of elementary tensors $\alpha\beta\otimes s_{\lambda}^{j}$ with $\alpha,\beta\in\mf{h}^\ast$ and $s_\lambda^{j}$ for $j=1,\dots,\text{mult}(\lambda)$ weight vectors of weight $\lambda\in P(S)\subset\mathring{\mf{h}}^\ast$. 
	As $S$ is a highest weight module each $\lambda$ can be written uniquely as $\lambda=\Lambda-\sum_{i=1}^m k_i \mathring{\beta}_i$, where the $\mathring{\beta}_i$ denote the simple roots of $\mathring{\mf{h}}^\ast$ and the $k_i$ are non-negative.
	As the basis of $U_{\mathbb{C}}$ is finite, there exist weights $\lambda$ of maximal depth (defined as $\sum k_i$) $k$ such that $s_\lambda$ occurs in the basis of $U_{\mathbb{C}}$. 
	To each such $s_\lambda^{i}$ there exists an element $e_+(\lambda,i)=\mathring{e}_{i_1}\cdots \mathring{e}_{i_k}\in \mathcal{U}(\mathring{\mf{n}}_+)$ such that $e_+(\lambda,i)s_\lambda^i \in S_\Lambda$ and is nonzero.
	This is because to each nonzero $x\in S_\lambda$ there exists an $i=1,\dots,m$ such that $\mathring{e}_{i}x\neq0$ unless $\lambda=\Lambda$ by uniqueness of the highest weight vector.  
	The same applies to any nonzero linear combination of the $s_\lambda^{i}$. 
	Furthermore, the $e_+(\lambda,i)$ map any $s_{\mu}^j$ with $\mu\neq\lambda$ but of the same or smaller depth to $0$. 
	Thus, with $\{b_1,\dots,b_N\}$ a basis of $\text{Sym}^2(\mf{h}^\ast)$ and given
	\[
	u=\sum_{i=1}^{N}\sum_{\mu\in P(S)}\sum_{j=1}^{mult\left(\mu\right)}c_{ij}^{(\mu)}b_{i}\otimes s^{j}_{(\mu)}\in U_{\mathbb{C}},
	\]
	there exists $e_+(\lambda)$ such that
	\begin{align*}
	(Id\otimes e_+(\lambda))u &= \sum_{i=1}^{N}\sum_{j=1}^{mult\left(\lambda\right)}c_{ij}^{(\lambda)}b_{i}\otimes e_+(\lambda) s^{j}_{(\lambda)}
	= \sum_{i=1}^{N}\left( b_{i} \otimes \left[\sum_{j=1}^{mult\left(\lambda\right)}c_{ij}^{(\lambda)}e_+(\lambda) s^{j}_{(\lambda)}\right] \right)\\
	&= \sum_{i=1}^{N}b_{i} \otimes k_i s_\Lambda =  \left(\sum_{i=1}^{N}k_{i}b_{i}\right) \otimes s_\Lambda
	\end{align*}
	with at least one $k_i\neq 0$.
	Therefore, an invariant submodule $U_{\mathbb{C}}$ of $\widetilde{\mathcal{S}}_{\frac{5}{2},\mathbb{C}}$ always contains an elementary tensor. 
	As in the proof of prop. \ref{prop:S32 irrreducible}, one now deduces from lemma \ref{lem:Split of the k(A)-action} that $\widetilde{\mathcal{S}}_{\frac{5}{2},\mathbb{C}}$ is irreducible if the action of $W(A)$ on $\Psi^\perp \subset \text{Sym}^2(\mf{h}^\ast)$ is.
	This in turn implies irreducibility of the real module $\widetilde{\mathcal{S}}_{\frac{5}{2}}$ because if $U$ were an invariant submodule of $\widetilde{\mathcal{S}}_{\frac{5}{2}}$, then $U_{\mathbb{C}}=U+iU$ is an invariant submodule of $\widetilde{\mathcal{S}}_{\frac{5}{2},\mathbb{C}}$.

\end{proof}
\begin{remark}
	Note that $\Psi^\perp\subset\text{Sym}^{2}\left(\mathfrak{h}^{\ast}\right)$ can be irreducible or not; 
	an example for the latter case is $A=A_{n-1}$. In this case, $W\left(A_{n-1}\right)$ is isomorphic to the symmetric group $\mf{S}_n$ and $\mathfrak{h}^{\ast}$ is isomorphic to its standard representation.
	According to \cite[ex. 4.19]{Fulton Harris}, $\text{Sym}^{2}V\cong U\oplus V\oplus V_{(n-2,2)}$	where $U$ denotes the trivial representation, $V$ denotes the standard representation and $V_{(n-2,2)}$ is the irreducible representation associated to the partition $(n-2,2)$ of $n$. 
	The exceptional diagrams $E_{6}$, $E_{7}$ and $E_{8}$ provide an example for $\Psi^\perp$ being irreducible: 
	According to \cite[tbls. C.4-6]{Geck-Pfeiffer}, $W\left(E_{n}\right)$ for $n=6,7,8$ admits an irreducible character of degree $\begin{pmatrix}n+1\\
		2
	\end{pmatrix}-1$ that occurs\footnote{This can be seen from the value of $b_{\chi}$ in this table. For
		an irreducible character $\chi$ a value of $b_{\chi}=d$ means that	$\text{Sym}^{d}\left(V\right)$ is the smallest symmetric product of $V$ that affords $\chi$ as an irreducible component.} 
		in $\text{Sym}^{2}\left(V\right)$, where $V$ denotes the standard representation of $W\left(E_{n}\right)$ as before.
\end{remark}
We would like to show that the image of $\mf{k}(A)$ under the representation $(\sigma,\mathcal{S}_{\frac{5}{2}})$ is a semi-simple Lie algebra by showing that its image is reductive as we did in cor. \ref{cor: image is semi-simple} and then exploit perfectness of $\mf{k}(A)$.
Since even the module  $\widetilde{\mathcal{S}}_{\frac{5}{2}}$ is not always irreducible, we need to show more generally that the module is completely reducible. 
By the factorization of the actions this essentially requires to show that $\text{Sym}^2(V)$ is a completely reducible $W(A)$-module which is not obvious if $W(A)$ is not finite.
\begin{lemma}\label{lem: Sym2V completely reducible}
	Let $A$ be an indecomposable, simply-laced, regular	GCM and denote by $V=\text{Sym}^2(\mf{h}^\ast)$ the $W(A)$-module induced by the natural action of $W(A)$ on $\mf{h}^\ast$. 
	The module $V$ is completely reducible.
\end{lemma}
\begin{proof}
	The module $V$ comes with a natural $W(A)$-invariant bilinear form $(\cdot\vert\cdot)$. 
	In order to show complete reducibility, we need to show that any invariant submodule $U$ has an invariant complement.
	The submodule $U^\perp$ orthogonal to $U$ w.r.t. $(\cdot\vert\cdot)$ is a natural candidate by $W(A)$-invariance.
	$U^\perp$ is an invariant complement if and only if $U^\perp \cap U=\{0\}$.
	If $A$ is of finite type there is nothing to prove since $(\cdot\vert\cdot)$ is positive definite in this case.
	For $A$ indefinite however, this requires some consideration. We will show first that $\mathbb{K}\Psi$ with $\Psi$ from (\ref{eq:symmetric element}) is the only $W(A)$-invariant subspace of dimension $1$. 
	We then show that $s_i.u\neq u$ for $u\in U^\perp \cap U$ leads to a contradiction, so that $U^\perp \cap U$ can only be trivial because $\Psi$ is not isotropic.
	
	First of all,  $U^\perp \cap U$ is a $W(A)$-invariant submodule, too. 
	As a consequence, $(u\vert u)=0$ for all $u\in U^\perp \cap U$. 
	Consider a basis $\omega_1,\dots,\omega_n$ of $\mf{h}^\ast$ defined by $(\alpha_i \vert \omega_j)=\delta_{ij}$ (since $A$ is simply-laced, these are  the fundamental weights) and spell out $u=\sum_{i,j} c_{ij}\omega_i \omega_j$ and $\alpha_i = \sum_{j} B^{(i)}_j \omega_j$ .
	Now $s_iu=u$ for all $i$ is equivalent to
	\begin{align*}
		0&=s_iu-u = c_{ii}(\omega_i-\alpha_i)(\omega_i-\alpha_i) - c_{ii}\omega_i \omega_i +(\omega_i-\alpha_i)\sum_{j\neq i} c_{ij}\omega_j - \omega_i \sum_{j\neq i} c_{ij}\omega_j \\
		&= c_{ii}\alpha_i\alpha_i-2c_{ii}\alpha_i\omega_i -\alpha_i \sum_{j\neq i} c_{ij}\omega_j
		= \alpha_i \left[ c_{ii}\sum_{j} B^{(i)}_j \omega_j -2c_{ii}\omega_i-\sum_{j\neq i} c_{ij}\omega_j\right] 
	\end{align*}
	which is only zero if the second factor is. This leads to the system of equations
	\begin{align*}
		c_{ii} B^{(i)}_i=2c_{ii}\quad,\quad c_{ii} B^{(i)}_j= c_{ij} \ \forall\,j\neq i\:.
	\end{align*}
	The first equation is satisfied for all $i$ because $s_i \alpha_i =-\alpha_i$ yields that $B^{(i)}_i=2$ so that the second equation fixes all $c_{ij}$ for $j\neq i$.
	Thus, if all $c_{ii}$ were $0$, then $u=0$ so that we can from now on assume otherwise.
	Let $i$ be such that $c_{ii}\neq0$ and pick $j$ such that $\{i,j\}\in\mathcal{E}(A)$ which is possible since $A$ is indecomposable.
	One has $0\neq (\alpha_i \vert \alpha_j)=B^{(i)}_j (\omega_j \vert \alpha_j)= B^{(i)}_j$ but also $0\neq (\alpha_i \vert \alpha_j)=B^{(j)}_i (\alpha_i \vert \omega_i)= B^{(j)}_i$ so that $B^{(j)}_i = B^{(i)}_j$ for all $i,j$. 
	Since $c_{ij}=c_{ji}$ because $u\in Sym^2(\mf{h^\ast})$ one has $c_{ii} B^{(i)}_j = c_{ij} = c_{ji} = c_{jj} B^{(j)}_i$ and therefore $c_{jj}=c_{ii}$ for all $j$.
	Thus, the space of all $W(A)$-invariant vectors is of dimension $1$ and therefore equal to $\mathbb{K}\Psi$.
	We had already seen that the norm of $\Psi$ is nonzero so that $\Psi^\perp \cap \mathbb{K}\Psi=\{0\}$.
	
	Now assume $u\in \Psi^\perp\cap U\cap U^\perp$ were nonzero. 
	Then there exists $i$ such that $s_iu\neq u$.
	Spelling out $u$ in the slightly altered basis $\{\alpha_i\}\cup\{ \omega_j\,\vert\,j\neq i\}$ as  $u=c_{ii} \alpha_i \alpha_i +\alpha_i\sum_{j\neq i} c_{ij}\omega_j + \sum_{j\neq i \neq k} c_{jk} \omega_j \omega_k$ yields
	\begin{align*}
		s_iu-u &= -2\alpha_i \sum_{j\neq i} c_{ij}\omega_j\eqqcolon \alpha_i\beta \in \Psi^\perp\cap U\cap U^\perp\:.
	\end{align*}
	One has that $(\beta\vert\alpha_i)=0$ and so $\alpha_i\beta \in \Psi^\perp\cap U\cap U^\perp$ implies 
	\begin{align*}
		0&= \left( \alpha_i\beta \vert \alpha_i \beta \right)= \frac{1}{2} \left(\alpha_i\vert\alpha_i\right)\left(\beta\vert\beta\right)+ \frac{1}{2} \left(\alpha_i\vert\beta\right)\left(\beta\vert\alpha_i\right) = \left(\beta\vert\beta\right)\:.
	\end{align*}
Since $A$ is regular and $\beta$ is nonzero, there exists $j\neq i$ such that $(\alpha_j\vert\beta)\eqqcolon c\neq0$.\\
Case 1: $(\alpha_i\vert\alpha_j)=0$.\\
One has $-s_j(\alpha_i\beta)+\alpha_i\beta = c\alpha_i\alpha_j\in \Psi^\perp\cap U\cap U^\perp$, 
but $ \left(c\alpha_i\alpha_j\, \vert\,c\alpha_i\alpha_j\right)=\frac{1}{2}c^2(\alpha_i\vert\alpha_i)(\alpha_j\vert\alpha_j)\neq0$ contradicts this.\\
Case 2: $(\alpha_i\vert\alpha_j)=-1$.
\begin{align*}
	s_j(\alpha_i\beta) -\alpha_i\beta &= c\alpha_i\alpha_j +\beta\alpha_j+c\alpha_j\alpha_j \eqqcolon w \:,\\
	\left(w\vert w\right) &= \frac{5}{2}c^2-2c\cdot\frac{1}{2}c -4c^2+\frac{1}{2}c^2+4c^2+4c^2\\
	&= 6c^2\neq0\:,
\end{align*}
which is again a contradiction to $w\in \Psi^\perp\cap U\cap U^\perp$.
Therefore $U\cap U^\perp =\{0\}$  and invariance of $(\cdot\vert\cdot)$ show that $V$ is a completely reducible $W(A)$-module.
\end{proof}

\begin{corollary}\label{cor: image of S52 is semi-simple}
	Let $A$ and $\mathcal{S}_{\frac{5}{2}}$ be as in prop. \ref{prop:trace-free S52}, then the image of $\mf{k}(A)$ under this representation is semi-simple.
\end{corollary}
\begin{proof}
	If $\text{Sym}^{2}\left(\mathfrak{h}^{\ast}\right)$ is a completely reducible $W(A)$-module, then $\mathcal{S}_{\frac{5}{2}}$ is a completely reducible $\mf{k}(A)$-module. 
	This implies the claim by the same reasoning as in cor. \ref{cor: image is semi-simple}.
\end{proof}

\begin{lemma} 
	\label{lem:kernels of products either or}
	Let $\rho_{i}:\mathfrak{k}\left(A\right)\left(\mathbb{K}\right)\rightarrow\text{End}(V_{i})$
	for $i=1,2$ be f.d. representations and let $x\in\ker\left(\rho_{1}\otimes\rho_{2}\right)$.
	Then either $x\in\ker\left(\rho_{1}\right)\cap\ker\left(\rho_{2}\right)$
	or there exists $0\neq\lambda\in\mathbb{C}$ such that
	\begin{equation}
		\rho_{1}(x)=\lambda\cdot Id_{V_{1}}\ \text{and }\rho_{2}(x)=-\lambda\cdot Id_{V_{2}}.\label{eq:nontrivial condition}
	\end{equation}
\end{lemma}
\begin{proof}
	One has $\ker\rho_{1}\cap\ker\rho_{2}\subset\ker\left(\rho_{1}\otimes\rho_{2}\right)$ because of 
	\begin{align*}
		\left(\rho_{1}\otimes\rho_{2}\right)(x)(v\otimes w) &=  \left(\rho_{1}(x)v\right)\otimes w+v\otimes\left(\rho_{2}\left(x\right)w\right)\\
		&=  0\otimes w+v\otimes0=0\ \forall\,v\in V_{1},w\in V_{2}.
	\end{align*}	
	Now assume $x\in\ker\left(\rho_{1}\otimes\rho_{2}\right)$ but $x\notin\ker\left(\rho_{1}\right)$, then	
	$(\rho_{1}\otimes\rho_{2}) (x)$ needs to vanish everywhere and so in particular on every element of a basis of $V_1\otimes V_2$. 
	Take bases $\left\{ e_{1},\dots,e_{n}\right\} $	and $\left\{ f_{1},\dots,f_{m}\right\} $ of $V_{1}$, resp. $V_{2}$ and compute
	\begin{align*}
		\rho_{1}(x)e_{i}\otimes f_{j}+e_{i}\otimes\rho_{2}(x)f_{j}  &=  0\\
		\Leftrightarrow\ \sum_{k\neq i}^{n}\rho_{1}(x)_{ki}e_{k}\otimes f_{j}+\rho_{1}(x)_{ii}e_{i}\otimes f_{j} &+\sum_{l\neq j}^{m}\rho_{2}(x)_{lj}e_{i}\otimes f_{l}+\rho_{2}(x)_{jj}e_{i}\otimes f_{j} =0\\
	\end{align*}
	for all $1\leq i\leq n,\,1\leq j\leq m$.
	This holds if and only if $\rho_{1}(x)$ and $\rho_{2}(x)$ are diagonal and s.t.
	$\rho_{1}(x)_{ii}=-\rho_{2}(x)_{jj}\neq0$ for all $1\leq i\leq n,\,1\leq j\leq m$.
\end{proof}

\begin{lemma} 
	Let $\rho:\mathfrak{k}\left(A\right)\left(\mathbb{K}\right)\rightarrow\text{End}(V)$ be a finite-dimensional representation, then
	\begin{align}
	\ker\left(\rho\otimes\rho\right)=\ker\left(\rho\right),\ \ker\left(\text{Sym}^{n}(\rho)\right)=\ker\left(\rho\right),\ \ker\left(\wedge^{n}\rho\right)=\ker\rho,\ \label{eq: kernels of sym and antisym powers}
	\end{align}
	as long as $n<\dim(V)$.
\end{lemma}

\begin{proof}
	For $\rho\otimes\rho$ one applies lemma \ref{lem:kernels of products either or}.
	For $\text{Sym}^{n}(\rho)$	consider a basis $\left\{ b_{1},\dots,b_{m}\right\} $ of $V$, then $x.b_{i}\cdot b_{i}\cdot\dots\cdot b_{i}=0$ is equivalent to $\rho\left(x\right)_{ji}=0$ for all $i,j=1,\dots,m$.
	Similarly, one computes on $\bigwedge^{n}V$ that $x.b_{i_{1}}\wedge\dots\wedge b_{i_{n}}=0$ is equivalent to
	$\sum\limits{k=1}^{n}\rho\left(x\right)_{i_{k}i_{k}}=0$ and $\rho\left(x\right)_{ij}=0$
	for all $j\neq i$ (it may be instructional to look at $\bigwedge^{2}V$ first).
\end{proof}

\begin{proposition}[This is \cite{Lautenbacher22}, prop. 6.3]
	\label{prop:product kernels}
	Let $A$ be a simply-laced and indecomposable GCM.	
	Then all the higher spin representations $\left(\rho_{\frac{2s+1}{2}},\mathcal{S}_{\frac{2s+1}{2}}\right)$	of $\mathfrak{k}\left(A\right)$ for $s=1,2,3$ (cp. thms. \ref{thm:Properties of spin rep's image}, \ref{thm:S32 and S52} and \ref{thm:7/2 spin rep}) satisfy $\ker\rho_{\frac{2s_{1}+1}{2}}\otimes\rho_{\frac{2s_{2}+1}{2}}\cong\ker\rho_{\frac{2s_{1}+1}{2}}\cap\ker\rho_{\frac{2s_{2}+1}{2}}$.
\end{proposition}

\begin{proof}
	There exists a non-degenerate bilinear form on each of the mentioned modules and the action of $\mathfrak{k}\left(A\right)$
	is skew (cp. lemma \ref{lem:skew-adjointness of rep. matrices}) w.r.t. this bilinear form.
	The representation matrices must therefore be traceless which excludes the second case of lemma \ref{lem:kernels of products either or}.
\end{proof}

\begin{proposition} \label{prop: no kernel inclusion}
	Let $A$ be a regular, indefinite, simply-laced GCM and denote by $\mathcal{I}_{\frac{1}{2}}$, $\mathcal{I}_{\frac{3}{2}}$ and $\mathcal{I}_{\frac{5}{2}}$ the kernels of the representations $(\rho,\mathcal{S}_{\frac{1}{2}})$,  $(\sigma,\mathcal{S}_{\frac{3}{2}})$, and $(\sigma,\widetilde{\mathcal{S}}_{\frac{5}{2}})$  respectively. 
	If the image of $\sigma$ does not contain a compact ideal, then $\mathcal{I}_{\frac{1}{2}}\cap \mathcal{I}_{\frac{m}{2}} \subsetneq \mathcal{I}_{\frac{k}{2}}$ for $m\in\{3,5\}$ and $k\in\{1,m\}$.
\end{proposition}
\begin{proof}
	We start by showing $\mathcal{I}_{\frac{1}{2}} \not\subseteq \mathcal{I}_{\frac{m}{2}}$.
	Let $x_{\alpha}\in\mathfrak{k}_{\alpha}\coloneqq\left(\mathfrak{g}_{\alpha}\oplus\mathfrak{g}_{-\alpha}\right)\cap\mathfrak{k}$,
	$x_{\beta}\in\mathfrak{k}_{\beta}$ with $\alpha,\beta\in\Delta_{+}^{re}$
	be such that (recall the def. of $\Gamma$-matrices \ref{def:gen Gamma matrix} and their relation to generalized spin representations \ref{prop:Existence of gen Gamma matrices})
	\begin{align*}
		\rho\left(x_{\alpha}\right)&=\frac{1}{2}\Gamma(\alpha)\ , &\ \rho\left(x_{\beta}\right)&=\frac{1}{2}\Gamma(\beta)\:,\\
		\sigma\left(x_{\alpha}\right)&=\left(s_{\alpha}-\frac{1}{2}Id\right)\otimes\Gamma(\alpha)\ , &\ \sigma\left(x_{\beta}\right)&=\left(s_{\beta}-\frac{1}{2}Id\right)\otimes\Gamma(\beta)\:.
	\end{align*}
Such elements exist due to lemma \ref{prop:Conjugation lemma for S12} and prop. \ref{prop:Conjjugation lemma for Sn2}.
We can furthermore choose $\alpha,\beta\in\Delta^{re}_{+}$ such that $\alpha-\beta\in2Q$ which implies $\Gamma(\alpha)=\Gamma(\beta)$ and therefore $x_{\alpha}-x_{\beta}\in\mathcal{I}_{\frac{1}{2}}$.
However, one has $\sigma(x_\alpha)-\sigma(x_\beta)=\left(s_{\alpha}-s_{\beta}\right)\otimes\Gamma(\alpha)\neq0$. 

In order to show  $\mathcal{I}_{\frac{m}{2}} \not\subseteq \mathcal{I}_{\frac{1}{2}}$ we are going to exploit that the image of the generalized spin representation $\rho$ is compact but that $\mathcal{S}_{\frac{3}{2}}$ and $\widetilde{\mathcal{S}}_{\frac{5}{2}}$ have an invariant bilinear form of mixed signature.
We had excluded the case $\mathcal{I}_{\frac{m}{2}} = \mathcal{I}_{\frac{1}{2}}$ and therefore it is the same to assume $\mathcal{I}_{\frac{m}{2}} \subsetneq \mathcal{I}_{\frac{1}{2}}$ and produce a contradiction.
We know by corollaries \ref{cor: image is semi-simple} and \ref{cor: image of S52 is semi-simple} that $\text{im}\,\sigma$ is semi-simple. 
If $\mathcal{I}_{\frac{m}{2}} \subsetneq \mathcal{I}_{\frac{1}{2}}$, then there exists a nontrivial homomorphism $\phi: \text{im}\,\sigma \rightarrow \text{im}\,\rho$ that factors through $\rho$ and $\sigma$, i.e., $\phi\circ\sigma=\rho$. 
Since  $\text{im}\,\rho$ is compact, there exists an invariant inner product on $\mathcal{S}_{\frac{1}{2}}$ so that $\phi$ provides a nontrivial finite-dimensional unitary representation of $\text{im}\,\sigma$.
As  non-compact simple Lie-algebras do not admit nontrivial finite-dimensional unitary representations, all non-compact simple factors of $\text{im}\,\sigma$ must act trivially on $\mathcal{S}_{\frac{1}{2}}$. 
Thus, if  $\mathcal{I}_{\frac{m}{2}} \subsetneq \mathcal{I}_{\frac{1}{2}}$, then $\text{im}\,\sigma$ has a nontrivial semi-simple compact factor $\mf{k}_0$ which we excluded by hypothesis.
\end{proof}
\begin{remark}
	The hypothesis on $A$ to be indefinite is crucial for the above proposition. 
	For $\mf{k}(E_9)$, it is shown in \cite{Gen holonomies and ke9} that the kernels of $\mathcal{S}_{\frac{2s+1}{2}}$ form an ascending chain $\mathcal{I}_{\frac{7}{2}}\subsetneq \mathcal{I}_{\frac{5}{2}}\subsetneq \mathcal{I}_{\frac{3}{2}}\subsetneq \mathcal{I}_{\frac{1}{2}}$.
	This chain of inclusions is due to affine Kac-Moody-algebras containing the loop algebra $\mathbb{K}[t^{-1},t]\otimes\mathring{\mf{g}}$ and the fact that $\mf{k}(A)$ is contained in the Loop algebra.
	This makes it possible to spell out $\mathcal{S}_{\frac{3}{2}}$ and $\mathcal{S}_{\frac{5}{2}}$ as generalized evaluation maps of the loop algebra involving also derivatives of order up to $2$.
	This leads to all but the first inclusions, which is more subtle.
	
	By the above proposition, the kernels of tensor products of representations can provide truly smaller ideals. In \cite{Lautenbacher22} it was shown for $\mf{k}(E_{10})$ by a computer-based analysis that the tensor products $\mathcal{S}_{\frac{3}{2}}\otimes \mathcal{S}_{\frac{1}{2}}$ and $\mathcal{S}_{\frac{3}{2}}\otimes \bigwedge^2 \mathcal{S}_{\frac{1}{2}}  $ are both irreducible. 
	We believe that this is generally true for $A$ regular, indefinite, simply-laced, where $\mathcal{S}_{\frac{3}{2}}$ can also be replaced by $\widetilde{\mathcal{S}}_{\frac{5}{2}}$, as long as each factor in the tensor product is irreducible.
	The proof given in \cite{Lautenbacher22} however (cp. \cite[lem. 6.5 and prop. 6.7]{Lautenbacher22}), contains an error that cannot be easily fixed:
	The applied strategy is similar to that from props. \ref{prop:S32 irrreducible} and \ref{prop:trace-free S52}, i.e., one tries to exploit a polynomial identity of the representation matrices that enables to act on each factor of the tensor product individually. 
	The issue is that the identity derived in \cite[lem. 6.5]{Lautenbacher22} does not hold if the sign error in \cite[lem. 5.6]{Lautenbacher22} is fixed as we did in lemma \ref{lem:Split of the k(A)-action}. 
	With the correct sign, the matrices $\sigma(X_i)\otimes Id$ and $Id\otimes \rho(X_i)$ are not contained in the span of $\mu(X_i)^n$ for $n\in\mathbb{N}$ with $\mu(X_i)=\sigma(X_i)\otimes Id+Id\otimes \rho(X_i)$.
	Instead one has that $\mu(X_i)$ satisfies the identity $\mu(X_i)^4 = -\frac{5}{2} \mu(X_i)^2 -\frac{9}{16}$.  
\end{remark}

\section{Lift to $Spin(A)$ and compatibility with action of $W^{spin}(A)$} \label{sec: Spin(A)}
In this section we show that the higher spin representations do not lift to the maximal compact subgroup $K(A)$ of the minimal split-real Kac-Moody group $G(A)$ but only to its spin cover $Spin(A)$.
We also analyze the interaction of the spin representations with the action of the spin-extended Weyl group introduced in \cite{Spin covers}, which we use to derive a parametrization result for the representation matrices.
Our exposition follows \cite[chp. 4]{Lautenbacher22} and we only briefly review spin covers, referring to \cite{Spin covers} for more details.

\subsection{Maximal compact subgroups, spin covers and a lift criterion}
There are different groups that can be associated to a Kac-Moody algebra (cp. \cite{Introduction to KM-groups over fields}).
The minimal simply-connected, split-real Kac-Moody group $G(A)$ (cp. \cite[chp. 7]{Introduction to KM-groups over fields}) possesses an involution $\theta$ that is the group analogue of the Chevalley involution $\omega$.
The fixed-point subgroup $K(A)\coloneqq G(A)^\theta$ is commonly referred to as the maximal compact subgroup and its Lie algebra is $\mf{k}(A)$.
We are mostly concerned with the spin cover $Spin(A)$ of $K(A)$ introduced in \cite{Spin covers} for simply-laced $A$ and are therefore able to avoid the machinery involving the constructive Tits functor.
Instead we follow \cite{Spin covers} and construct $K(A)$ as an $SO(2)$-amalgam (defined below) without constructing $G(A)$ first.
We merely note that the amalgamation approach to $K(A)$ as well as for $G(A)$ is equivalent to the constructive Tits functor for two-spherical diagrams, i.e., diagrams whose rank-2 sub-diagrams are all of finite type.

\begin{defn}
	(Cp.  \cite[3.1 \& 3.4]{Spin covers})
	For $I=\{1,\dots,n\}$ and $i\neq j\in I$ let $G_{i}$, $G_{ij}$  be groups with monomorphisms $\psi_{ij}^{i}:G_{i}\rightarrow G_{ij}$.
	One calls
	\[
	\mathcal{A}:=\left\{ G_{i},G_{ij},\psi_{ij}^{i}\vert i\neq j\in I\right\} 
	\]
	and the $\psi_{ij}^{i}$ an \emph{amalgam of groups} and \emph{connecting homomorphisms}, respectively. 
	If $G_{i}\cong U$ for all $i\in I$, one calls $\mathcal{A}$ an $U$-amalgam.
	It is called \emph{continuous} if all $G_{i}$, $G_{ij}$ are topological groups with continuous connecting homomorphisms $\psi_{ij}^{i}$.
\end{defn}

\begin{defn}(Cp. \cite[3.5-6]{Spin covers})
	Let $\mathcal{A}=\left\{ G_{i},G_{ij},\psi_{ij}^{i}\vert i\neq j\in I\right\} $
	be an amalgam of groups.
	A group $G$ together with homomorphisms $\tau:=\left\{ \tau_{ij}:G_{ij}\rightarrow G\right\} $
	such that $\tau_{ij}\circ\psi_{ij}^{i}=\tau_{ik}\circ\psi_{ik}^{i}$	for all $j\neq i\neq k\in I$ is called	an \emph{enveloping group} of $\mathcal{A}$ with \emph{enveloping	homomorphisms} $\tau_{ij}$. 
	One calls $(G,\tau)$ \emph{faithful}, if all $\tau_{ij}$ are injective, and \emph{universal} if to any  enveloping group $\left(H,\tilde{\tau}\right)$ of $\mathcal{A}$, there exists a unique epimorphism	$\pi:G\rightarrow H$ such that $\pi\circ\tau_{ij}=\tilde{\tau}_{ij}$.
\end{defn}
Given a fixed amalgam $\mathcal{A}$, two universal enveloping groups  are uniquely isomorphic by universality. 
The canonical universal enveloping group (CUEG) 
\begin{equation}
	G\left(\mathcal{A}\right):=\left\langle \bigcup_{i\neq j\in I}G_{ij}\,\vert\,\text{all relations in }G_{ij},\ \forall\,i\neq j\neq k,\,\forall\,x\in G_{j}:\psi_{ij}^{j}(x)=\psi_{kj}^{j}(x)\right\rangle .\label{eq:def of CUEG}
\end{equation}
is a universal enveloping group  (\cite[1.3.2]{Geom. spor. grps}), the above phrasing is as in \cite{Spin covers}.
For general GCMs $A$ and field $\mathbb{K}$, the split minimal Kac-Moody group $G(A)$ over $\mathbb{K}$ associated to $\mf{g}(A)(\mathbb{K})$ is defined via the constructive Tits functor (cp. \cite{pres of KM groups}).
For $A$ two-spherical however, one has from the main result of \cite{presentations of BN-pair amalgams}  that the split minimal Kac-Moody group $G(A)$ over $\mathbb{R}$  is the universal enveloping group of the amalgam $\mathcal{A}=\left\{ G_{i},G_{ij},\phi_{ij}^{i}\right\} $, where $G_{i}=SL\left(2,\mathbb{R}\right)$ and $G_{ij}$ is the split-real algebraic group of type $A_{\{i,j\}}$.
The $\phi_{ij}^{i}:G_{i}\hookrightarrow G_{ij}$ are the canonical inclusion maps induced from $G_{ij}$ being generated by its fundamental rank-$1$ subgroups $G_i$ and $G_j$.
The restriction $\theta\vert_{G_{ij}}$ of $\theta$ to any fundamental rank-$2$ subgroup yields the classical Cartan-Chevalley involution on the split-real Lie group $G_{ij}$ because $A$ is two-spherical.
\cite{Spin covers} shows that $K(A)$ is an amalgam of the   $G_{ij}^{\theta}$ which are the classical maximal-compact subgroups.
In order to state their result precisely we need to introduce a few more objects first and we restrict ourselves to the simply-laced situation.

For $A$ simply-laced one has $G_i^\theta\cong SO(2)$ and $G_{ij}^\theta$ isomorphic to $SO(3)$ or $SO(2)\times SO(2)$.
We set
\begin{equation}
	K_{ij}:=\begin{cases}
		SO(3) & \text{if }\{i,j\}\in\mathcal{E}\\
		SO(2)\times SO(2) & \text{if }\{i,j\}\notin\mathcal{E}
	\end{cases}\label{eq:G_ij for SO}
\end{equation}
For a group $H$ set 
\begin{equation}
	i_{1}:H\rightarrow H\times H,\ h\mapsto(h,e),\ i_{2}:H\rightarrow H\times H,\ h\mapsto(e,h).\label{eq:diag embedding maps}
\end{equation}
Furthermore, denote by $\varepsilon_{12}:SO(2)\hookrightarrow SO(3)$ the embedding
via the upper-left $SO(2)$-subgroup and by $\varepsilon_{23}:SO(2)\hookrightarrow SO(3)$
that via the lower-right $SO(2)$-subgroup. 

\begin{defn}	\label{def:so(2) amalgam}(Cp. \cite[def. 9.1]{Spin covers})
	Let	$A\in\mathbb{Z}^{n\times n}$ be a simply-laced GCM, $\mathcal{E}$ its Dynkin diagram's edges and $I=\left\{ 1,\dots,n\right\} $. 
	The \emph{standard $SO(2)$-amalgam} of type $A$ is defined as
	\[
	\mathcal{A}\left(A,SO(2)\right):=\left\{ K_{i}\cong SO(2),K_{ij},\phi_{ij}^{i}\vert i\neq j\in I\right\} 
	\]
	with $K_{ij}$ as in (\ref{eq:G_ij for SO}) and for all $i<j\in I$:
	\[
	\phi_{ij}^{i}=\begin{cases}
		\varepsilon_{12} & \text{if }\{i,j\}\in\mathcal{E}\\
		i_{1} & \text{if }\{i,j\}\notin\mathcal{E}
	\end{cases},\ \phi_{ij}^{j}=\begin{cases}
		\varepsilon_{23} & \text{if }\{i,j\}\in\mathcal{E}\\
		i_{2} & \text{if }\{i,j\}\notin\mathcal{E}.
	\end{cases}
	\]
\end{defn}
According to \cite[9.5]{Spin covers}, changing the Dynkin diagram's labeling $I$  does not affect the isomorphism type of $\mathcal{A}$ (we have not defined this term, cp. \cite[3.2]{Spin covers}).
It is possible to concatenate the $\phi_{ij}^{i}$ by an isomorphism of $SO(2)$, which  \cite{Spin covers} call simply an $SO(2)$-amalgam and
such an  $SO(2)$-amalgam is only guaranteed to be isomorphic to the standard  one if the connecting homomorphisms are continuous.
Split-real Kac-Moody groups of $2$-spherical type naturally carry the so-called Kac-Peterson topology that induces the Lie topology on their spherical subgroups (cp. \cite{Regular functions} and \cite{top. twin buildings}). 
Thus, it is natural in our setting to consider only continuous connecting homomorphisms.

The block-embeddings $\varepsilon_{12},\varepsilon_{23}$ have a canonical lift $\tilde{\varepsilon}_{12}$, $\tilde{\varepsilon}_{23}:Spin(2)\hookrightarrow Spin(3)$ (cp. \cite[lem. 6.10]{Spin covers}) which enables one to define the standard
$Spin(2)$-amalgam of type $A$ in the same manner as the standard $SO(2)$-amalgam:
\begin{defn}\label{def:spin amalgam}
	(Cp. \cite[def. 10.1]{Spin covers})
	Let	$A\in\mathbb{Z}^{n\times n}$  be a simply-laced GCM, $\mathcal{E}$ its Dynkin diagram's edges and $I=\left\{ 1,\dots,n\right\} $. 
	Set
	\begin{equation}
		K_{ij}:=\begin{cases}
			Spin(3) & \text{if }\{i,j\}\in\mathcal{E}\\
			Spin(2)\times Spin(2)\diagup\left\langle (-1,-1)\right\rangle  & \text{if }\{i,j\}\notin\mathcal{E}.
		\end{cases}\label{eq:G_ij for Spin}
	\end{equation}
	The \emph{standard $Spin(2)$-amalgam} of type $A$ is defined as
	\[
	\mathcal{A}\left(A,Spin(2)\right):=\left\{ K_{i}\cong Spin(2),K_{ij},\phi_{ij}^{i}\vert i\neq j\in I\right\} 
	\]
	with $K_{ij}$ as in (\ref{eq:G_ij for Spin}) and for all $i<j\in I$:
	\begin{align*}
	\phi_{ij}^{i}=\begin{cases}
		\tilde{\varepsilon}_{12} & \text{if }\{i,j\}\in\mathcal{E}\\
		i_{1} & \text{if }\{i,j\}\notin\mathcal{E}
	\end{cases},\ \phi_{ij}^{j}=\begin{cases}
		\tilde{\varepsilon}_{23} & \text{if }\{i,j\}\in\mathcal{E}\\
		i_{2} & \text{if }\{i,j\}\notin\mathcal{E}.
	\end{cases}
	\end{align*}
\end{defn}
As before the Dynkin diagram's labeling does not matter (see \cite[cor. 10.7]{Spin covers}) and  any continuous $Spin(2)$-amalgam
of type $A$ is isomorphic to $\mathcal{A}\left(A,Spin(2)\right)$ (cp. \cite[thm. 10.9]{Spin covers}).
\begin{defn}	(Cp. \cite[def. 11.5]{Spin covers})
	Let $A$ be a simply-laced GCM. Define $Spin(A)$ to be the CUEG of the standard $Spin(2)$-amalgam
	$\mathcal{A}\left(A,Spin(2)\right)$ of type $A$.
\end{defn}

\begin{theorem}
	(Cp. \cite[thm. 11.2]{Spin covers})
	Let $A$ be a simply-laced GCM and $G(A)$ the minimal, simply-connected split-real Kac-Moody group of type $A$. 
	Denote its  Cartan-Chevalley involution by $\theta$, then the maximal compact subgroup $K\left(A\right):=G(A)^{\theta}$ is a faithful	universal enveloping group of the standard $SO(2)$-amalgam $\mathcal{A}\left(A,SO(2)\right)$.	
\end{theorem}

\begin{theorem}	(Cp. \cite[thm. 11.17]{Spin covers}) \label{thm: Spin(A) is central extension}
	Let $A$ be a simply-laced GCM , then  $Spin(A)$ is a $2^{s}$-fold central extension of $K(A)$ where $s$ is the number of connected
	components of the Dynkin diagram of type $A$.
\end{theorem}

We are now in the position to formulate a criterion that decides whether or not a given representation of $\mf{k}(A)$ lifts to $K(A)$ or to $Spin(A)$.

\begin{proposition}[This is \cite{Lautenbacher22}, prop. 4.8]	\label{prop:Lift criterion} 
	Let $A$ be a simply-laced, indecomposable GCM and let $\rho:\mathfrak{k}(A)\left(\mathbb{R}\right)\rightarrow\text{End}(V)$
	be a finite-dimensional representation. 
	As usual, denote the Berman-generators of $\mathfrak{k}(A)\left(\mathbb{R}\right)$ by $X_i$, then one of the following two cases applies:	
	\begin{equation}
		\exp\left(2\pi\rho\left(X_{i}\right)\right)=\begin{cases}
			-Id_{V} & \ \forall\,i\in I\\
			Id_{V} & \ \forall\,i\in I
		\end{cases}\:.\label{eq: case distinction lift}
	\end{equation}	
	In both cases, $\rho$ lifts to a representation $\Omega$ of	$Spin(A)$, but it lifts to $K\left(A\right)$ only	in the second case.
\end{proposition}
\begin{proof}
We denote the canonical subalgebras generated by the subdiagram $J\subset I$ by $\mathfrak{k}_{J}:=\left\langle X_{j}:j\in J\right\rangle $.
Assume $J$ is such that $A_J$ is of finite type and let $\left(\phi,U\right)$ be a finite-dimensional irreducible
representation of $\mf{k}_J$.
The representation always lifts to the simply-connected Lie group $\widetilde{K}_{J}$ with Lie algebra $\mathfrak{k}_{J}$ and the first diagram in fig. \ref{fig:Com diagrams 4_8} commutes. 
For $J$ such that $A_J$ is of finite type, in particular for $\vert J \vert \leq 2$ one has that $\rho$ restricted to $\mathfrak{k}_{J}$ is completely reducible.
In the rank $1$ case, there is no distinction in the lift of $\rho$   since the involved groups are $SO(2)\cong U(1)\cong Spin(2)$. 
For  $J=\{i,j\}$ such that $\left\{ i,j\right\}\in\mathcal{E}$,  $\mf{k}_{\{i,j\}}\cong\mf{so}(3)$ and all finite-dimensional representations lift to the universal cover $Spin(3)$ but not all lift to $SO(3)$.
By comparison with the adjoint action of $\mathfrak{so}(3)$ on itself, one can determine if a given representation lifts to $SO(3)$.
Due to $\left[X_{i},\left[X_{i},X_{j}\right]\right]=-X_{j}$  one has
\begin{eqnarray*}
	\exp\left(\phi\cdot\text{ad}_{X_{i}}\right)\left(X_{j}\right) & = & \sum_{n=0}^{\infty}\frac{\left(-1\right)^{n}\phi^{2n}}{(2n)!}X_{j}+\sum_{n=0}^{\infty}\frac{\left(-1\right)^{n}\phi^{2n+1}}{(2n+1)!}\left[X_{i},X_{j}\right]\\
	& = & \cos\left(\phi\right)X_{j}+\sin\left(\phi\right)\left[X_{i},X_{j}\right]
\end{eqnarray*}
 for all $i,j$ such that $\{i,j\}\in\mathcal{E}$.
Since the exponential of the $\text{ad}_{X_{i}}$ is $2\pi$-periodic in the sense that $\exp\left(2\pi\cdot\text{ad}_{X_{i}}\right)=Id$, a representation that lifts to $SO(3)$ has to satisfy
the second case of (\ref{eq: case distinction lift}).
Also, the first case of (\ref{eq: case distinction lift}) is incompatible with a lift to $SO(3)$ and therefore such representations only lift to $Spin(3)$.

	\begin{figure}
		\begin{centering}
			\includegraphics[scale=1.5]{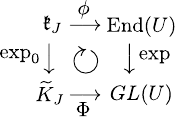}$\quad\qquad\qquad\qquad$\includegraphics[scale=1.5]{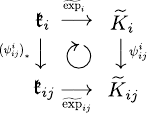}
			\par\end{centering}
		\caption{\label{fig:Com diagrams 4_8} Two important commutative diagrams for a spherical subdiagram $J$. $\exp_0$ denotes the abstract exponential map of a finite-dimensional Lie algebra to the identity component of a Lie group with that Lie algebra.}
		
	\end{figure}

We now need to extend this lifting behavior from the $\mf{k}_J$ to all of $\mf{k}$. 
We have already seen by the above that lifts exist locally for fixed spherical subalgebras and that the normalization of a single Berman generator tells us all we need to know for the others.
The extension to $Spin(A)$ and $K(A)$ respectively works similar to the proof of \cite[thm. 11.14]{Spin covers}.
For every $i,j$, the local lift of $\mf{k}_{\{i,j\}}$ to the group level induces an enveloping homomorphism $\tau_{\{i,j\}}:\widetilde{K}_{ij}\rightarrow GL(V)$ in case one and $ \tau_{ij}:K_{ij}\rightarrow GL(V)$ in case two of (\ref{eq: case distinction lift}).
In both cases we collect them in the set $\tau$.
One obtains these lifts as follows:
In the first case  one identifies  $\widetilde{K}_{\{i,j\}}$ with the exponential image of $\mathfrak{k}_{\{i,j\}}$ under $\widetilde{\exp}_{\{i,j\}}:\mathfrak{k}_{\{i,j\}}\rightarrow\widetilde{K}_{\{i,j\}}$ which is possible as the exponential map is surjective for compact Lie groups. 
This naturally induces connecting monomorphisms $\psi_{ij}^{i}:\widetilde{K}_{i}\rightarrow\widetilde{K}_{ij}$
that are compatible with exponentiation in the sense that the second diagram in figure \ref{fig:Com diagrams 4_8} commutes.
This enables one to define $\tau_{ij}:\widetilde{K}_{ij}\rightarrow GL(V)$ by $\tau_{ij}\left(\widetilde{\exp}_{ij}\left(x\right)\right)=\exp\left(\rho\left(x\right)\right)\ \forall\,x\in\mathfrak{k}_{\{i,j\}}$.
Since $\rho$ is globally defined on $\mathfrak{k}$ one has
\[
\tau_{ij}\circ\psi_{ij}^{i}\left(\widetilde{\exp}_{i}\left(x\right)\right)=\exp\left(\rho\left(x\right)\right)=\tau_{ik}\circ\psi_{ik}^{i}\left(\widetilde{\exp}_{i}\left(x\right)\right)\ \forall\,x\in\mathfrak{k}_{\{i\}}.
\]
This shows that the exponential of $\text{im}\,(\rho)$ is  an enveloping group of the standard $Spin(2)$-amalgam of type $A$ so that there exists a unique homomorphism $\Omega:Spin(A)\rightarrow GL(V)$.
If $\rho$ restricted to the rank $2$ subalgebras isomorphic to $\mf{so}(3)$ lifts to $SO(3)$, then the enveloping homomorphisms are $\tau_{\{i,j\}}:K_{ij}\rightarrow GL(V)$ and the above argument shows that the exponential of $\text{im}\,(\rho)$ is  an enveloping group of the standard $SO(2)$-amalgam of type $A$.
\end{proof}

\subsection{Lifting and interaction with spin-extended Weyl group}

Let $\sigma\left(X_{i}\right):=\tau\left(\alpha_{i}\right)\otimes\Gamma\left(\alpha_{i}\right)$ be a higher spin representation
as in (\ref{thm:S32 and S52}) and (\ref{thm:7/2 spin rep}) and denote the induced one-parameter subgroups by
\begin{equation}
	\Sigma_{i}\left(\phi\right):=\exp\left(\phi\cdot\sigma\left(X_{i}\right)\right)\:.\label{eq:def of Sigma_i}
\end{equation}
We would now like to compute an explicit formula for the above representation matrix. This has been done in \cite[sec. 5]{On higher spin} in a second-quantized form and without explicit reference to the Weyl group action, which is why we repeat the computation in our terminology.
\begin{proposition}
	\label{prop:Lift of 3/2 and 5/2 spin rep}
	Let $A$ be simply-laced and denote by $\left(\sigma,V\right)$ the representation $\mathcal{S}_{\frac{3}{2}}$ or $\mathcal{S}_{\frac{5}{2}}$ of $\mathfrak{k}\left(A\right)\left(\mathbb{R}\right)$	from thm. \ref{thm:S32 and S52}. Then 
	\begin{eqnarray}
		\Sigma_{i}\left(\phi\right) & = & \left[\cos\left(\phi\right)\cos\left(\frac{\phi}{2}\right)\cdot Id\otimes Id-\cos\left(\phi\right)\sin\left(\frac{\phi}{2}\right)\cdot Id\otimes\Gamma\left(\alpha_{i}\right)+\right.\nonumber \\
		&  & \left.\sin\left(\phi\right)\sin\left(\frac{\phi}{2}\right)\cdot\eta\left(s_{i}\right)\otimes Id+\sin\left(\phi\right)\cos\left(\frac{\phi}{2}\right)\cdot\eta\left(s_{i}\right)\otimes\Gamma\left(\alpha_{i}\right)\right],\label{eq:formula for spin rep exponential}
	\end{eqnarray}
	and $\left(\sigma,V\right)$ lifts to a representation $(\Sigma,V)$ of $Spin\left(A\right)$ but not
	to $K(A)$. The restriction of $(\Sigma,V)$ to the fundamental one-parameter subgroups are given by the $\Sigma_i$.
\end{proposition}

\begin{proof}
	One computes with $\Gamma\left(\alpha_{i}\right)^{2n}=(-1)^{n}$ and the shorthand $\tau\left(\alpha_{i}\right)^{n}=a(n)+b(n)\eta\left(s_{i}\right)$ for the representations $\mathcal{S}_{\frac{3}{2}}$ and $\mathcal{S}_{\frac{5}{2}}$ that 
	\begin{align*}
		\Sigma_{i}\left(\phi\right) &= A_{1}\left(\phi\right)\cdot Id\otimes Id+A_{2}\left(\phi\right)\cdot Id\otimes\Gamma\left(\alpha_{i}\right)+A_{3}\left(\phi\right)\cdot\eta\left(s_{i}\right)\otimes Id+A_{4}\left(\phi\right)\cdot\eta\left(s_{i}\right)\otimes\Gamma\left(\alpha_{i}\right),
	\end{align*}	
	with 
	\[
	A_{1}\left(\phi\right)=\sum_{n=0}^{\infty}\frac{\left(-1\right)^{n}\phi^{2n}}{(2n)!}a(2n),\ A_{2}\left(\phi\right)=\sum_{n=0}^{\infty}\frac{\left(-1\right)^{n}\phi^{2n+1}}{(2n+1)!}a(2n+1),
	\]
	\[
	A_{3}\left(\phi\right)=\sum_{n=0}^{\infty}\frac{\left(-1\right)^{n}\phi^{2n}}{(2n)!}b(2n),\ A_{4}\left(\phi\right)=\sum_{n=0}^{\infty}\frac{\left(-1\right)^{n}\phi^{2n+1}}{(2n+1)!}b(2n+1).
	\]	
	Because $\begin{pmatrix}n\\
		k
	\end{pmatrix}=0$ for $k>n$ and $s_{i}^{2}=e$ one has
	\begin{align*}
		\tau\left(\alpha_{i}\right)^{n} &= \eta\left(s_{i}\right)^{n}\cdot\left[\sum_{k=0}^{\infty}\begin{pmatrix}n\\
			2k
		\end{pmatrix}\left(-2\right)^{-2k}+\eta\left(s_{i}\right)\sum_{k=0}^{\infty}\begin{pmatrix}n\\
			2k+1
		\end{pmatrix}\left(-2\right)^{-2k-1}\right]\:.
	\end{align*}
	The coefficients of a holomorphic function's Laurent series $f\left(z\right)=\sum_{n=0}^{\infty}a_{n}z^{n}$ satisfy $\sum_{m=0}^{\infty}a_{2m}z^{2m}=\frac{1}{2}\left[f(z)+f(-z)\right]$ and  $\sum_{m=0}^{\infty}a_{2m+1}z^{2m+1}=\frac{1}{2}\left[f(z)-f(-z)\right]$.
	Specialized to $f\left(z\right):=\left(1+z\right)^{n}$ this yields
	\begin{align*}
		\sum_{k=0}^{\infty}\begin{pmatrix}n\\
			2k
		\end{pmatrix}\left(-2\right)^{-2k} &= \frac{1}{2}\left[f(z)+f(-z)\right]_{z=-\frac{1}{2}}=\frac{1}{2}\cdot\left(\frac{1}{2}\right)^{n}+\frac{1}{2}\cdot\left(\frac{3}{2}\right)^{n},\\
		\sum_{k=0}^{\infty}\begin{pmatrix}n\\
			2k+1
		\end{pmatrix}\left(-2\right)^{-2k-1} &= \frac{1}{2}\left[f(z)-f(-z)\right]_{z=-\frac{1}{2}}=\frac{1}{2}\cdot\left(\frac{1}{2}\right)^{n}-\frac{1}{2}\cdot\left(\frac{3}{2}\right)^{n}
	\end{align*}
	and therefore 
	\begin{align*}
		a\left(2n\right) &=\frac{1}{2}\cdot\left[\left(\frac{1}{2}\right)^{2n}+\left(\frac{3}{2}\right)^{2n}\right]&,\ a\left(2n+1\right)&=\frac{1}{2}\cdot\left[\left(\frac{1}{2}\right)^{2n+1}-\left(\frac{3}{2}\right)^{2n+1}\right],\\
		b\left(2n\right)&=\frac{1}{2}\cdot\left[\left(\frac{1}{2}\right)^{2n}-\left(\frac{3}{2}\right)^{2n}\right]&,\ b\left(2n+1\right)&=\frac{1}{2}\cdot\left[\left(\frac{1}{2}\right)^{2n+1}+\left(\frac{3}{2}\right)^{2n+1}\right].		
	\end{align*}
	From this, one has with some trigonometric identities that
	\begin{align*}
		A_{1}\left(\phi\right) & = \frac{1}{2}\cos\left(\frac{\phi}{2}\right)+\frac{1}{2}\cos\left(\frac{3\phi}{2}\right)=\cos\left(\phi\right)\cos\left(\frac{\phi}{2}\right), \\
		A_{2}\left(\phi\right) & = \frac{1}{2}\sin\left(\frac{\phi}{2}\right)-\frac{1}{2}\sin\left(\frac{3\phi}{2}\right) = -\cos\left(\phi\right)\sin\left(\frac{\phi}{2}\right), \\
		A_{3}\left(\phi\right) & = \frac{1}{2}\cos\left(\frac{\phi}{2}\right)-\frac{1}{2}\cos\left(\frac{3\phi}{2}\right) = \sin\left(\phi\right)\sin\left(\frac{\phi}{2}\right), \\
		A_{4}\left(\phi\right) & = \frac{1}{2}\sin\left(\frac{\phi}{2}\right)+\frac{1}{2}\sin\left(\frac{3\phi}{2}\right) = \sin\left(\phi\right)\cos\left(\frac{\phi}{2}\right)\:.
	\end{align*}
	Finally, one evaluates this for $\phi=2\pi$ and obtains $A_{1}=-1$ and $A_{2}=A_{3}=A_{4}=0$ so that 
	\[
	\exp\left(2\pi\cdot\sigma\left(X_{i}\right)\right)=-Id\otimes Id
	\]
	Now prop. \ref{prop:Lift criterion} implies that $\sigma$ lifts only to $Spin(A)$.
\end{proof}

\begin{proposition}	\label{prop:Lift of 7/2 spin rep}
	Let $A$ be simply-laced and denote by $\left(\sigma,V\right)$ the representation $\mathcal{S}_{\frac{7}{2}}$ of $\mathfrak{k}\left(A\right)\left(\mathbb{R}\right)$	from prop. \ref{thm:7/2 spin rep}. 
	Denote the image of the fundamental one-parameter subgroups by  $\widetilde{\Sigma}_{i}$, then
	\begin{align}
		\widetilde{\Sigma}_{i}\left(\phi\right) &= \Sigma_{i}\left(\phi\right) + \frac{1}{4}\left[\cos\left(\frac{5}{2}\phi\right)-\cos\left(\frac{3}{2}\phi\right)\right]f\left(\alpha_{i}\right)\otimes Id  \nonumber \\ 
		& +\frac{1}{4}\left[\sin\left(\frac{5}{2}\phi\right) +\sin\left(\frac{3}{2}\phi\right) \right] f\left(\alpha_{i}\right)\otimes\Gamma\left(\alpha_{i}\right) ,\label{eq:7/2 rep exponential}
	\end{align}
	with $\Sigma_{i}\left(\phi\right)$  as in (\ref{eq:formula for spin rep exponential})	and $\left(\sigma,\mathcal{S}_{\frac{7}{2}}\right)$	lifts to $Spin\left(A\right)$ but not to $K(A)$.
\end{proposition}

\begin{proof}
	
	Recall from (\ref{eq:ansatz 7/2}) that 
	\begin{align*}
	\sigma\left(X_{i}\right) &= \tau\left(\alpha_{i}\right) \otimes \Gamma\left(\alpha_{i}\right),\quad \tau\left(\alpha\right)\coloneqq \eta\left(s_{\alpha}\right) -\frac{1}{2}Id+f\left(\alpha\right)\ \forall\,\alpha\in\Delta_{+}^{re}
	\end{align*}
	with $\eta:W\rightarrow GL(V)$ denoting the induced action of the Weyl group on $V=\text{Sym}^{3}\left(\mathfrak{h}^{\ast}\right)$ and $f\left(\alpha\right) = v(\alpha)\left( v(\alpha) \vert \cdot \right)$ with $v(\alpha)\in \text{Sym}^{3}\left(\mathfrak{h}^{\ast}\right)$ defined in thm. \ref{thm:7/2 spin rep} and eq. (\ref{eq:ansatz 7/2}) satisfying 
	\begin{align*}
		f\left(\alpha\right)^{2} &=4\cdot f\left(\alpha\right),\ \eta\left(s_{\alpha}\right)f\left(\alpha\right) =f\left(\alpha\right)\eta\left(s_{\alpha}\right) =-f\left(\alpha\right).
	\end{align*}
	Setting $\tilde{\tau}\left(\alpha\right)\coloneqq\eta\left(s_{\alpha}\right)-\frac{1}{2}Id$,
	one has
	\begin{align*}
		\tilde{\tau}\left(\alpha\right)f\left(\alpha\right) &= f\left(\alpha\right)\tilde{\tau}\left(\alpha\right)=-\frac{3}{2}f\left(\alpha\right),
	\end{align*}
	With this, one computes
	\begin{align*}
		\tau\left(\alpha\right)^{n} & = \left[\tilde{\tau}\left(\alpha\right)+f\left(\alpha\right)\right]^{n} =\sum_{k=0}^{n}\begin{pmatrix}n\\
			k
		\end{pmatrix} \tilde{\tau}\left(\alpha\right)^{n-k} f\left(\alpha\right)^{k} \\ &=\tilde{\tau}\left(\alpha\right)^{n}+\sum_{k=1}^{n}\begin{pmatrix}n\\
			k
		\end{pmatrix} \left(-\frac{3}{2}\right)^{n-k}4^{k-1} f\left(\alpha\right)
		= \tilde{\tau}\left(\alpha\right)^{n}+\frac{1}{4}\left[\left(\frac{5}{2}\right)^{n}-\left(-\frac{3}{2}\right)^{n}\right]f\left(\alpha\right)\:,
	\end{align*}
	\begin{align*}
		\widetilde{\Sigma}_{i}\left(\phi\right) &\coloneqq \exp\left(\phi\sigma\left(X_{i}\right)\right) =\sum_{n=0}^{\infty} \frac{\phi^{n}}{n!} \tau\left(\alpha_{i}\right)^{n} \otimes \Gamma\left(\alpha_{i}\right)^{n}\\
		&= \sum_{n=0}^{\infty} \frac{\left(-1\right)^{n}\phi^{2n}}{\left(2n\right)!} \tau\left(\alpha_{i}\right)^{2n}\otimes Id +\sum_{n=0}^{\infty} \frac{\left(-1\right)^{n}\phi^{2n+1}}{\left(2n+1\right)!} \tau\left(\alpha_{i}\right)^{2n+1}\otimes\Gamma\left(\alpha_{i}\right)\\
		&= \underset{\eqqcolon\Sigma_{i}\left(\phi\right)} {\underbrace{\sum_{n=0}^{\infty} \frac{\left(-1\right)^{n}\phi^{2n}}{\left(2n\right)!} \tilde{\tau}\left(\alpha_{i}\right)^{2n} \otimes Id + \sum_{n=0}^{\infty} \frac{\left(-1\right)^{n}\phi^{2n+1}}{\left(2n+1\right)!} \tilde{\tau}\left(\alpha_{i}\right)^{2n+1} \otimes \Gamma\left(\alpha_{i}\right)}}\\
		&\ +\frac{1}{4}\sum_{n=0}^{\infty} \frac{\left(-1\right)^{n}\phi^{2n}}{\left(2n\right)!} \left[\left(\frac{5}{2}\right)^{2n}-\left(-\frac{3}{2}\right)^{2n}\right] f\left(\alpha_{i}\right)\otimes Id\\
		&\ +\frac{1}{4}\sum_{n=0}^{\infty} \frac{\left(-1\right)^{n}\phi^{2n+1}}{\left(2n+1\right)!} \left[\left(\frac{5}{2}\right)^{2n+1}-\left(-\frac{3}{2}\right)^{2n+1}\right] f\left(\alpha_{i}\right)\otimes\Gamma\left(\alpha_{i}\right)
	\end{align*}
	\begin{align*}
		\widetilde{\Sigma}_{i}\left(\phi\right) & = \Sigma_{i}\left(\phi\right) +\frac{1}{4} \left[ \cos\left(\frac{5}{2}\phi\right) - \cos\left(\frac{3}{2}\phi\right) \right] f\left(\alpha_{i}\right)\otimes Id\\
		&   +\frac{1}{4}\left[ \sin\left(\frac{5}{2}\phi\right) + \sin\left(\frac{3}{2}\phi\right) \right]  f\left(\alpha_{i}\right) \otimes \Gamma\left(\alpha_{i}\right)\:.
	\end{align*}
	Above, $\Sigma_{i}\left(\phi\right)$ stands for the same expression	as in (\ref{eq:formula for spin rep exponential}) but now with $\eta:W(A)\rightarrow GL(V)$ for $V=\text{Sym}^3(\mf{h}^\ast)$.
	Concerning the periodicity of $\Sigma_{i}$, the proof of prop. \ref{prop:Lift of 3/2 and 5/2 spin rep} only relies on $\eta$ being a representation of $W(A)$ and the properties of $\Gamma$-matrices.
	Therefore, $\Sigma_{i}$ here is also $4\pi$-periodic and so is the remainder of $\widetilde{\Sigma}_{i}$.
	As before, prop. \ref{prop:Lift criterion} now implies that $\sigma$ lifts only to $Spin(A)$.
\end{proof}
	
\begin{remark}
	One observes that eigenvalues of largest absolute value of $\sigma(X_i)$ are $\pm\frac{5}{2}$ for $(\sigma,\mathcal{S}_{\frac{7}{2}})$ and $\pm\frac{3}{2}$ for $(\sigma,\mathcal{S}_{\frac{5}{2}})$   and $(\sigma,\mathcal{S}_{\frac{3}{2}})$, so that the notion of $\frac{n}{2}$-representation used in \cite{Higher spin realizations} and \cite{Ext gen spin reps} doesn't quite fit. 
	Essential to the occurrence of $\frac{5}{2}$ as an eigenvalue is that $f(\alpha)^2=4f(\alpha)$. 
	Consider abstractly $f(\alpha)^2=af(\alpha)$ and compute the adjoint action on the representation side, i.e., $\widetilde{\Sigma}_{i}\left(\frac{\pi}{2}\right)\sigma\left(X_{j}\right)\widetilde{\Sigma}_{i}\left(-\frac{\pi}{2}\right)$.
	This must be proportional to $\sigma\left(\left[X_{i},X_{j}\right]\right)$ which it only is if $a=0\,\text{mod }4$.
\end{remark}
In the above remark, we claimed without further details that $\widetilde{\Sigma}_{i}\left(\frac{\pi}{2}\right)\sigma\left(X_{j}\right)\widetilde{\Sigma}_{i}\left(-\frac{\pi}{2}\right)$ must be proportional to $\sigma\left(\left[X_{i},X_{j}\right]\right)$.
The reasoning behind this is how representations interact with the extended and spin-extended Weyl group and their action on $\mf{k}$, which we will now analyze in detail.

\begin{defn} (Cp. \cite[def. 18.4]{Spin covers})
	Let $A$ be a symmetrizable GCM and define $n(i,j)$ to be $0$ if $a_{ij}$ is even and $n(i,j)=1$ if $a_{ij}$ is odd.
	Furthermore, define the Coxeter coefficients $m_{ij}$ for $i\neq j$ via 
	
\begin{center}
		\begin{tabular}{c|c|c|c|c|c}
		$a_{ij}a_{ji}$	& $0$ & $1$ & $2$ & $3$ & $\geq4$		\tabularnewline
		\hline
		$m_{ij}$ & $2$ & $3$ & $4$ & $6$ & $\infty$ 
	\end{tabular}
\end{center}	
\vspace{5pt}	
	\noindent Define the \emph{extended Weyl group} $W^{ext}\left(A\right)$ as
	\begin{subequations}
	\begin{eqnarray}
		W^{ext}\left(A\right) & = & \left\langle t_{1},\dots,t_{n}\right|\ t_{i}^{4}=e\ \forall\,i\in I,\label{eq:W ext T1}\\
		\qquad\qquad &  & t_{j}^{-1}t_{i}^{2}t_{j}=t_{i}^{2}t_{j}^{2n(i,j)}\ \forall\,i\neq j\in I,\label{eq:W ext T2}\\
		 &  & \left.\underset{m_{ij}\text{ factors}}{\underbrace{t_{i}t_{j}t_{i}\cdots}}=\underset{m_{ij}\text{ factors}}{\underbrace{t_{j}t_{i}t_{j}\cdots}}\ \forall\,i\neq j\in I\right\rangle .\label{eq:W ext T3}
	\end{eqnarray}
	\end{subequations}
	and similarly define the \emph{spin-extended Weyl group} $W^{spin}\left(A\right)$ as	
	\begin{eqnarray}
		W^{spin}\left(A\right) & = & \left\langle r_{1},\dots,r_{n}\right|\ r_{i}^{8}=e\ \forall\,i\in I,\label{eq:W spin R1}\\
		\ \qquad\qquad &  & r_{j}^{-1}r_{i}^{2}r_{j}=r_{i}^{2}r_{j}^{2n(i,j)}\ \forall\,i\neq j\in I,\label{eq:W spin R2}\\
		 &  & \left.\underset{m_{ij}\text{ factors}}{\underbrace{r_{i}r_{j}r_{i}\cdots}}=\underset{m_{ij}\text{ factors}}{\underbrace{r_{j}r_{i}r_{j}\cdots}}\ \forall\,i\neq j\in I\right\rangle .\label{eq:W spin R3}
	\end{eqnarray}
\end{defn}

Note that $W\left(A\right)$ is not a subgroup of the minimal Kac-Moody group $G\left(A\right)$. 
However, to any integrable representation $\left(\pi,V\right)$ of $\mathfrak{g}$ one can introduce 
\[
t_{i}\coloneqq\exp\pi\left(f_{i}\right)\exp\left(-\pi\left(e_{i}\right)\right)\exp\pi\left(f_{i}\right)
\]
which have the following effect on the weight spaces $V_{\lambda}$ (cp. \cite[lem. 3.8]{Kac:}) 
\begin{equation}
	t_{i}\left(V_{\lambda}\right)=V_{s_{i}.\lambda},\label{eq:action of Wext on weight spaces}
\end{equation}
with the simple Weyl reflection $s_{i}\in W\left(A\right)$.
Denote by $G^{\pi}\leq GL\left(V\right)$ the group generated by $\exp\pi\left(k f_{i}\right),\exp\pi\left(k\alpha_{i}^{\vee}\right)$ and $\exp\pi\left(k e_{i}\right)$ for $k\in\mathbb{K}$.
Then the subgroup $W^\pi\left(A\right)<G^{\pi}$ generated by the $t_i$ contains an abelian normal subgroup $D^{\pi}=\left\langle t_{i}^{2}\,\vert\,i\in I\right\rangle $ that satisfies $W^\pi\left(A\right)\diagup D^{\pi}\cong W\left(A\right)$ if $\ker\pi\subset\mathfrak{h}$ (this is \cite[rem. 3.8]{Kac:}, originally due to \cite{Def relations of certain inf-dim groups}).
Furthermore, $W^{ext}\left(A\right)$ and $W^{spin}\left(A\right)$ are in fact subgroups of $K(A)$ and $Spin(A)$, respectively, which can be seen by the following construction:
\begin{defn} (Cp. \cite[def. 18.3]{Spin covers}) 
	Let $A$ be simply-laced	and let $\mathcal{A}\left(A,Spin(2)\right)$	be the associated standard spin-amalgam  whose connecting monomorphisms we denote by	$\tilde{\phi}_{ij}^{i}:\tilde{G}_{i}\rightarrow\tilde{G}_{ij}$, and similarly	let $\mathcal{A}\left(A,SO(2)\right)$ be the associated standard $SO(2)$-amalgam with connecting monomorphisms $\phi_{ij}^{i}:G_{i}\rightarrow G_{ij}$.
	We denote the enveloping homomorphisms by $\tilde{\psi}_{ij}:\tilde{G}_{ij}\rightarrow Spin\left(A\right)$
	and $\psi_{ij}:G_{ij}\rightarrow K(A)$, respectively. 
	The $2\pi$-periodic covering maps $S:\mathbb{R}\rightarrow Spin(2)$ and $D:\mathbb{R}\rightarrow SO(2)$ are given explicitly by 
	$S(\alpha)=\cos\alpha+\sin\alpha e_{1}e_{2}$, where $e_{1},e_{2}\in Cl\left(\mathbb{R}^{2}\right)$, the Clifford algebra over $\mathbb{R}^2$, and 
	$D(\alpha)=\begin{pmatrix}\cos\alpha & \sin\alpha\\
		-\sin\alpha & \cos\alpha
	\end{pmatrix}$
	 respectively (cp. \cite[8.1]{Spin covers}). We set for $i< j$:
	\[
	\hat{r}_{i}\coloneqq\tilde{\psi}_{ij}\circ\tilde{\phi}_{ij}^{i}\left(S\left(\frac{\pi}{4}\right)\right),\quad\widehat{W}\left(A\right)\coloneqq\left\langle \hat{r}_{i}\vert i\in I\right\rangle <Spin\left(A\right),
	\]
	\[
	\tilde{s}_{i}\coloneqq\psi_{ij}\circ\phi_{ij}^{i}\left(D\left(\frac{\pi}{2}\right)\right),\quad\widetilde{W}\left(A\right)\coloneqq\left\langle \tilde{s}_{i}\vert i\in I\right\rangle <K\left(A\right).
	\]
\end{defn}

Then $\widetilde{W}\left(A\right)\cong W^{ext}\left(A\right)$ by \cite[cor. 2.4]{Def relations of certain inf-dim groups} (and a few steps explained in \cite[rem. 18.5]{Spin covers} in more detail). 
Furthermore, the map $\hat{r}_{i}\mapsto r_{i}$ for $i=1,\dots,n$ defines an isomorphism from $\widehat{W}\left(A\right)$ to $W^{spin}\left(A\right)$ (cp. \cite[thm. 18.15]{Spin covers}).
Recall from \cite[thm. 11.17]{Spin covers}  (cited here as thm. \ref{thm: Spin(A) is central extension}) that $Spin(A)$ is a central extension of $K(A)$.
\begin{lemma}
	\label{lem: adjoint W^spin-action} The adjoint action of $Spin\left(A\right)$
	on $\mathfrak{k}$ factors through the natural projection $\varphi:Spin\left(A\right)\rightarrow K\left(A\right)$ and the induced action
	\begin{equation}
		Ad_{g}\left(x\right):=Ad_{\varphi(g)}\left(x\right)\ \forall\,g\in Spin\left(A\right),\,x\in\mathfrak{g}(A)\label{eq:Adjoint action of Spin(=00005CPi) on g}
	\end{equation}
	on $\mf{g}(A)$ satisfies for all $i\in I$:
	\begin{equation}
		Ad_{r_{i}}\left(\mathfrak{g}_{\alpha}\right)=\mathfrak{g}_{s_{i}.\alpha},\ \forall\,\alpha\in\Delta\label{eq:spin extended Wely group and root spaces}\:.
	\end{equation} 
	Furthermore, given $w\in W\left(A\right)$, there exists $\hat{w}\in W^{spin}\left(A\right)$ 
	such that	$Ad_{\hat{w}}\left(\mathfrak{g}_{\alpha}\right)=\mathfrak{g}_{w.\alpha}$ for all $\alpha\in\Delta$.
\end{lemma}

\begin{proof}
	By \cite[18.11]{Spin covers} one has that $Z=\ker\varphi\subset W^{spin}\left(A\right)$ and so the
	adjoint action of $W^{spin}\left(A\right)$ on $\mathfrak{g}(A)$ factors through the projection of $W^{spin}(A)$ to $W^{ext}\left(A\right)$.
	But $W^{ext}\left(A\right)$ acts precisely like (\ref{eq:spin extended Wely group and root spaces}) as has been noted in (\ref{eq:action of Wext on weight spaces}).
	By \cite[cor. 2.3 b)]{Def relations of certain inf-dim groups} there exists a unique map from $W\left(A\right)$ to $W^{ext}\left(A\right)$ satisfying
	\[
	e  \mapsto  e\,,\quad s_{i}  \mapsto  t_{i}\,,\quad \omega\omega'  \mapsto  \widetilde{\omega}\widetilde{\omega}'\ \text{if }l\left(\omega\omega'\right)=l\left(\omega\right)+l\left(\omega'\right)\:.
	\]
	Thus, to any reduced word $w\in W(A)$ one finds a word $\widetilde{w}$ in
	$W^{ext}(A)$ and one just uses the analogous word in $W^{spin}\left(A\right)$:
	\[
	w=s_{i_{1}}\cdots s_{i_{k}}\mapsto\widehat{w}=r_{i_{1}}\cdots r_{i_{k}}\in W^{spin}\left(A\right).
	\]
	This works because the action factors through $\varphi$ and hence through the projection from $W^{spin}\left(A\right)$ to $W^{ext}\left(A\right)$.
\end{proof}
	
\begin{lemma} \label{lem:f.d. rep compatible with Ad}
	Let $A$ be	simply-laced and indecomposable and let $\rho:\mathfrak{k}(A)\rightarrow End(V)$ be a finite-dimensional representation whose lift to $K\left(A\right)$ and/or $Spin\left(A\right)$ we denote by $\Omega$.
	The lift satisfies
	\[
	\rho\left(Ad_{g}\left(x\right)\right)=\Omega\left(g\right)\rho\left(x\right)\Omega\left(g\right)^{-1}\ \forall\,g\in Spin\left(A\right),\ \forall\,x\in\mathfrak{k}(A).
	\]
\end{lemma}

\begin{proof}
	Our proof is essentially based on the formula
	\begin{equation}
		\exp\left(\rho\left(a\right)\right)\rho\left(x\right)\exp\left(-\rho\left(a\right)\right)=\rho\left(\exp\left(\text{ad }a\right)(x)\right)\label{eq:exp, rho and ad}
	\end{equation}
	with $a,x\in\mathfrak{k}(A)$ such that $\rho(a)$ and $\text{ad}(a)$ are locally finite.
	The above formula is shown in \cite[(3.8.1)]{Kac:} for all $a,x\in\mathfrak{k}(A)$ such that $\rho(a)$, $\text{ad}(a)$	and $\rho(x)$ are locally nilpotent but later in \cite[sec. 3.8]{Kac:} it is also shown to be correct for $\rho(a)$ locally finite and such that the span of $\text{ad}(a)^{n}(x)$ for $n\in\mathbb{N}$ is
	finite-dimensional.
	Thus, (\ref{eq:exp, rho and ad}) is in particular applicable if $a$ is $\text{ad}$-locally finite and $(\rho,V)$ is a finite-dimensional $\mf{k}(A)$-module.
	Now the Berman-elements $x_{\alpha}\coloneqq e_{\alpha}-\omega\left(e_{\alpha}\right)$ for	$\alpha\in\Delta_{+}^{re}$ are $\text{ad}$-locally finite because the $e_{\alpha}$ are locally	nilpotent as long as $\alpha\in\Delta_{+}^{re}$ (because real root spaces are conjugate to simple root spaces by the extended Weyl group).
	Since subalgebras $\mf{k}_{J}\coloneqq\left\langle X_{j}\vert\,j\in J\subset I\right\rangle $ corresponding to $J\subset \{1,\dots,n\}$ such that $A_J$ is spherical consist only of real root spaces, this also means that all $x\in\mf{k}_{J}$ are $\text{ad}$-locally finite.
	Now $\mf{k}_{J}$ possesses an exponential map  $\exp_{J}:\mathfrak{k}_{J}\rightarrow K_{J}$. 
	Here, $K_J$ denotes both the maximal compact subgroup $K_{J}<G_{J}$ and its spin cover, since this detail will not matter.
	Although we won't need it in this strength, note that  $\exp_{J}$ is surjective because $K_J$ is compact.
	We denote the restrictions of  $\Omega$ to $K_{J}$ and $\rho$ to $\mf{k}_{J}$ by $\Omega_{J}$ and $\rho_{J}$ respectively.
	Then the diagram in fig. \ref{fig:Com_diagram 4_14} commutes, i.e., one has for $a\in\mathfrak{k}_{J}$ and $g=\exp_{J}(a)$ that 
	\[
	\Omega\left(g\right)=\Omega_{J}\left(\exp_{J}(a)\right)=\exp\left(\rho_{J}(a)\right)=\exp\left(\rho(a)\right)\ \forall\,g=\exp_{J}a\in K_{J}.
	\]
	\begin{figure}
		
			\includegraphics[scale=1.5]{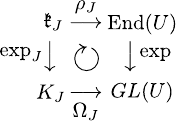}
			
		\caption{\label{fig:Com_diagram 4_14}A commutative diagram for $J$ spherical and $U$ a f.d. module.}
		
	\end{figure}
	Further, one has for all $a\in\mathfrak{k}_{J},x\in\mathfrak{k}(A)$ that
	\begin{eqnarray*}
		\Omega_{J}\left(\exp_{J}(a)\right)\rho(x)\Omega_{J}\left(\exp_{J}(a)\right)^{-1} & = & \exp\left(\rho(a)\right)\rho(x)\exp\left(-\rho(a)\right)\\
		& \overset{(\ref{eq:exp, rho and ad})}{=} & \rho\left(\exp\left(\text{ad }a\right)(x)\right)\\
		& = & \rho\left(Ad_{\exp_{J}a}\left(x\right)\right)=\rho\left(Ad_{g}(x)\right).
	\end{eqnarray*}
	The second-to-last equality uses that the involved Lie group and representation are finite-dimensional.
	Finally, one uses that $K(A)$ and $Spin(A)$ are generated by their fundamental rank-$1$ subgroups (or rank-$2$ if one prefers). 
\end{proof}
We are now in the position to determine how the one-parameter subgroups associated to simple roots act via conjugation on generalized $\Gamma$-matrices. As in prop. \ref{prop:Lift of 3/2 and 5/2 spin rep}, similar computations have been carried out in \cite{On higher spin} in a second-quantized form.
\begin{lemma}	\label{lem:conjugation lemma for S12}
	Let $A$ be simply-laced and let $\left(S,\rho\right)$ be a generalized spin representation of $\mathfrak{k}(A)$ as in def.	(\ref{Def:gen spin rep for simply laced case}) with associated generalized $\Gamma$-matrix as in prop.  \ref{prop:Existence of gen Gamma matrices}. With
	\begin{equation}
		\widehat{r}_{i}\left(\phi\right):=\exp\left(\phi\cdot\rho\left(X_{i}\right)\right),\label{eq:r_i hat in 1/2 spin rep}
	\end{equation}
	 one has for all $\alpha\in\Delta$ that
	\[
	\widehat{r}_{i}\left(\phi\right)\Gamma\left(\alpha\right)\widehat{r}_{i}\left(\phi\right)^{-1}=\begin{cases}
		\Gamma\left(\alpha\right) & \text{if }\left(\alpha\vert\alpha_{i}\right)\in2\mathbb{Z}\\
		\cos\phi\cdot\Gamma\left(\alpha\right)+\sin\phi\cdot\Gamma\left(\alpha_{i}\right)\Gamma\left(\alpha\right) & \text{if }\left(\alpha\vert\alpha_{i}\right)\in2\mathbb{Z}+1
	\end{cases}
	\]
	\[
	r_{i}\Gamma\left(\alpha\right)r_{i}^{-1}=\begin{cases}
		\Gamma\left(\alpha\right) & \text{if }\left(\alpha\vert\alpha_{i}\right)\in2\mathbb{Z}\\
		\varepsilon\left(\alpha_{i},\alpha\right)\Gamma\left(s_{i}.\alpha\right) & \text{if }\left(\alpha\vert\alpha_{i}\right)\in2\mathbb{Z}+1
	\end{cases}\:,
	\]
	where  $\varepsilon:Q\times Q\rightarrow\{\pm1\}$	denotes the standard normalized $2$-cocycle from lem. \ref{lem:def and existence of standard 2-cocycle}, and $r_{i}\coloneqq\widehat{r}_{i}\left(\frac{\pi}{2}\right)$ are
	the images of the generators of $W^{spin}(A)$.
\end{lemma}
\begin{proof}
	Recall that $\rho\left(X_{i}\right)=\frac{1}{2}\Gamma\left(\alpha_{i}\right)$ so that
	\[
	\widehat{r}_{i}\left(2\phi\right):=\exp\left(2\phi\cdot\rho\left(X_{i}\right)\right)=\cos\phi\cdot Id+\sin\phi\Gamma\left(\alpha_{i}\right).
	\]
	Thus,
	\begin{align*}
		\widehat{r}_{i}\left(2\phi\right) \Gamma\left(\alpha\right) \widehat{r}_{i}\left(2\phi\right)^{-1} &=  \exp\left(2\phi\rho\left(X_{i}\right)\right) \Gamma(\alpha) \exp\left(-2\phi\rho\left(X_{i}\right)\right)\\
		&= \cos^{2}\phi\cdot \Gamma(\alpha) -\sin^{2}\phi \Gamma\left(\alpha_{i}\right) \Gamma(\alpha) \Gamma\left(\alpha_{i}\right)\\
		&+ \sin\phi \cos\phi \left( \Gamma\left(\alpha_{i}\right) \Gamma\left(\alpha\right) -\Gamma\left(\alpha\right) \Gamma\left(\alpha_{i}\right) \right)\\
		&= \cos^{2}\phi \cdot\Gamma(\alpha) -\sin^{2}\phi \Gamma\left(\alpha_{i}\right)^{2} \Gamma\left(\alpha\right)\\
		&- \sin^{2}\phi \Gamma\left(\alpha_{i}\right) \left[ \Gamma\left(\alpha\right), \Gamma\left(\alpha_{i}\right) \right]\\
		&+ \sin\phi \cos\phi \left[ \Gamma\left(\alpha_{i}\right), \Gamma\left(\alpha\right) \right]\\
		&= \Gamma\left(\alpha\right) -\left(\sin^{2}\phi \Gamma\left(\alpha_{i}\right) + \sin\phi\cos\phi \right) \left[ \Gamma\left(\alpha\right), \Gamma\left(\alpha_{i}\right) \right]\:.
	\end{align*}
	By the use of (\ref{eq:Gamma matrices (anti-) commutator-1}) one has 
	\[
	\left[\Gamma\left(\alpha\right),\Gamma\left(\alpha_{i}\right)\right]=\begin{cases}
		0 & \text{if }\left(\alpha\vert\alpha_{i}\right)\in2\mathbb{Z}\\
		-2\Gamma\left(\alpha_{i}\right)\Gamma\left(\alpha\right) & \text{if }\left(\alpha\vert\alpha_{i}\right)\in2\mathbb{Z}+1
	\end{cases}
	\]
	which yields in combination with $2\Gamma\left(\alpha_{i}\right) \Gamma\left(\alpha_{i}\right) \Gamma\left(\alpha\right) =-2\Gamma\left(\alpha\right)$ and a few trigonometric identities that
	\[
	\widehat{r}_{i}\left(2\phi\right)\Gamma\left(\alpha\right)\widehat{r}_{i}\left(2\phi\right)^{-1}=\begin{cases}
		\Gamma\left(\alpha\right) & \text{if }\left(\alpha\vert\alpha_{i}\right)\in2\mathbb{Z}\\
		\cos2\phi\Gamma\left(\alpha\right)+\sin2\phi\Gamma\left(\alpha_{i}\right)\Gamma\left(\alpha\right) & \text{if }\left(\alpha\vert\alpha_{i}\right)\in2\mathbb{Z}+1
	\end{cases}
	\]
	It is possible to use (\ref{eq:Gamma matrices homomorphism property-1}),	(\ref{eq:Gamma matrices of certain sums of roots}) and $2Q\subset \ker \Gamma$ to write $\Gamma\left(\alpha_{i}\right)\Gamma\left(\alpha\right)=\varepsilon\left(\alpha_{i},\alpha\right)\Gamma\left(s_{i}.\alpha\right)$ if $\left(\alpha\vert\alpha_{i}\right)\in2\mathbb{Z}+1$.
	With $r_{i}\coloneqq\widehat{r}_{i}\left(\frac{\pi}{2}\right)$ one then obtains the claimed relations
	\[
	r_{i}\Gamma\left(\alpha\right)r_{i}^{-1}=\begin{cases}
		\Gamma\left(\alpha\right) & \text{if }\left(\alpha\vert\alpha_{i}\right)\in2\mathbb{Z}\\
		\varepsilon\left(\alpha_{i},\alpha\right)\Gamma\left(s_{i}.\alpha\right) & \text{if }\left(\alpha\vert\alpha_{i}\right)\in2\mathbb{Z}+1
	\end{cases}.
	\]

\end{proof}
We had shown in prop. \ref{prop:rep matrix is always prop to Gamma matrix} that $\rho\left(x_{\alpha}\right)=c\left(x_{\alpha}\right)\Gamma\left(\alpha\right)$ for all $x_{\alpha}\in\mf{k}_{\alpha}$.
We are now in the position to make a more precise statement on $c\left(x_{\alpha}\right)\in\mathbb{R}$.
	
\begin{proposition}	\label{prop:Conjugation lemma for S12} 
	Let $A$ be simply-laced and	indecomposable and let $\left(\rho, S\right)$ be a generalized	spin representation of $\mathfrak{k}(A)$ as in def.	\ref{Def:gen spin rep for simply laced case}. 
	If $ x\in\mathfrak{k}_{\alpha}$ is nonzero, then 
	\begin{align*}
	\rho\left(x\right)=c\cdot\Gamma(\alpha)\ s.t.\ c\neq0\ \text{if }\alpha\in\Delta^{re}(A),\quad\rho\left(x\right)=0\ \text{if }\alpha\text{ is an isotropic root.}
	\end{align*}
	Furthermore, to $\alpha,\beta\in\Delta^{re}$ s.t. $\alpha-\beta\in2Q(A)$ and	$0\neq x_{\alpha}\in\mathfrak{k}_{\alpha}$, $0\neq x_{\beta}\in\mathfrak{k}_{\beta}$,	there exists $c\in\mathbb{K}\setminus\left\{ 0\right\} $ s.t. $\rho\left(x_{\alpha}\right)=c\cdot\rho\left(x_{\beta}\right)$.
	If $x_{\alpha}$ and $x_{\beta}$ are both conjugate to $\pm X_{i}$ for some $i=1,\dots,n$,	then $c\in\left\{ -1,+1\right\} $.
\end{proposition}
\begin{remark}
	The close connection between generalized spin representations and generalized $\Gamma$-matrices has already been observed in \cite{On higher spin} as well as the fact that elements associated to isotropic root spaces are represented trivially.
	The approach presented here differs in so far as we treat the $\Gamma$-matrices as abstract objects and not elements of a Clifford algebra but more importantly that we derive the parametrization result by the action of the Weyl group instead of successive commutators which provides more control over the factor $c$, in particular if it is nonzero or not.
\end{remark}	
\begin{proof}
	Let $w\in W(A)$ and let $\widetilde{\omega}\in W^{ext}\left(A\right)$ and $\hat{\omega}\in W^{spin}(A)$
	be the corresponding elements from lem. \ref{lem: adjoint W^spin-action}.
	 Together with lem. \ref{lem:f.d. rep compatible with Ad} one has for $x_{\alpha}\in\mf{k}_{\alpha}$:
	 \begin{align*}
	 	\rho\left(Ad_{\widetilde{\omega}}\left(x_{\alpha}\right)\right) &=  \rho\left( Ad_{\widehat{\omega}}\left(x_{\alpha}\right) \right) =\Omega\left(\widehat{\omega}\right) \rho\left(x_{\alpha}\right) \Omega\left(\widehat{\omega}\right)^{-1}\:.
	 \end{align*}
	Now $\Delta^{re}=W(A)\cdot\left\{ \alpha_{1},\dots,\alpha_{n}\right\} $ and $A$  indecomposable
	and simply-laced imply that $\mathfrak{k}_{\alpha}$ for $\alpha\in\Delta^{re}$ is conjugate to $\mathfrak{k}_{\alpha_{i}}=\mathbb{R} X_{i}$ for any $i=1,\dots,n$.
	As $\rho\left(X_{i}\right)\neq0$ for all $i=1,\dots,n$ this shows $\rho\left(x\right)\neq0$ for all $0\neq x\in\mathfrak{k}_{\alpha}$	for all $\alpha\in\Delta^{re}$.
	
	On $\mf{k}(A)$, the invariant bilinear form is negative-definite (a consequence of \cite[thm. 11.7]{Kac:}) and invariant under the action of $\mf{k}(A)$ and hence invariant under the action of $K(A)$.
	Since $W^{ext}(A)<K(A)$ this shows that the norm (defined as the negative of the invariant bilinear form) of $Ad_{\widetilde{\omega}}\left(x_{\alpha}\right))$ is equal to that of $x_{\alpha}$.
	Therefore, $c=\pm 1$ if the norm of $x_{\alpha}$ is that of an $X_i$.
	
	Now by (\ref{eq:Gamma matrices of certain sums of roots}) one has for $\alpha,\beta\in\Delta^{re}$ such that $2\gamma\coloneqq\alpha-\beta\in2Q$ that
	\[
	\Gamma\left(\alpha\right)= \Gamma\left(\beta+2\gamma\right) =\left(-1\right)^{\left(\gamma\vert\gamma\right)}\Gamma\left(\beta\right) = \Gamma\left(\beta\right)
	\]
	and as both $\rho\left(x_{\alpha}\right)$ and  $\rho\left(x_{\beta}\right)$ are nonzero, they must be proportional.
	
	Any isotropic root (roots $\alpha\in\Delta$ that satisfy $(\alpha\vert\alpha)=0$) is $W(A)$-conjugate to the multiple of an affine null root, i.e., a root whose support is an affine subdiagram of $A$ (cp. \cite[prop. 5.7]{Kac:}).
	Given an affine null root $\delta$ one has that $\mathfrak{k}_{m\delta}$	is spanned by all $\left[x_{\alpha_{i}},x_{m\delta-\alpha_{i}}\right]$ with $i\in\text{supp}\,\delta$.
	Since $(\alpha_i\vert \delta)=0$ for all $i\in\text{supp}\,\delta$ one has  $\left(\alpha_{i}\vert m\delta-\alpha_{i}\right)=-2$ so that $\left[\Gamma\left(\alpha_{i}\right), \Gamma\left( m\delta-\alpha_{i}\right)\right]=0$ shows $\mathfrak{k}_{m\delta}\subset\ker\rho$.
	One concludes with lem. \ref{lem: adjoint W^spin-action} that $\rho\left(Ad_{\widetilde{\omega}}\left(x_{m\delta}\right)\right) = \Omega\left(\widehat{\omega}\right) \rho\left(x_{m\delta}\right) \Omega\left(\widehat{\omega}\right)^{-1}=0$.
\end{proof}

\begin{proposition}	\label{prop:Conjjugation lemma for Sn2}  
	Let $A$ be simply-laced and indecomposable and let $\sigma: \mathfrak{k}\left(A\right) \rightarrow \text{End}\left(V\right)\otimes\text{End}\left(S\right)$	denote a higher spin	representation from thms. \ref{thm:S32 and S52} or \ref{thm:7/2 spin rep}.
	Let $\alpha\in\Delta^{re}_+$ and $0\neq x_{\alpha}\in\mathfrak{k}_{\alpha}$,	then there exists $c\left(x_{\alpha}\right)\neq0$ s.t. 
	\[
	\sigma\left(x_{\alpha}\right)=c\left(x_{\alpha}\right)\cdot\tau\left(\alpha\right)\otimes\Gamma\left(\alpha\right).
	\]
	If $x_{\alpha}$ has the same norm than $ X_{j}$ for $j=1,\dots,n$,	then $c\left(x_{\alpha}\right)\in\left\{ -1,+1\right\} $.
\end{proposition}
\begin{remark}
	It was observed in  \cite{On higher spin} that the formula for the Berman generators' representation matrices can easily be used for any real root as well. 
	It was shown further that if $\alpha=\beta+\gamma$ all real such that $(\beta\vert\gamma)=-1$ and the formula holds for $x\in\mf{k}_\beta$ and $y\in\mf{k}_\gamma$, then it holds for $[x,y]\in \mf{k}_{\beta+\gamma}\oplus\mf{k}_{\beta-\gamma}$.
	The formula therefore extends to all real roots that can be written as a successive sum of real roots with product equal to $-1$.
	However, it remains unclear if any real root can be decomposed this way in a simply-laced root system.
	In a discussion, the first author of \cite{On higher spin} expressed the idea of bypassing this problem by Weyl-group conjugation, as any real root is $W(A)$-conjugate to a simple root.
	This proposition is the realization of this idea with a solid link to the mathematical literature and involved objects.
	In particular, we work in the required setting of spin covers and spin-extended Weyl groups developed in \cite{Spin covers}. 
\end{remark}
\begin{proof}
	If $\sigma$ is a representation as in thm. \ref{thm:S32 and S52}, denote the lift to $Spin(A)$ by $\Sigma$ and if it is a representation as in thm. \ref{thm:7/2 spin rep}, denote the lift by $\widetilde{\Sigma}$ so that the formulas of props. \ref{prop:Lift of 3/2 and 5/2 spin rep} and \ref{prop:Lift of 7/2 spin rep} apply verbatim.
	Also, denote the lifts to the fundamental rank-$1$ subgroups by $\Sigma_{i}$ and $\widetilde{\Sigma}_i$, respectively.
	By props.  \ref{prop:Lift of 3/2 and 5/2 spin rep} and \ref{prop:Lift of 7/2 spin rep} one has (note that the formula for general argument $\phi$ differs)
	\begin{align*}
		\Sigma_{i}\left(\pm\frac{\pi}{2}\right) &= \frac{1}{\sqrt{2}} \cdot \eta\left(s_{i}\right) \otimes \left(Id\pm\Gamma\left(\alpha_{i}\right)\right)\:,\\
		\widetilde{\Sigma}_{i} \left(\pm\frac{\pi}{2}\right) &= \frac{1}{\sqrt{2}}\cdot\eta\left(s_{i}\right)\otimes\left(Id\pm\Gamma\left(\alpha_{i}\right)\right)\:.
	\end{align*}
	The $\Sigma_{i}\left(\frac{\pi}{2}\right)$, resp. $\widetilde{\Sigma}_{i}\left(\frac{\pi}{2}\right)$, generate the image of  $W^{spin}(A)$ under $\Sigma$, resp. $\widetilde{\Sigma}$.
	Similar to the previous proof we obtain the representation matrices by conjugation:
	\[
	\sigma\left(Ad_{\widetilde{\omega}}\left(x_{\alpha}\right)\right)=\sigma\left(Ad_{\widehat{\omega}}\left(x_{\alpha}\right)\right)=\Sigma\left(\widehat{\omega}\right)\sigma\left(x_{\alpha}\right)\Sigma\left(\widehat{\omega}\right)^{-1},
	\]
	where $\widetilde{\omega}\in W^{ext}(A)$ is the projection of $\widehat{\omega}\in W^{spin}(A)$.
	For both types of representations one computes 
	\begin{align*}
		\Sigma_{i}\left(\frac{\pi}{2}\right) \sigma\left(X_{j}\right) \Sigma_{i}\left(-\frac{\pi}{2}\right) &=  \frac{1}{2} \eta\left(s_{i}\right) \otimes \left(Id+\Gamma\left(\alpha_{i}\right)\right) \cdot \tau\left(\alpha_{j}\right) \otimes \Gamma\left(\alpha_{j}\right) \cdot \eta\left(s_{i}\right) \otimes \left(Id-\Gamma\left(\alpha_{i}\right)\right)\\
		&= \frac{1}{2} \left( \eta\left(s_{i}\right) \tau\left(\alpha_{j}\right) \eta\left(s_{i}\right) \right) \otimes \left(Id+\Gamma\left(\alpha_{i}\right)\right) \Gamma\left(\alpha_{j}\right) \left(Id-\Gamma\left(\alpha_{i}\right)\right)
	\end{align*}
	and
	\begin{align*}
		\left(Id+\Gamma\left(\alpha_{i}\right)\right) \Gamma\left(\alpha_{j}\right) \left(Id-\Gamma\left(\alpha_{i}\right)\right) &=  \Gamma\left(\alpha_{j}\right) +\Gamma\left(\alpha_{i}\right) \Gamma\left(\alpha_{j}\right) -\Gamma\left(\alpha_{j}\right) \Gamma\left(\alpha_{i}\right) -\Gamma\left(\alpha_{i}\right) \Gamma\left(\alpha_{j}\right) \Gamma\left(\alpha_{i}\right)\\
		&=  \begin{cases}
			2\Gamma\left(\alpha_{j}\right) & \text{if }\left(\alpha_{i}\vert\alpha_{j}\right)=0\\
			2\underset{=\varepsilon\left(\alpha_{i},\alpha_{j}\right)\Gamma\left(\alpha_{i}+\alpha_{j}\right)}{\underbrace{\Gamma\left(\alpha_{i}\right)\Gamma\left(\alpha_{j}\right)}} & \text{if }\left(\alpha_{i}\vert\alpha_{j}\right)=-1.
		\end{cases}
	\end{align*}
	For $\mathcal{S}_{\frac{3}{2}}$ and $\mathcal{S}_{\frac{5}{2}}$ one has $\tau\left(\alpha_{j}\right)=\eta\left(s_{j}\right)-\frac{1}{2}Id$ and computes further:
	\begin{align*}
		\eta\left(s_{i}\right) \tau\left(\alpha_{j}\right) \eta\left(s_{i}\right) &=  \eta\left(s_{i}\right) \left(\eta\left(s_{j}\right) -\frac{1}{2} Id\right) \eta\left(s_{i}\right) =\eta\left(s_{i}s_{j}s_{i}\right)-\frac{1}{2}Id\\
		&=  \eta\left(s_{s_{i}.\alpha_{j}}\right)-\frac{1}{2}Id=\tau\left(s_{i}.\alpha_{j}\right).
	\end{align*}
	For $\mathcal{S}_{\frac{7}{2}}$ one has $\tau\left(\alpha_{j}\right)=\eta\left(s_{j}\right)-\frac{1}{2}Id+f\left(\alpha_{j}\right)$
	and from lem. \ref{lem:scalar products of natural elements in Sym^3V}
	one derives that 
	$\eta\left(s_{i}\right)f\left(\alpha_{j}\right)\eta\left(s_{i}\right) = f\left(s_{i}.\alpha_{j}\right)$
	and therefore 
	\begin{align*}
		\eta\left(s_{i}\right)\tau\left(\alpha_{j}\right) \eta\left(s_{i}\right) &=  \eta\left(s_{i}\right) \left[\eta\left(s_{j}\right)-\frac{1}{2}Id+f\left(\alpha_{j}\right) \right] \eta\left(s_{i}\right)\\
		&=  \eta\left(s_{s_{i}.\alpha_{j}}\right) -\frac{1}{2} Id +f\left(s_{i}.\alpha_{j}\right) =\tau\left(s_{i}.\alpha_{j}\right).
	\end{align*}
	This yields for both cases (from now one we denote the lift simply by $\Sigma$) that 
	\[
	\Sigma_{i}\left(\frac{\pi}{2}\right)\sigma\left(X_{j}\right)\Sigma_{i}\left(-\frac{\pi}{2}\right)=\begin{cases}
		\tau\left(\alpha_{j}\right)\otimes\Gamma\left(\alpha_{j}\right) & \text{if }\left(\alpha_{i}\vert\alpha_{j}\right)=0\:,\\
		\varepsilon\left(\alpha_{i},\alpha_{j}\right)\tau\left(\alpha_i+\alpha_j\right)\otimes\Gamma\left(\alpha_{i}+\alpha_{j}\right) & \text{if }\left(\alpha_{i}\vert\alpha_{j}\right)=-1\:.
	\end{cases}
	\]
	Spelled out differently this yields 
	\[
	\Sigma_{i}\left(\frac{\pi}{2}\right)\sigma\left(X_{j}\right)\Sigma_{i}\left(-\frac{\pi}{2}\right) =\varepsilon\left(\alpha_{i},\alpha_{j}\right) \cdot\tau\left(s_{i}.\alpha_{j}\right) \otimes\Gamma\left(s_{i}.\alpha_{j}\right),
	\]
	since $\varepsilon\left(\alpha_{i},\alpha_{j}\right)=1$ if $\left(\alpha_{i}\vert\alpha_{j}\right)=0$.
	In contrast to the formula before, this formula is also correct if $\alpha_j$ is replaced by an arbitrary real root $\beta\in\Delta^{re}_+$: While the $\Gamma$-matrix only counts modulo $2Q$, one has $\tau(s_i.\beta)\neq \tau(\beta)$ if $(\alpha_i\vert\beta)\in 2\mathbb{Z}$. 
	Therefore it is important to include all actions of $W(A)$.
	By repeated action of $\Sigma_{i}\left(\frac{\pi}{2}\right)$ for $i=1,\dots,n$ this yields for a reduced expression $\widehat{\omega} = r_{i_n}\cdots r_{i_1}$:
	\begin{subequations}
		\begin{align}
		\Sigma\left(\widehat{\omega}\right)\sigma\left(X_{j}\right)\Sigma\left(\widehat{\omega}\right)^{-1} &= 
		c\cdot \tau\left(\omega.\alpha_{j}\right) \otimes \Gamma\left(\omega.\alpha_{j}\right), \label{eq:conjugation in S3/2 and S5/2} \\
		c &= \prod_{k=1}^n \varepsilon(\alpha_{i_k}, s_{i_{k-1}}\cdots s_{i_1}.\alpha_j) \label{eq: sign by conjugation}
		\end{align}	
	\end{subequations}	
	where  $\omega\in W$ is the projection of $\widehat{\omega}$.
	
	As before $\mf{k}_\alpha$ for $\alpha\in\Delta^{re}_+$ is conjugate to $\mf{k}_{\alpha_i}=\mathbb{R}X_i$ via $W^{ext}(A)$ and hence also by $W^{spin}(A)$.
	Thus, all nonzero $ x_{\alpha}\in\mathfrak{k}_{\alpha}$ have a nontrivial image and if the norm of $x_{\alpha}$ is equal to that of $X_i$ the constant in (\ref{eq:conjugation in S3/2 and S5/2}) is $\pm 1$.
\end{proof}
\begin{remark}
	Note that (\ref{eq: sign by conjugation}) suggests that the sign were fully under control. 
	This is only the case, if one fixes the basis of $\mf{k}_\alpha$ to be $\mathrm{Ad}_{\tilde{\omega}}(X_j)$ for the appropriate $\tilde{\omega}\in W^{ext}(A)$. 
	In general, it is a hard problem to decide upon the sign if the basis of a real root space is already fixed (cp. \cite[example 4.25]{Introduction to KM-groups over fields}).
	Thus, the difficulty of the sign is actually not due to the spin representation but due to the subtleties of the adjoint action of $W^{ext}(A)$. 
\end{remark}
\subsection*{Acknowledgment}
We gratefully acknowledge the support of the DFG under the grant KO 4323/13 and the Studienstiftung des Deutschen Volkes for partially funding the research presented in  this paper. 
We would like to thank Axel Kleinschmidt and Hermann Nicolai for discussions and Paul Zellhofer for his comments on an earlier version of this manuscript.
Last, we would like to thank an anonymous reviewer for helpful comments and suggestions that helped improve the quality of our manuscript.

\end{document}